\documentclass[9pt]{article}
\usepackage{listings}
\usepackage{pdflscape}
\usepackage{amsfonts}
\usepackage{epsfig}
\usepackage{lscape}
\usepackage{longtable}
\usepackage{amsmath,amssymb,amsthm}
\usepackage{graphicx}
\usepackage{subfigure}
\usepackage{color}
\usepackage{placeins}
\usepackage{url}
\usepackage{cases}
\usepackage{longtable}
\usepackage{supertabular}
\usepackage{multirow}
\newcommand{\tabincell}[2]{\begin{tabular}{@{}#1@{}}#2\end{tabular}}
\usepackage{amsmath}
\allowdisplaybreaks[4]
\usepackage{cite}
\usepackage{algorithmic}
\usepackage{algorithm,float}
\usepackage{booktabs}
\usepackage[shortlabels]{enumitem}
\usepackage[online]{threeparttablex}
\usepackage{siunitx}
\usepackage{tikz}



\newtheorem{theorem}{Theorem}

\newtheorem{remark}{Remark}

\newtheorem{proposition}{Proposition}

\begin{document}
\title{\bf \Large{An augmented Lagrangian method with constraint generation for shape-constrained convex regression problems}\footnotemark[1]}
\author{Meixia Lin\footnotemark[2], \quad Defeng Sun\footnotemark[3], \quad Kim-Chuan Toh\footnotemark[4]}

\date{November 08, 2021}
\maketitle

\renewcommand{\thefootnote}{\fnsymbol{footnote}}
\footnotetext[1]{{\bf Funding}: Defeng Sun is supported in part by Hong Kong Research Grant Council under grant number 15304019 and Kim-Chuan Toh by the Ministry of Education, Singapore, under its Academic Research Fund Tier 3 grant call (MOE-2019-T3-1-010).}
\footnotetext[2]{Corresponding author. Institute of Operations Research and Analytics, National University of Singapore, Singapore 119076, Singapore ({\tt lin\_meixia@u.nus.edu}).}
\footnotetext[3]{Department of Applied Mathematics, The Hong Kong Polytechnic University, Hung Hom, Hong Kong ({\tt defeng.sun@polyu.edu.hk}).}
\footnotetext[4]{Department of Mathematics and Institute of Operations Research and Analytics, National University of Singapore, Singapore 119076, Singapore ({\tt mattohkc@nus.edu.sg}).}
\renewcommand{\thefootnote}{\arabic{footnote}}

\begin{abstract}
Shape-constrained convex regression problem deals with fitting a convex function to the observed data, where additional constraints are imposed, such as component-wise monotonicity and uniform Lipschitz continuity. This paper provides a unified framework for computing the least squares estimator of a multivariate shape-constrained convex regression function in $\mathbb{R}^d$. We prove that the least squares estimator is computable via solving an essentially constrained convex quadratic programming (QP) problem with $(d+1)n$ variables, $n(n-1)$ linear inequality constraints and $n$ possibly non-polyhedral inequality constraints, where $n$ is the number of data points. To efficiently solve the generally very large-scale convex QP, we design a proximal augmented Lagrangian method ({\tt proxALM}) whose subproblems are solved by the semismooth Newton method ({\tt SSN}). To further accelerate the computation when $n$ is huge, we design a practical implementation of the constraint generation method such that each reduced problem is efficiently solved by our proposed {\tt proxALM}. Comprehensive numerical experiments, including those in the pricing of basket options and estimation of production functions in economics, demonstrate that our proposed {\tt proxALM} outperforms the state-of-the-art algorithms, and the proposed acceleration technique further shortens the computation time by a large margin.
\end{abstract}

\\[5pt]

\section{Introduction}
\label{sec:introduction}
Convex (or concave) regression is meant to estimate a convex (or concave) function based on a finite number of observations. It is a topic of interest in many fields such as economics, operations research and financial engineering. In economics, production functions \cite{hildreth1954point,varian1984nonparametric,allon2007nonparametric}, demand functions \cite{varian1982nonparametric} and utility functions \cite{meyer1968consistent} are often required to be concave. In operations research, the performance measure expectations can be proved to be convex in the underlying model parameters, e.g. in the context of queueing network \cite{chen2013fundamentals}. In financial engineering, the option pricing function has the convexity restriction under the no-arbitrage condition, as can be seen from \cite{ait2003nonparametric}. In the literature, there are various methods for solving the convex regression problem. With the specification of a functional form, one can apply a parametric approach to estimate the convex function. For example, the Cobb-Douglas production function is a particular functional form of the production function that is widely used in applied production economics. To avoid strong prior assumptions on the functional form, one can also use a non-parametric approach to perform the function estimation. Generally, the nonparametric estimation is based on a given collection of primitive functions, such as local polynomial \cite{longstaff2001valuing}, trigonometric series, spline estimator \cite{dontchev2003quadratic,qi2007regularity} and kernel-type estimator \cite{aybat2014parallel}. However, such an approach may face some difficulties in imposing the convexity constraint and choosing appropriate smoothing parameters (e.g. the degree of the polynomial, or the kernel density bandwidth). To overcome these difficulties, one may choose to estimate the functions by empirical risk minimization \cite{cui2018composite} over the set of convex functions, wherein the squared error loss \cite{hildreth1954point} and the absolute error loss \cite{blanchet2019multivariate} are studied. In this paper, we focus on the least squares estimator for convex regression, whose theoretical properties are carefully studied in \cite{hanson1976consistency,seijo2011nonparametric,lim2012consistency}.

Suppose that we observe $n$ data points $\{(X_i,Y_i)\}_{i=1}^n$, which satisfy the regression model $Y = \psi(X)+\varepsilon$ for an unknown convex function $\psi:\Omega \rightarrow \mathbb{R}$, where $\Omega\subset \mathbb{R}^d$ is a $\delta$-neighborhood of ${\rm conv}(X_1,\cdots,X_n)$ (the convex hull of $\{X_i\}_{i=1}^n$), $\varepsilon$ is a random variable with expectation $\mathbb{E}[\varepsilon\vert X]=0$. The least squares estimator $\hat{\psi}$ of $\psi$ is defined as
\begin{align*}
\hat{\psi}\in \underset{\psi\in\mathcal{C}}{\arg\min}\ \sum_{i=1}^n(\psi(X_i)-Y_i)^2,\quad \mathcal{C}:=\{ \psi:\Omega \rightarrow \mathbb{R}\mid \psi\mbox{ is a convex function} \}.
\end{align*}
This infinite dimensional model appears to be intractable. Fortunately, the authors in \cite{kuosmanen2008representation,seijo2011nonparametric} have provided a computationally tractable optimal solution to it. They showed that the family of convex functions can be characterized by a subset of continuous, piecewise linear functions $\theta_i + \langle \xi_i,X-X_i\rangle$, $i=1,\cdots,n$, whose intercepts $\theta_i$'s and gradient vectors $\xi_i$'s are restricted to satisfy the convexity conditions. That is, a convex quadratic programming (QP) problem
\begin{align}
\min_{\begin{subarray}{c}\theta_1,\ldots,\theta_n\in\mathbb{R},\\ \xi_1,\ldots,\xi_n \in\mathbb{R}^{d}\end{subarray}} \
\left\{ \frac{1}{2} \sum_{i=1}^n (\theta_i - Y_i)^2 \Bigm\lvert  \theta_i \geq \theta_j + \langle \xi_j,X_i - X_j\rangle,\  1\leq i ,j \leq n \right\}\label{eq: convex_QP}
\end{align}
needs to be solved. The problem \eqref{eq: convex_QP} with $(d+1)n$ variables and $n(n-1)$ linear inequality constraints can be solved by interior point solvers such as those implemented in {\tt MOSEK} when $n$ is not too large, as stated in \cite{seijo2011nonparametric}. However, interior point solvers may quickly run out of memory when $n$ is large due to the presence of a large number of $n(n-1)$ linear inequality constraints. Mazumder et al. \cite{mazumder2019computational} adapted a three-block alternating direction method of multipliers ({\tt ADMM}) to solve \eqref{eq: convex_QP} but the method has no convergence guarantee. It needs about $1000$ seconds to solve an instance with $d=4$, $n\sim 3000$ to get a rough approximate solution. As the objective function in \eqref{eq: convex_QP} is not strongly convex, some papers including \cite{aybat2014parallel,chen2020multivariate} do not deal with \eqref{reformulated_problem} exactly but perturb the problem by adding an additional $\ell_2$ regularization term on the $\xi_i$'s. The regularization term allows one to apply the accelerated proximal gradient ({\tt APG}) method to the dual of the perturbed QP. For example, Aybat et al. \cite{aybat2014parallel} proposed a parallel {\tt APG} method. However, it is still not fast enough for solving large problems as it needs $17$ minutes to solve a problem with $d=80$, $n=1600$ on a $16$-core machine sharing $32$ GB. It should be noted that the regularization parameter may need to be extremely small in order for a solution of the perturbed QP to be optimal to the original QP under some kind of exact penalty property, while the dual of the perturbed QP also becomes harder to solve as the parameter becomes smaller. The computational challenge in solving the problem \eqref{eq: convex_QP} still remains in need of more progress, especially for the case when $d$ and $n$ are relatively large where existing methods are too expensive even for computing a solution with a moderate accuracy.

In many real applications, one may need to impose more shape constraints on the convex function $\psi$, such as component-wise monotonicity and uniform Lipschitz continuity. For example, the option pricing function under the no-arbitrage condition needs to be non-decreasing as well as convex as described in \cite{ait2003nonparametric}. In addition, when dealing with the Lipschitz convex regression as in \cite{lim2014convergence,balazs2015near,mazumder2019computational}, the uniform Lipschitz property of the convex function is added when performing the estimation. For the shape-constrained convex regression problem, the least squares estimator $\hat{\psi}$ is defined as
\begin{equation}\label{eq:min_CS}
\begin{aligned}
&\hat{\psi}\in \underset{\psi\in\mathcal{C}_\mathcal{S} }{\arg\min} \ \sum_{i=1}^n(\psi(X_i)-Y_i)^2,\\ &\mathcal{C}_\mathcal{S}=\{\psi:\Omega \rightarrow \mathbb{R}\mid \psi\mbox{ is a convex function with Property }\mathcal{S} \},
\end{aligned}
\end{equation}
where Property $\mathcal{S}$ specifies the shape constraint of $\psi$. We restrict ourselves to the case when Property $\mathcal{S}$ takes one of the following forms:
\begin{enumerate}[label=(S\arabic*),leftmargin=1.05cm]
	\item (monotone constraint) $\psi$ is non-decreasing in some of the coordinates (denoted as $K_1$) and non-increasing in some others (denoted as $K_2$), where $K_1$ and $K_2$ are disjoint subsets of $\{1,\cdots,d\}$;
	\item (box constraint) the elements in $\partial \psi(x)$ for any $x\in \Omega$ are bounded by two given vectors $L,U\in\mathbb{R}^d$;
	\item (Lipschitz constraint) $\psi$ is Lipschitz, i.e., $|\psi(x)-\psi(y)|\leq L \|x-y\|_p$ for any $x,y\in \Omega$, where $p=1,2,\infty$, and $L$ is a given positive constant.
\end{enumerate}

In this paper, we provide a unified framework for computing the least squares estimator for the shape-constrained convex regression problem \eqref{eq:min_CS}. We prove that the minimal sum of squared error can be achieved via a set of piecewise linear functions whose intercepts and gradient vectors are constrained to satisfy the convexity conditions and required shape constraints (see Theorem \ref{thm:relation_two_fun}). This conclusion leads us to an essentially constrained QP with $(d+1)n$ variables, $n(n-1)$ linear inequality constraints and $n$ possibly non-polyhedral constraints\footnote{Strictly speaking, it is no longer a conventional QP problem in the presence of the non-polyhedral constraints. Slightly abusing the notation, here we use QP for convenience.}. The addition of the shape constraints obviously would make the QP even more complicated and difficult to solve. Note that the estimator obtained in this way is nonsmooth, one can apply the Moreau proximal smoothing technique to obtain a smooth approximation. In addition, we can use a generalized form of the proposed constrained QP model as well as a data-driven Lipschitz estimation method to handle the boundary effect of the least squares estimator.

The main task in this framework is to solve the constrained QP in a robust and efficient manner. Most existing methods for the QP in the standard convex regression problem are either not extendable or difficult to be modified to efficiently solve the constrained QP due to the additional shape constraints. Moreover, except for interior point solvers which are only suitable for moderate size problems, almost all the other existing methods are first-order methods which may suffer from slow convergence rate when solving large-scale problems. For the multivariate shape-constrained convex regression problem, even with only a moderate number of observations, say $n=10^3$, the memory cost and computational cost are already massive since the underlying QP has about a million constraints. To tackle the potentially very large-scale constrained QPs, we design an asymptotically superlinearly convergent proximal augmented Lagrangian method ({\tt proxALM}), whose subproblems are solved by the semismooth Newton method ({\tt SSN}), a second order method that has quadratic convergence. In the algorithm, the second order sparsity structure of the problem is fully uncovered and exploited to highly reduce the computational cost of solving the Newton systems. Comprehensive numerical experiments demonstrate that our proposed {\tt proxALM} outperforms the state-of-the-art algorithms such as {\tt MOSEK} and {\tt ADMM} by a large margin.

Note that when the number of observations is very large, memory issues may appear. For the case when $n$ is huge, say $n=10^5$, the constrained QP contains $10^{10}$ linear inequality constraints. As an illustration, a vector with dimension $10^{10}$ requires $74.5$GB of RAM to store in dense double precision, which implies that it is almost impossible to solve the constrained QP with $n=10^5$ on an ordinary desktop PC. This motivates us to explore the problem structure to overcome the computational and memory challenges of solving high-sample problems. As constraint generation techniques (also known as cutting plane methods) have been popular in solving linear programs with a large number of constraints \cite{bertsimas1997introduction}, some researchers have applied this idea to solve convex regression problems. Hannah and Dunson \cite{hannah2013multivariate} considered a globally convex regression model from locally linear estimates fitted on adaptively selected observations, and Bal$\acute{\rm a}$zs et al. \cite{balazs2015near} proposed an aggregate cutting plane method for solving the convex regression problem, but their computation was limited to moderate problem sizes or low accuracy. Bertsimas and Mundru \cite{bertsimas2021sparse} used a cutting plane method with each reduced problem solved by the commercial solver {\tt Gurobi}. They reported solving an instance with $(d,n)=(10^2,10^5)$ to moderate accuracy in about $7$ hours. Recently, Chen and Mazumder \cite{chen2020multivariate} adapted the constraint generation method to solve the perturbed QP for the case when $n=10^4,10^5$, $d\leq 10$, where they applied the {\tt APG} method to solve the dual of each reduced problem. However, the solutions they obtained are not guaranteed to satisfy the optimality conditions.

The main challenges of applying the constraint generation method to solve convex regression problems are summarized in two aspects. First, each reduced problem of the original QP without the perturbation term needs to be solved to sufficiently high accuracy in order to determine the violated constraint unambiguously. Second, given an approximate optimal solution, it is computationally expensive to search all $O(n^2)$ constraints to find the violated ones and check the optimality conditions. Note that existing interior point solvers or first-order algorithms (such as {\tt APG} and {\tt ADMM}) could not solve large-scale problems to high accuracy efficiently. Thus a constraint generation method employing those solvers needs to be conservative in allowing a small number of violated constraints to be added in each round. As a result, it may take many rounds of the constraint generation to find a solution with the required accuracy for the original QP. This implies that the computational cost of searching for violated constraints and checking optimality conditions can be very large, which is unaffordable in practice. Fortunately, our proposed {\tt proxALM} allows us to solve large-scale problems to high accuracy efficiently, which motivates us to design a practical implementation of the constraint generation method to solve the shape-constrained convex regression problem. In our implementation, we add a relatively large number of most violated constraints in each round to greatly reduce the number of rounds of the constraint generation. For each reduced problem, we apply our {\tt proxALM}, which is demonstrated to be much more efficient in solving large-scale problems than other state-of-the-art algorithms.

We summarize our main contributions in this paper as follows.
\begin{itemize}[noitemsep,topsep=0pt]
	\item[1] We provide a unified framework for computing the least squares estimator in the shape-constrained convex regression problem \eqref{eq:min_CS}, wherein a constrained QP with $(d+1)n$ variables, $n(n-1)$ linear inequality constraints and $n$ possibly non-polyhedral inequality constraints needs to be solved.
	\item[2] To solve the constrained QP, we propose an asymptotically superlinearly convergent proximal augmented Lagrangian method, where each subproblem of the {\tt proxALM} is solved by the semismooth Newton method. We analyse the second order sparsity structure of the subproblems and develop novel numerical techniques to solve the semismooth Newton linear systems efficiently through exploiting the uncovered structure. Comprehensive numerical experiments, including those in the pricing of basket options and estimation of production functions, demonstrate that the proposed {\tt proxALM} outperforms other state-of-the-art algorithms such as {\tt MOSEK} and {\tt ADMM} by a large margin, especially for large-scale problems.
	\item[3] To solve the shape-constrained convex regression problem with a huge sample size, we design a practical implementation of the constraint generation method where each of its reduced problem is solved by our proposed {\tt proxALM}. Numerical experiments are also performed to demonstrate the high efficiency of the constraint generation method with {\tt proxALM}.
\end{itemize}

In the remaining part of the paper, we provide a unified framework for estimating the multivariate shape-constrained convex function in Section \ref{sec:mechanism}. For solving the involved constrained QP, the proximal augmented Lagrangian method is described in Section \ref{sec:pALM}. The implementation details of the proposed {\tt proxALM} can be found in Section \ref{sec:proximalmapping}. In Section \ref{sec:adaptive sieving}, we design a practical implementation of the constraint generation method to solve shape-constrained convex regression problems with huge samples sizes. Section \ref{sec:numerical} provides the numerical comparison of {\tt proxALM} with other start-of-the-art algorithms. Experiments are also conducted to demonstrate the superior performance of the constraint generation method combined with the {\tt proxALM} for solving instances with huge samples sizes. Then we apply our framework to perform the function estimation in several interesting real applications in Section \ref{sec:realapplication}. Finally, we conclude the paper.

\paragraph{Notation.} Denote $X=(X_1,\cdots,X_n)\in \mathbb{R}^{d\times n}$, $e_n\in \mathbb{R}^n$ be the vector of all ones, and $I_n$ be the $n\times n$ identity matrix. For any matrix $Z\in \mathbb{R}^{m\times n}$, $Z_i$ denotes the $i$-th column of $Z$. We use ``${\rm Diag}(z)$" to denote the diagonal matrix whose diagonal is given by the vector $z$, and use ${\rm Diag}(Z_1,\cdots,Z_n)$ to denote the block diagonal matrix whose $i$-th block is the matrix $Z_i$. For any symmetric and positive semidefinite matrix $H\in \mathbb{R}^{n\times n}$, we define $\langle x,x'\rangle_{H}:=\langle x,H x'\rangle$, and $\|x\|_{H}:=\sqrt{\langle x,x\rangle_{H}}$ for all $x,x'\in \mathbb{R}^n$. For a given closed subset $C$ of $\mathbb{R}^n$ and $x\in\mathbb{R}^n$, we define ${\rm dist}_{H}(x,C)=\min\{\|x-y\|_H\mid y\in C\}$. The largest (smallest) eigenvalue of $H$ is denoted as $\lambda_{\max}(H)$ ($\lambda_{\min}(H)$). Given $x\in \mathbb{R}^n$ and an index set $K\subset \{1,\cdots,n\}$, $x_K$ denotes the sub-vector of $x$ with those elements not in $K$ being removed. Let $q:\mathbb{R}^n\rightarrow (-\infty,\infty]$ be a closed proper convex function. The conjugate of $q$ is $q^*(z):=\sup_{x\in\mathbb{R}^n}\{\langle x,z\rangle-q(x)\}$. The Moreau envelope of $q$ at $x$ is defined by
\[
{\rm E}_q(x):=\min_{y\in \mathbb{R}^n}\Big\{ q(y)+\frac{1}{2}\|y-x\|^2\Big\},
\]
and the associated proximal mapping ${\rm Prox}_q(x)$ is defined as the unique solution of the above minimization problem. As proved in \cite{moreau1965proximite}, ${\rm E}_q(\cdot)$ is finite-valued, convex and differentiable with $\nabla {\rm E}_q(x)=x-{\rm Prox}_q(x)$. In addition, we can see from \cite{rockafellar1976monotone,nocedal2006numerical} that ${\rm Prox}_q(x)$ is Lipschitz continuous with modulus $1$.

\section{A unified framework to estimate the multivariate shape-constrained convex function}
\label{sec:mechanism}
In this section, we provide a unified framework for computing the least squares estimator for the multivariate shape-constrained convex function defined in \eqref{eq:min_CS}. Before describing the process, we first characterize Property $\mathcal{S}$ in the following proposition. For brevity, we omit the proof.
\begin{proposition}\label{prop:defD}
	A convex function $\psi$ has Property $\mathcal{S}$ if and only if for any $x\in \mathbb{R}^d$, the subdifferential of $\psi$ satisfies $\partial \psi(x)\subset \mathcal{D}$, where $\mathcal{D}$ is defined corresponding to Property $\mathcal{S}$ as follows:
	\begin{enumerate}[label=(S\arabic*),leftmargin=1.05cm]
		\item (monotone constraint) $\mathcal{D}=\{x\in \mathbb{R}^d\mid x_{K_1}\geq 0,x_{K_2}\leq 0\}$,
		\item (box constraint) $\mathcal{D}=\{x\in \mathbb{R}^d\mid L\leq x\leq U\}$,
		\item (Lipschitz constraint) $\mathcal{D}=\{x\in \mathbb{R}^d\mid \|x\|_q\leq L\}$, where $q$ satisfies $1/p+1/q=1$. In particular, $q=\infty,2,1$ when $p=1,2,\infty$, respectively.
	\end{enumerate}
\end{proposition}

The least squares estimation problem \eqref{eq:min_CS} attempts to find a best-fitting function $\hat{\psi}$ from the function family $\mathcal{C}_\mathcal{S}$, which is infinite dimensional. Therefore, this problem is intractable in practice. In order to design a tractable approach, we establish the following representation theorem to \eqref{eq:min_CS}, which is motivated by \cite{kuosmanen2008representation}.

\begin{theorem}\label{thm:relation_two_fun}
	Define the set of piecewise linear functions as
	\begin{equation}\label{eq:KS}
	\mathcal{K}_\mathcal{S}:=\left\{\phi:\Omega \rightarrow \mathbb{R}\left|
	\begin{aligned}
	&\phi(x)=\max_{1\leq j\leq n}\{  \theta_j+\langle \xi_j,x-X_j\rangle\},\\
	&(\theta_1,\cdots,\theta_n,\xi_1,\cdots,\xi_n)\in \mathcal{F}_\mathcal{S}
	\end{aligned}
	\right\},\right.
	\end{equation}
	where
	\begin{equation}\label{eq:FS}
	\mathcal{F}_\mathcal{S}:=\left\{ (\theta_1,\cdots,\theta_n,\xi_1,\cdots,\xi_n)\left|
	\begin{aligned}
	&\theta_i\in \mathbb{R},\xi_i\in \mathcal{D},i=1,\cdots,n, \\
	&\theta_i \geq \theta_j + \langle \xi_j,X_i - X_j\rangle,1\leq i ,j \leq n
	\end{aligned}
	\right\},\right.
	\end{equation}
	and $\mathcal{D}$ is defined as in Proposition \ref{prop:defD}. Consider the problem
	\begin{align}\label{min_KS}
	\min_{\phi\in \mathcal{K}_\mathcal{S}}\sum_{i=1}^n(\phi(X_i)-Y_i)^2.
	\end{align}
	Then the following equality holds:
	\begin{align}\label{eq:twomin}
	\min_{\psi\in\mathcal{C}_\mathcal{S}}
	\sum_{i=1}^n(\psi(X_i)-Y_i)^2=\min_{\phi\in\mathcal{K}_\mathcal{S}}
	\sum_{i=1}^n(\phi(X_i)-Y_i)^2.
	\end{align}
	Moreover, any solution $\hat{\phi}$ to \eqref{min_KS} is a solution to the problem \eqref{eq:min_CS}.
\end{theorem}
\begin{proof}
	We first prove that $\mathcal{K}_\mathcal{S}\subset \mathcal{C}_\mathcal{S}$, that is, the functions in $\mathcal{K}_\mathcal{S}$ are convex functions with Property $\mathcal{S}$. Convexity comes from the fact that any pointwise maximum function is convex. Given any function $\phi\in \mathcal{K}_\mathcal{S}$ determined by $(\theta_1,\cdots,\theta_n,\xi_1,\cdots,\xi_n)\in \mathcal{F}_\mathcal{S}$, the subdifferential of this piecewise linear function is a polyhedron according to \cite[Theorem 25.6]{rockafellar1970convex}, which is given by
	\begin{align*}
	\partial \phi(x)={\rm conv}\{\xi_i\mid i\in I(x)\},\quad I(x):=\{i\mid \theta_i+\langle \xi_i,x-X_i\rangle=\phi(x) \}.
	\end{align*}
	By the definition of $\mathcal{D}$ and $\mathcal{F}_\mathcal{S}$, we can see that $\partial \phi(x)\subset\mathcal{D}$ for any $x\in \Omega$. According to Proposition \ref{prop:defD}, the convex function $\phi$ has Property $\mathcal{S}$, which means that $\phi\in\mathcal{C}_\mathcal{S}$. Hence $\mathcal{K}_\mathcal{S}\subset \mathcal{C}_\mathcal{S}$. Therefore, we have that
	\begin{align*}
	\min_{\psi\in\mathcal{C}_\mathcal{S}}
	\sum_{i=1}^n(\psi(X_i)-Y_i)^2\leq\min_{\phi\in\mathcal{K}_\mathcal{S}}
	\sum_{i=1}^n(\phi(X_i)-Y_i)^2.
	\end{align*}
	
	Next we prove the reverse inequality. Let  $\varepsilon>0$ be an arbitrary positive number. Then there exists $\hat{\psi}_{\varepsilon}\in \mathcal{C}_\mathcal{S}$ such that
	\begin{align*}
	\sum_{i=1}^n(\hat{\psi}_{\varepsilon}(X_i)-Y_i)^2\leq\min_{\psi\in\mathcal{C}_\mathcal{S}}
	\sum_{i=1}^n(\psi(X_i)-Y_i)^2+\varepsilon.
	\end{align*}
	For $i=1,\cdots,n$,  choose $\hat{\xi}_{\varepsilon,i}\in \partial \hat{\psi}_{\varepsilon}(X_i)$. Then
	\begin{align*}
	(\hat{\psi}_{\varepsilon}(X_1),\cdots,\hat{\psi}_{\varepsilon}(X_n),\hat{\xi}_{\varepsilon,1},\cdots,\hat{\xi}_{\varepsilon,n})\in \mathcal{F}_\mathcal{S} 
	\end{align*}
	and
	\begin{align*}
	\hat{\phi}_{\varepsilon}(x):=\max_{1\leq j\leq n}\{  \hat{\psi}_{\varepsilon}(X_j)+\langle \hat{\xi}_{\varepsilon,j},x-X_j\rangle\}\in\mathcal{K}_\mathcal{S}.
	\end{align*}
	The fact that inequalities $\hat{\psi}_{\varepsilon}(X_i) \geq \hat{\psi}_{\varepsilon}(X_j) + \langle \hat{\xi}_{\varepsilon,j},X_i - X_j\rangle$ hold for all $i,j$ implies that
	\begin{align*}
	\hat{\phi}_{\varepsilon}(X_i) = \max_{1\leq j\leq n}\{  \hat{\psi}_{\varepsilon}(X_j)+\langle \hat{\xi}_{\varepsilon,j},X_i-X_j\rangle\}=\hat{\psi}_{\varepsilon}(X_i),\quad i = 1,\cdots,n.
	\end{align*}
	Then, it holds that
	\begin{align*}
	\min_{\phi\in\mathcal{K}_\mathcal{S}}
	\sum_{i=1}^n(\phi(X_i)-Y_i)^2&\leq \sum_{i=1}^n(\hat{\phi}_{\varepsilon}(X_i)-Y_i)^2=\sum_{i=1}^n(\hat{\psi}_{\varepsilon}(X_i)-Y_i)^2\\
	&\leq\min_{\psi\in\mathcal{C}_\mathcal{S}}
	\sum_{i=1}^n(\psi(X_i)-Y_i)^2+\varepsilon.
	\end{align*}
	Since the above inequality holds for any $\varepsilon>0$, the equality \eqref{eq:twomin} follows. Now suppose that $\hat{\phi}$ is an optimal solution to \eqref{min_KS}.  Since $\hat{\phi}\in \mathcal{C}_\mathcal{S}$, from \eqref{eq:twomin} we know  that $\hat{\phi}$ is a solution to the problem \eqref{eq:min_CS}. \qed
\end{proof}

The theorem above provides a tractable approach to compute \eqref{eq:min_CS} through solving \eqref{min_KS}. By definition, any function $\phi$ in $\mathcal{K}_\mathcal{S}$, which is determined by $(\theta_1,\cdots,\theta_n,\xi_1,\cdots,\xi_n)\in \mathcal{F}_\mathcal{S}$, satisfies
\begin{align*}
\phi(X_i)=\max_{1\leq j\leq n}\{ \theta_j+\langle \xi_j,X_i-X_j\rangle\}=\theta_i,\quad i=1,\cdots,n.
\end{align*}
Therefore, we can conclude the framework for computing an optimal solution to \eqref{eq:min_CS} as follows.

\paragraph{A unified framework for shape-constrained convex regression.} Suppose that $\{(\hat{\theta}_i,\hat{\xi}_i)\}_{i=1}^n$ is an optimal solution to
\begin{align}\label{convex_LSE_D}
\min_{\theta_1,\ldots,\theta_n\in\mathbb{R};  \xi_1,\ldots,\xi_n \in\mathbb{R}^{d}}  \
\Big\{  \frac{1}{2}\|\theta-Y\|^2 \Bigm\lvert (\theta_1,\cdots,\theta_n,\xi_1,\cdots,\xi_n)\in \mathcal{F}_\mathcal{S}\Big\},
\end{align}
where the feasible set $\mathcal{F}_\mathcal{S}$ is defined as in \eqref{eq:FS}. We can construct an optimal solution to \eqref{eq:min_CS} by taking
\begin{align}
\hat{\psi}(x) = \max_{1\leq j\leq n} \ \Big\{ \hat{\theta}_j + \langle \hat{\xi}_j,x- X_j\rangle
\Big\}, \quad  x\in \Omega.\label{extension_shape}
\end{align}
As one can see, the main task in our framework for estimating the shape-constrained convex function is to solve the constrained convex quadratic programming problem \eqref{convex_LSE_D}.

Define the matrix $A=e_n\otimes I_n-I_n\otimes e_n\in \mathbb{R}^{n^2\times n}$, where $``\otimes"$ denotes the Kronecker product. Then it could be seen that $A^TA = 2nI_n-2e_ne_n^T$. Denote $\xi=(\xi_1;\cdots;\xi_n)\in \mathbb{R}^{dn}$ and $B={\rm Diag}(B_1,\cdots,B_n)\in \mathbb{R}^{n^2\times dn}$ with $B_i=e_nX_i^T-X^T\in \mathbb{R}^{n\times d}$ for $i=1,\cdots,n$. Based on these notations, the problem \eqref{convex_LSE_D} can equivalently be written as
\begin{align}
\min_{\theta\in \mathbb{R}^n,\xi\in \mathbb{R}^{dn}} \ \Big\{\frac{1}{2}\|\theta-Y\|^2+p(\xi)+\delta_{+}(A\theta+B\xi)\Big\},\tag{P}\label{reformulated_problem}
\end{align}
where $p(\xi)= \sum_{i=1}^n \delta_\mathcal{D}(\xi_i)$ and $\delta_{\pm}(\cdot)$ is the indicator function of $\mathbb{R}^{n^2}_{\pm}$.

\paragraph{Smooth approximation.} Note that the function $\hat{\psi}$ obtained by \eqref{extension_shape} is nonsmooth. When a smooth function is required, we can compute a smooth approximation to $\hat{\psi}$. The idea of Nesterov's smoothing \cite{nesterov2005smooth} could be applied, and the details is described in \cite[Section 3]{mazumder2019computational}. Alternatively, one can use the Moreau envelope as a smooth approximation of $\hat{\psi}$, namely
\begin{align}
\hat{\psi}^{\rm M}_{\tau}(x)=\tau {\rm E}_{\hat{\psi}/\tau}(x)=\min_{y\in \mathbb{R}^d}\  \Big\{\hat{\psi}(y)+\frac{\tau}{2}\|y-x\|^2
\Big\},\label{smooth_moreau}
\end{align}
where $\tau>0$ is a regularization parameter. Note that
\begin{align*}
\hat{\psi}^{\rm M}_{\tau}(x) = \min_{y\in \mathbb{R}^d,t\in \mathbb{R}}\ \Big\{t+\frac{\tau}{2}\|y-x\|^2\Bigm\lvert
t\geq \langle \hat{\xi}_j,y\rangle -\langle \hat{\xi}_j,X_j\rangle+\hat{\theta}_j,\ j=1,\cdots,n
\Big\},
\end{align*}
the unique optimal solution ${\rm Prox}_{\hat{\psi}/\tau}(x)$ of \eqref{smooth_moreau} can be obtained by solving a quadratic programming of dimension $d+1$, which could be efficiently computed by {\tt Gurobi} or {\tt MOSEK}. One can see that for any $\tau>0$, $\hat{\psi}^{\rm M}_{\tau}$ is convex, and differentiable with $\nabla \hat{\psi}^{\rm M}_{\tau}(x)=\tau(x-{\rm Prox}_{\hat{\psi}/\tau}(x))$. In addition, according to \cite{beck2012smoothing}, the approximation $\hat{\psi}^{\rm M}_{\tau}$ of $\hat{\psi}$ satisfies the approximation bound
\begin{align*}
0\leq \hat{\psi}(x)-\hat{\psi}^{\rm M}_{\tau}(x)\leq \frac{1}{2\tau}{\rm dist}^2(0,\partial \hat{\psi}(x))\leq \frac{L^2}{2\tau},\quad \forall x\in \Omega,
\end{align*}
where $L=\max\{\|\xi_j\|_2\mid j=1,\cdots,n\}$.

\section{A proximal augmented Lagrangian method ({\tt proxALM}) for \eqref{reformulated_problem}}
\label{sec:pALM}
The augmented Lagrangian method is a desirable method for solving convex composite programming problems due to its superlinear convergence. To take advantage of the fast local convergence, we design a proximal augmented Lagrangian method for solving \eqref{reformulated_problem}. In order to solve the {\tt proxALM} subproblems, we propose a semismooth Newton method, which is proved to have quadratic convergence. By making full use of the special structure of the problem, we can exploit the second-order sparsity structure of the underlying subproblems to greatly reduce the computational cost. It should be noted that in addition to the algorithmic design, the most important part of the {\tt proxALM} is the numerical implementation, which will be discussed in detail in the next section.

The Lagrangian function associated with the unconstrained minimization problem \eqref{reformulated_problem} is given by
\begin{align*}
&l(\theta,\xi;u,v)\\
&=\inf_{\eta\in \mathbb{R}^{n^2},y\in \mathbb{R}^{dn}}\ \Big\{\frac{1}{2}\|\theta-Y\|^2+p(\xi-y)
+\delta_{+}(A\theta+B\xi-\eta) -\langle v,y\rangle-\langle u,\eta\rangle\Big\}\\
&=\frac{1}{2}\|\theta-Y\|^2-p^*(-v)-\langle v,\xi\rangle -\delta_{+}(u)-\langle u,A\theta+B\xi\rangle .
\end{align*}
The dual problem of \eqref{reformulated_problem}, $\max_{u\in\mathbb{R}^{n^2},v\in \mathbb{R}^{dn}}\min_{\theta\in\mathbb{R}^n,\xi\in\mathbb{R}^{dn}}l(\theta,\xi;u,v)$, is explicitly given as follows:
\begin{equation}\tag{D}\label{D}
\begin{aligned}
&\max_{u\in\mathbb{R}^{n^2},v\in \mathbb{R}^{dn}}\ \Big\{-\frac{1}{2}\|A^Tu\|^2-\langle Y,A^Tu\rangle-p^*(-v)-\delta_{+}(u)\Big\}
\\ \qquad  & {\rm s.t.}\quad B^Tu+v=0.
\end{aligned}
\end{equation}
The Karush-Kuhn-Tucker (KKT) conditions associated with \eqref{reformulated_problem} and \eqref{D} are:
\begin{align}\label{KKT_system}
\theta-Y-A^T u=0,\ B^T u+v=0,\ -v\in\partial p(\xi),\ -u\in \partial \delta_{+}(A\theta+B\xi).
\end{align}
The augmented Lagrangian function associated with \eqref{reformulated_problem} for any fixed $\sigma>0$ can be derived as
\begin{align*}
&\mathcal{L}_{\sigma}(\theta,\xi;u,v)\\
&=\sup_{s\in \mathbb{R}^{n^2},t\in \mathbb{R}^{dn}}\
\Big\{ l(\theta,\xi;s,t)-\frac{1}{2\sigma}\|s-u\|^2-\frac{1}{2\sigma}
\|t-v\|^2\Big\}\\
&=\frac{1}{2}\|\theta-Y\|^2+\sigma{\rm E}_{p}(\xi-\frac{v}{\sigma})+\sigma{\rm E}_{\delta_{+}}(A\theta+B\xi-\frac{u}{\sigma}) -\frac{1}{2\sigma}\|u\|^2-\frac{1}{2\sigma}\|v\|^2.
\end{align*}
Our proposed {\tt proxALM} for solving \eqref{reformulated_problem} has the template as in Algorithm \ref{alg:pALM}.

\begin{algorithm}[h]
	\caption{: Proximal augmented Lagrangian method for \eqref{reformulated_problem}}
	\label{alg:pALM}
	\begin{algorithmic}[1]
		\STATE {\bfseries Initialization:} Let $H_1\in \mathbb{R}^{n\times n}$, $H_2\in \mathbb{R}^{dn\times dn}$ be given symmetric and positive definite matrices, and $\{\varepsilon_k\}$ be a given summable sequence of nonnegative numbers. Choose an initial point $(\theta^0,\xi^0,u^0,v^0)\in \mathbb{R}^n\times \mathbb{R}^{dn}\times\mathbb{R}^{n^2}\times \mathbb{R}^{dn}$, $\sigma_0 > 0$. For $k = 0, 1, 2, \dots$
		\REPEAT
		\STATE {\bfseries Step 1}. Compute
		\begin{eqnarray}\label{eq:solve_thetaxi}
		\begin{aligned}
		&(\theta^{k+1},\xi^{k+1})\\
		&\! \approx \! \underset{\theta\in \mathbb{R}^n,\xi\in \mathbb{R}^{dn}}{\arg\min}  \!\Big\{\Phi_k(\theta,\xi)\!=\!\mathcal{L}_{\sigma_k}(\theta,\xi;u^k,v^k)\!+
\!\frac{1}{2\sigma_k}\|\theta-\theta^k\|_{H_1}^2\!+\!\frac{1}{2\sigma_k}\|\xi-\xi^k\|_{H_2}^2\Big\}
		\end{aligned}
		\end{eqnarray}
		such that the approximate solution $(\theta^{k+1},\xi^{k+1})$ satisfies the following stopping criterion:
		\begin{align}
		\|\nabla \Phi_k(\theta^{k+1},\xi^{k+1})\|&\leq \frac{\sqrt{\lambda_{\min}}}{\sigma_k}\varepsilon_k, \tag{A}\label{A}
		\end{align}
		where $\lambda_{\min} =\min\{\lambda_{\min}(H_1),\lambda_{\min}(H_2),1 \}$.
		\\[3pt]
		\STATE {\bfseries Step 2}. Update $u$, $v$ by
		\begin{align}
		&u^{k+1}=-\sigma_k\Big[ A\theta^{k+1}+B\xi^{k+1}-u^k/\sigma^k-\Pi_{+}(A\theta^{k+1}+B\xi^{k+1}-u^k/\sigma^k)\Big],\notag\\
		&v^{k+1}=-\sigma_k \Big[ \xi^{k+1}-v^k/\sigma_k-{\rm Prox}_p(\xi^{k+1}-v^k/\sigma^k) \Big].\notag
		\end{align}
		\\[3pt]
		\STATE {\bfseries Step 3}. Update $\sigma_{k+1} \uparrow \sigma_{\infty} \leq \infty$.
		\UNTIL{Stopping criterion is satisfied.}
	\end{algorithmic}
\end{algorithm}

\subsection{Convergence results for the {\tt proxALM}}
\label{subsec:convergencepALM}
Define the maximal monotone operator $\mathcal{T}_{l}$ as
\begin{align*}
\mathcal{T}_{l}(\theta,\xi,u,v)=\Big\{(\theta',\xi',u',v')\mid(\theta',\xi',-u',-v')\in \partial l(\theta,\xi,u,v)
\Big\},
\end{align*}
and the block diagonal operator $\Sigma = {\rm Diag}(H_1,H_2,I_{n^2},I_{dn})$. Note that the solution set of the KKT system \eqref{KKT_system} is exactly $\mathcal{T}_{l}^{-1}(0)$.

We follow the idea of \cite[Theorem 2.3 and Theorem 2.5]{li2020asymptotically} to get the following convergence results of Algorithm \ref{alg:pALM}, where the details of the proof are omitted here.

\begin{theorem}\label{thm:convergence_pALM}
	Suppose that the solution set to the KKT conditions \eqref{KKT_system} is nonempty, that is $\Lambda:=\mathcal{T}_{l}^{-1}(0)\neq \emptyset$.\\
	(1) Let $\{(\theta^{k},\xi^{k},u^{k},v^{k})\}$ be the infinite sequence generated by Algorithm \ref{alg:pALM}. Then $\{(\theta^{k},\xi^{k},u^{k},v^{k})\}$ is bounded,  $\{(\theta^{k},\xi^{k})\}$ converges to an optimal solution of \eqref{reformulated_problem}, and $\{(u^{k},v^{k})\}$ converges to an optimal solution of \eqref{D}.\\
	(2) Let $r:=\sum_{k=0}^{\infty}\varepsilon_k+{\rm dist}_{\Sigma}((\theta^0,\xi^0,u^0,v^0),\Lambda)$. Assume that for this $r>0$, there exists a constant $\kappa>0$ such that $\mathcal{T}_{l}$ satisfies the following error bound condition: for all $(\theta,\xi,u,v)$ satisfying ${\rm dist}((\theta,\xi,u,v),\Lambda)\leq r$, it holds that
	\begin{equation}\label{error_bound}
	{\rm dist}((\theta,\xi,u,v),\Lambda)\leq \kappa {\rm dist}(0,\mathcal{T}_{l}(\theta,\xi,u,v)).
	\end{equation}
	Suppose that $\{(\theta^{k},\xi^{k},u^{k},v^{k})\}$ is the sequence generated by Algorithm \ref{alg:pALM}, where in \textbf{Step 1}, the approximate solution $(\theta^{k+1},\xi^{k+1})$ also satisfies the stopping criterion
	\begin{align}
	\|\nabla \Phi_k(\theta^{k+1},\xi^{k+1})\|\leq \frac{\delta_k\sqrt{\lambda_{\min}}}{\sigma_k}\|(\theta^{k+1},\xi^{k+1},u^{k+1},v^{k+1})-(\theta^{k},\xi^{k},u^{k},v^{k}) \|_{\Sigma},\tag{B}\label{B}
	\end{align}
	and $\{\delta_k\mid 0\leq \delta_k<1\}$ is a given summable sequence. Then it holds for all $k\geq 0$ that
	\begin{equation}\label{eq:rate_pppa}
	{\rm dist}_{\Sigma}((\theta^{k+1},\xi^{k+1},u^{k+1},v^{k+1}),\Lambda)\leq \mu_k {\rm dist}_{\Sigma}((\theta^{k},\xi^{k},u^{k},v^{k}),\Lambda),
	\end{equation}
	where
	\begin{align*}
	\mu_k=\frac{1}{1-\delta_k}\Big( \delta_k+\frac{(1+\delta_k)\kappa \lambda_{\max} }{\sqrt{\sigma_k^2+\kappa^2\lambda_{\max}^2}}\Big)\rightarrow \mu_{\infty}=\frac{\kappa\lambda_{\max}}{\sqrt{\sigma_{\infty}^2+\kappa^2\lambda_{\max}^2}}<1,
	\end{align*}
	and $\lambda_{\max} = \max\{\lambda_{\max}(H_1),\lambda_{\max}(H_2),1\}$.
\end{theorem}

As one can see from Theorem \ref{thm:convergence_pALM}, the fast linear convergence rate of Algorithm \ref{alg:pALM} depends on the error bound condition \eqref{error_bound} for the maximal monotone operator $\mathcal{T}_{l}$. For specifying whether the error bound condition \eqref{error_bound} holds for different choices of the closed convex set $\mathcal{D}$, we give the following remark.
\begin{remark}\label{remark:Tl}
	It is well known that any polyhedral multifunction is upper Lipschitz continuous at every point of its domain according to \cite{robinson1981some}, which means it satisfies the error bound condition \eqref{error_bound} for any $r>0$. For the cases when $\mathcal{D}$ is a polyhedral set, e.g. $\mathcal{D}=\mathbb{R}_{+}^d(\mathbb{R}_{-}^d)$ or $\mathcal{D}=\{x\in \mathbb{R}^d\mid \|x\|_q\leq L\}$ with $q=1$ or $q=\infty$, $\mathcal{T}_{l}$ is a polyhedral multifunction, and hence it satisfies the error bound condition \eqref{error_bound}. In general, one needs addtional assumptions such as partial complementarity for the error bound condition \eqref{error_bound} to hold with the presence of nonpolyhedral constraints.
\end{remark}

\subsection{A semismooth Newton method for solving the {\tt proxALM} subproblems}
\label{subsec:ssn}
One can see that the most computationally intensive step in the {\tt proxALM} is in solving the subproblem \eqref{eq:solve_thetaxi}. Here we describe how it can be solved efficiently by the semismooth Newton method. For any given $\sigma>0$, $(\tilde{\theta},\tilde{\xi},\tilde{u},\tilde{v})\in \mathbb{R}^n\times \mathbb{R}^{dn}\times \mathbb{R}^{n^2}\times \mathbb{R}^{dn}$, we aim to solve the {\tt proxALM} subproblem, which has the form:
\begin{align}
\min_{\theta\in \mathbb{R}^n,\xi\in\mathbb{R}^{dn}}\ \Big\{ \Phi(\theta,\xi):=\mathcal{L}_{\sigma}(\theta,\xi;\tilde{u},\tilde{v})+\frac{1}{2\sigma}\|\theta-\tilde{\theta}\|_{H_1}^2+\frac{1}{2\sigma}\|\xi-\tilde{\xi}\|_{H_2}^2\Big\}.\label{eq:pALMsub}
\end{align}
Since $\Phi(\cdot,\cdot)$ is strongly convex, the above minimization problem admits a unique solution $(\bar{\theta},\bar{\xi})$, which can be computed by solving the nonsmooth equation
\begin{align}\label{eq:nablapsi}
\nabla \Phi(\theta,\xi)=0,
\end{align}
where
\begin{align*}
\nabla \Phi(\theta,\xi)&=\left(
\begin{aligned}
\sigma A^T \Big[A\theta+B\xi-\frac{\tilde{u}}{\sigma}-\Pi_{+}(A\theta+B\xi-\frac{\tilde{u}}{\sigma})\Big]\\
\sigma B^T \Big[A\theta+B\xi-\frac{\tilde{u}}{\sigma}-\Pi_{+}(A\theta+B\xi-\frac{\tilde{u}}{\sigma})\Big]
\end{aligned}\right)\\
&\qquad +
\left(
\begin{aligned}
&\theta-Y+\frac{1}{\sigma}H_1(\theta-\tilde{\theta})\\
\sigma\Big[\xi-\frac{\tilde{v}}{\sigma}-&{\rm Prox}_p(\xi-\frac{\tilde{v}}{\sigma})
\Big]+\frac{1}{\sigma}H_2(\xi-\tilde{\xi})
\end{aligned}\right).
\end{align*}

In order to apply the {\tt SSN} to solve the above nonsmooth equation, we need a suitable generalized Jacobian of $\nabla \Phi(\cdot,\cdot)$. Here we choose the following set as the candidate:
\begin{align*}
\hat{\partial}^2\Phi(\theta,\xi)&=\sigma \left(\begin{aligned}
A^T\\
B^T
\end{aligned}\right)\Big[I_{n^2}-\partial \Pi_{+}(A\theta+B\xi-\frac{\tilde{u}}{\sigma})\Big]\left(\begin{aligned}
&A & B
\end{aligned}\right)\\
&\quad +\left(\begin{aligned}
I_n+&\frac{1}{\sigma}H_1\\
&0
\end{aligned}\right.
\left.\begin{aligned}
&0\\
\sigma \Big[ I_{dn}-\partial{\rm Prox}_p&(\xi-\frac{\tilde{v}}{\sigma})\Big]+\frac{1}{\sigma}H_2
\end{aligned}\right),
\end{align*}
where $\partial \Pi_{+}$ is the Clarke generalized Jacobian of $\Pi_{+}(\cdot)$ defined as
\begin{align*}
\partial \Pi_{+}(\eta) =
\left\{ {\rm Diag}(q) \left |
\begin{array}{ll}
q_i=0 & \mbox{if $\eta_i < 0$}
\\
q_i\in [0,1]  & \mbox{if $\eta_i = 0$}
\\
q_i =1 & \mbox{otherwise}
\end{array}\right.
\right\},\quad \forall \eta\in \mathbb{R}^{n^2} 
\end{align*}	
and $\partial{\rm Prox}_p$ is the Clarke generalized Jacobian of ${\rm Prox}_p$ which will be described in Section \ref{sec:proximalmapping}.

We give the following proposition to identify the strong semismoothness of $\nabla \Phi(\cdot,\cdot)$ with respect to $\hat{\partial}^2 \Phi(\cdot,\cdot)$, where the definition of strong semismoothness could be found in \cite{mifflin1977semismooth,kummer1988newton,qi1993nonsmooth,sun2002semismooth}.
\begin{proposition}\label{prop: semismooth_of_nablaPhi}
Suppose that 
${\rm Prox}_{p} (\cdot)$ is strongly semismooth with respect to the Clarke generalized Jacobian $\partial{\rm Prox}_{p} (\cdot)$. Then $\nabla \Phi(\cdot,\cdot)$ is strongly semismooth with respect to $\hat{\partial}^2 \Phi(\cdot,\cdot)$.
\end{proposition}
\begin{proof}
	By the definition of $\partial \Pi_{+}(\cdot)$, we can see that $\partial \Pi_{+}(\cdot)$ is nonempty, compact valued, and upper-semicontinuous. Together with the property of $\partial{\rm Prox}_{p} (\cdot)$, it could be seen that the multifunction $\hat{\partial}^2 \Phi(\cdot,\cdot)$ is nonempty, compact valued, and upper-semicontinuous.
	
	Note that for any $(\theta,\xi)\in \mathbb{R}^n\times \mathbb{R}^{dn}$, $\nabla \Phi(\cdot,\cdot)$ is directionally differentiable at $(\theta,\xi)$. Let $(\Delta \theta,\Delta\xi)\in \mathbb{R}^n\times \mathbb{R}^{dn}$ be such that $\|(\Delta \theta,\Delta\xi)\|$ is sufficiently small. Let ${\cal H}\in \hat{\partial}^2 \Phi(\theta+\Delta \theta,\xi+\Delta \xi)$, then by definition, there exists $P\in \partial \Pi_{+}(A(\theta+\Delta \theta)+B(\xi+\Delta\xi)-\tilde{u}/\sigma)$ and $Q\in \partial {\rm Prox}_p(\xi+\Delta\xi-\tilde{v}/\sigma)$ such that
	\begin{align*}
	{\cal H}=\sigma \left(\begin{aligned}
	A^T\\
	B^T
	\end{aligned}\right)\Big(I_{n^2}-P\Big)\left(\begin{aligned}
	&A & B
	\end{aligned}\right)+\left(\begin{aligned}
	I_n+&\frac{1}{\sigma}H_1\\
	&0
	\end{aligned}\right.
	\left.\begin{aligned}
	&0\\
	\sigma \Big[ I_{dn}-Q&\Big]+\frac{1}{\sigma}H_2
	\end{aligned}\right).
	\end{align*}
	Since $\Pi_{+}(\cdot)$ is piecewise affine, we know that
	\begin{align*}
	\Pi_{+}(A(\theta+\Delta \theta)+B(\xi+\Delta\xi)-\tilde{u}/\sigma)=\Pi_{+}(A\theta+B\xi-\tilde{u}/\sigma)+P(\Delta \theta;\Delta \xi).
	\end{align*}
	By the strong semismoothness of ${\rm Prox}_p(\cdot)$ with respect to $\partial {\rm Prox}_p(\cdot)$, we have that
	\begin{align*}
	{\rm Prox}_p(\xi+\Delta\xi-\tilde{v}/\sigma) = {\rm Prox}_p(\xi-\tilde{v}/\sigma)+Q\Delta\xi +O(\|\Delta\xi\|^2).
	\end{align*}
	Therefore, it holds that
	\begin{align*}
	\nabla \Phi(\theta+\Delta\theta,\xi+\Delta\xi)-\nabla \Phi(\theta,\xi)-{\cal H}(\Delta\theta,\Delta\xi)
	=O(\|(\Delta\theta,\Delta\xi)\|^2),
	\end{align*}
	which means $\nabla \Phi(\cdot,\cdot)$ is strongly semismooth with respect to $\hat{\partial}^2 \Phi(\cdot,\cdot)$. \qed
\end{proof}

With the suitably chosen generalized Jacobian $\hat{\partial}^2 \Phi(\cdot,\cdot)$, we can design the semismooth Newton method in Algorithm \ref{alg:ssn}, which is a generalization of the standard Newton method, for solving \eqref{eq:pALMsub}.

\begin{algorithm}
	\caption{: Semismooth Newton method for \eqref{eq:pALMsub}}
	\label{alg:ssn}
	\begin{algorithmic}[1]
		\STATE {\bfseries Initialization:} Given $(\theta^0,\xi^0)\in \mathbb{R}^n\times \mathbb{R}^{dn}$, $\bar{\gamma}\in (0, 1)$, $\tau \in (0, 1]$, $\delta \in (0, 1)$, and $\mu \in (0, 1/2)$. For $j = 0, 1, 2, \dots$
		\REPEAT
		\STATE {\bfseries Step 1}. Select an element $\mathcal{H}_j \in \hat{\partial}^{2} \Phi(\theta^j,\xi^j)$. Apply a direct method or the preconditioned conjugate gradient ({\tt PCG}) method to find an approximate solution $(\Delta \theta^j,\Delta \xi^j)\in \mathbb{R}^{n}\times \mathbb{R}^{dn}$ to
		\begin{equation}\label{eq: cg-system}
		\mathcal{H}_j(\Delta \theta^j,\Delta \xi^j) \approx - \nabla \Phi(\theta^j,\xi^j),
		\end{equation}
		such that $R_j:=\mathcal{H}_j(\Delta \theta^j,\Delta \xi^j)+\nabla \Phi(\theta^j,\xi^j)$ satisfies $\|R_j\| \leq \min(\bar{\gamma}, \|\nabla\Phi(\theta^j,\xi^j)\|^{1+\tau})$.
		\\[3pt]
		\STATE {\bfseries Step 2}. Set $\alpha_j = \delta^{m_j}$, where $m_j$ is the smallest nonnegative integer $m$ such that
		$$\Phi(\theta^j + \delta^m \Delta \theta^j,\xi^j + \delta^m \Delta \xi^j) \leq \Phi(\theta^j,\xi^j) + \mu\delta^m \langle \nabla\Phi(\theta^j,\xi^j), (\Delta \theta^j,\Delta \xi^j)\rangle .$$
		\\[3pt]
		\STATE{\bfseries Step 3}. Set $\theta^{j+1} = \theta^j + \alpha_j \Delta \theta^j$, $\xi^{j+1}=\xi^j+\alpha_j \Delta \xi^j$.
		\UNTIL{Stopping criterion \eqref{A} or criterion \eqref{B} based on $\theta^{j+1}$ and $\xi^{j+1}$ is satisfied.}
	\end{algorithmic}
\end{algorithm}

The convergence analysis for Algorithm \ref{alg:ssn} can be established as follows.
\begin{theorem}\label{thm:convergenceSSN}
Suppose that 
${\rm Prox}_{p} (\cdot)$ is strongly semismooth with respect to \\
$\partial{\rm Prox}_{p} (\cdot)$. Let $\{(\theta^{j},\xi^{j})\}$ be the infinite sequence generated by Algorithm \ref{alg:ssn}. Then, $\{(\theta^{j},\xi^{j})\}$ converges to the unique optimal solution $(\bar{\theta},\bar{\xi})$ of problem \eqref{eq:pALMsub}, and
	\begin{align*}
	\|(\theta^{j+1},\xi^{j+1})-(\bar{\theta},\bar{\xi})\|=O(\|(\theta^{j},\xi^{j})-(\bar{\theta},\bar{\xi})\|^{1+\tau}).
	\end{align*}
\end{theorem}
\begin{proof}
	According to Proposition \ref{prop: semismooth_of_nablaPhi}, we have that $\nabla \Phi(\cdot,\cdot)$ is strongly semismooth with respect to $\hat{\partial}^2 \Phi(\cdot,\cdot)$. From \cite[Proposition 3.3 and Theorem 3.4]{zhao2010newton}, we can see that $\{(\theta^j,\xi^j)\}$ converges to the unique optimal solution $(\bar{\theta},\bar{\xi})$. By the formulation of $\hat{\partial}^2\Phi(\cdot,\cdot)$, we have that all the elements in $\hat{\partial}^2\Phi(\theta,\xi)$ are symmetric and positive definite for any $(\theta,\xi)\in\mathbb{R}^{n}\times\mathbb{R}^{dn}$ due to the positive definiteness of $H_1$ and $H_2$. Then for sufficiently large $j$, we have that $\{\|{\cal H}_j^{-1}\|\}$ is uniformly bounded from \cite[Lemma 7.5.2]{facchinei2007finite}, and thus
	\begin{align}
	&\|(\theta^{j},\xi^{j})+(\Delta \theta^j,\Delta \xi^j)-(\bar{\theta},\bar{\xi})\|=
	\|(\theta^{j},\xi^{j})-(\bar{\theta},\bar{\xi})+{\cal H}_j^{-1}(R_j-\nabla \Phi(\theta^j,\xi^j))\|\notag\\
	&\leq \|{\cal H}_j^{-1}\|\Big( \|\nabla\Phi(\theta^j,\xi^j)\|^{1+\tau}+\|{\cal H}_j((\theta^{j},\xi^{j})-(\bar{\theta},\bar{\xi}))-\nabla\Phi(\theta^j,\xi^j)\|\Big)\notag\\
	&= O(\|\nabla\Phi(\theta^j,\xi^j)-\nabla\Phi(\bar{\theta},\bar{\xi})\|^{1+\tau})\notag\\
	&\quad +O(\|\nabla\Phi(\theta^j,\xi^j)-\nabla\Phi(\bar{\theta},\bar{\xi})-{\cal H}_j((\theta^{j},\xi^{j})-(\bar{\theta},\bar{\xi}))\|)\notag\\
	&=O(\|(\theta^{j},\xi^{j})-(\bar{\theta},\bar{\xi})\|^{1+\tau}),
	\label{eq: onetau}
	\end{align}
	where we have used the strong semismoothness property of $\nabla \Phi(\cdot,\cdot)$ at $(\bar{\theta},\bar{\xi})$ to get the the last equality. In addition, we could prove that there exists $\hat{\delta}>0$ such that
	\begin{align*}
	\langle \nabla\Phi(\theta^j,\xi^j), (\Delta \theta^j,\Delta \xi^j)\rangle \leq -\hat{\delta} \|(\Delta \theta^j,\Delta \xi^j)\|^2.
	\end{align*}
	Together with \cite[Proposition 7]{li2018efficiently} and \cite[Proposition 8.3.18]{facchinei2007finite}, we can derive that for $\mu\in (0,1/2)$, there exists an integer $j_0$ such that for all $j\geq j_0$,
	\begin{align*}
	\Phi(\theta^j + \Delta \theta^j,\xi^j + \Delta \xi^j) \leq \Phi(\theta^j,\xi^j) + \mu \langle \nabla\Phi(\theta^j,\xi^j), (\Delta \theta^j,\Delta \xi^j)\rangle,
	\end{align*}
	which implies $\theta^{j+1} = \theta^j +\Delta \theta^j$, $\xi^{j+1}=\xi^j+\Delta \xi^j$, for $j\geq j_0$. Combing with \eqref{eq: onetau}, we complete the proof.\qed
\end{proof}

Note that in the above theorem, we have proved the Q-superlinear convergence of the sequence $\{(\theta^j,\xi^j)\}$, which implies the R-superlinear convergence of $\{\|\nabla \Phi(\theta^j,\xi^j)\|\}$ due to the fact that
\begin{align*}
\|\nabla\Phi(\theta^j,\xi^j)\|=\|\nabla\Phi(\theta^j,\xi^j)-\nabla\Phi(\bar{\theta},\bar{\xi})\|=O(\|(\theta^{j},\xi^{j})-(\bar{\theta},\bar{\xi})\|).
\end{align*}
This further implies that  condition (A) or condition (B) in Algorithm \ref{alg:ssn} can be met in a small number of iterations, typically at most dozens of steps. 

\begin{remark}\label{remark:stronglysemismooth}
	As a side note, for each closed convex set $\mathcal{D}$ in Proposition \ref{prop:defD}, we will prove in Proposition \ref{prop: semismooth_p} that the assumption on  $\partial{\rm Prox}_{p} (\cdot)$ in Theorem \ref{thm:convergenceSSN} always holds.
\end{remark}

\section{Numerical implementation of Algorithm {\tt proxALM}}
\label{sec:proximalmapping}
In this section, we discuss some numerical details concerning the efficient implementation of the proposed {\tt proxALM}. For implementing the {\tt proxALM}, we need the proximal mapping ${\rm Prox}_p(\xi)$ for any $\xi\in\mathbb{R}^{dn}$ and its generalized Jacobian. In addition, when evaluating the function value of the problem \eqref{D}, we need the formula for $p^*(\cdot)$.

\subsection{Computation associated with $\mathcal{D}$}
\label{subsec:D}
For any $\xi=(\xi_1;\cdots;\xi_n)\in \mathbb{R}^{dn}$, since $p(\xi)= \sum_{i=1}^n \delta_\mathcal{D}(\xi_i)$, we have that
\begin{equation}\label{eq:pfunction}
\begin{aligned}
&p^*(\xi)=\sum_{i=1}^n \delta_\mathcal{D}^*(\xi_i),\quad
{\rm Prox}_{ p}(\xi)=\left(\begin{array}{c}
\Pi_\mathcal{D}(\xi_1)\\
\vdots\\
\Pi_\mathcal{D}(\xi_n)
\end{array}\right),\\
&\partial {\rm Prox}_{ p}(\xi)=\left(\begin{array}{ccc}
\partial \Pi_\mathcal{D}(\xi_1) & &\\
&\ddots&\\
&&\partial\Pi_\mathcal{D}(\xi_n)
\end{array}\right),
\end{aligned}
\end{equation}
which means that we only need to focus on $\delta_\mathcal{D}^*(\cdot)$, $\Pi_\mathcal{D}(\cdot)$ and $\partial \Pi_\mathcal{D}(\cdot)$ for each of the set $\mathcal{D}$ defined in Proposition \ref{prop:defD}. We summarize the results in Table \ref{table_conjugate} -- Table \ref{table_jacobian}, where the detailed derivation associated with the case when $\mathcal{D}=\{x\in \mathbb{R}^d\mid \|x\|_{1}\leq L\}$ is given in Appendix \ref{appendix: derive_1norm}.

\begin{table}[ht]
	\caption{Conjugate function $\delta_\mathcal{D}^*(\cdot)$}
	\label{table_conjugate}
	\begin{threeparttable}
		\begin{tabular}{cc}
			\hline\noalign{\smallskip}
			$\mathcal{D}$ &  $\delta_\mathcal{D}^*(x)$   \\
			\noalign{\smallskip}\hline\noalign{\smallskip}
			$\mathcal{D}=\{x\in \mathbb{R}^d\mid x_{K_1}\geq 0,x_{K_2}\leq 0\}$ & $\delta_\mathcal{D}^*(x)=\delta_{-}(x_{K_1})+\delta_{+}(x_{K_2})+\delta_{\{0\}}(x_{K_3})$\tnote{*}\\
			$\mathcal{D}=\{x\in \mathbb{R}^d\mid L\leq x\leq U\}$ &  $\delta_\mathcal{D}^*(x)= \langle U,\max\{x,0\}\rangle + \langle L,\min\{x,0\}\rangle$\\
			$\mathcal{D}=\{x\in \mathbb{R}^d\mid \|x\|_{q}\leq L\}$ & $\delta_\mathcal{D}^*(x)=L\|x\|_p,\quad 1/p+1/q=1$\\
			\noalign{\smallskip}\hline
		\end{tabular}
		\begin{tablenotes}\footnotesize
			\item[*] $K_3:=\{1,\cdots,d\} \backslash (K_1 \cup K_2)$
		\end{tablenotes}
	\end{threeparttable}
\end{table}

\begin{table}[ht]
	\caption{Proximal mapping $\Pi_\mathcal{D}(\cdot)$}
	\label{table_proximalmapping}
	\tabcolsep 2.5pt
	\begin{threeparttable}
		\begin{tabular}{cc}
			\hline\noalign{\smallskip}
			$\mathcal{D}$ &  $\Pi_\mathcal{D}(\cdot)$   \\
			\noalign{\smallskip}\hline\noalign{\smallskip}
			$\mathcal{D}=\{x\in \mathbb{R}^d\mid x_{K_1}\geq 0,x_{K_2}\leq 0\}$ & $(\Pi_\mathcal{D}(x))_i=\left\{\begin{aligned}
			&0 && \mbox{if}\ i\in K_1, x_{i}< 0, \mbox{ or }i\in K_2,x_i>0\\
			&x_i && \mbox{otherwise}
			\end{aligned}\right.$ \\
			$\mathcal{D}=\{x\in \mathbb{R}^d\mid L\leq x\leq U\}$ &  $(\Pi_\mathcal{D}(x))_i=\left\{\begin{aligned}
			&x_i && \mbox{if}\ L_i\leq x_i\leq U_i\\
			&0 && \mbox{otherwise}
			\end{aligned}\right.$\\
			$\mathcal{D}=\{x\in \mathbb{R}^d\mid \|x\|_{\infty}\leq L\}$ & $(\Pi_\mathcal{D}(x))_i=\left\{\begin{aligned}
			&x_i && \mbox{if}\ |x_i|\leq L \\
			&{\rm sign}(x_i)L && \mbox{if}\ |x_i|> L
			\end{aligned}\right.$\\
			$\mathcal{D}=\{x\in \mathbb{R}^d\mid \|x\|_{2}\leq L\}$ & $\Pi_\mathcal{D}(x)=\left\{\begin{aligned}
			&x && \mbox{if}\ \|x\|_2\leq L \\
			&L\frac{x}{\|x\|_2} && \mbox{otherwise}
			\end{aligned}\right.$\\
			$\mathcal{D}=\{x\in \mathbb{R}^d\mid \|x\|_{1}\leq L\}$ & $\Pi_\mathcal{D}(x)=\left\{\begin{aligned}
			&x && \mbox{if}\ \|x\|_1\leq L \\
			&L P_x \Pi_{\Delta_d}(P_x x/L)\tnote{*}&& \mbox{otherwise}
			\end{aligned}\right.$\\
			\noalign{\smallskip}\hline
		\end{tabular}
		\begin{tablenotes}\footnotesize
			\item[*] $P_x={\rm Diag}({\rm sign}(x))\in \mathbb{R}^{d\times d}$, $\Pi_{\Delta_d}(\cdot)$ denotes the projection onto the simplex $\Delta_d=\{x\in \mathbb{R}^d\mid e_d^T x=1,x\geq 0\}$, which can be computed in $O(d\log (d))$ operations.
		\end{tablenotes}
	\end{threeparttable}
\end{table}

\begin{table}[ht]
	\caption{Generalized Jacobian of $\Pi_\mathcal{D}(\cdot)$}
	\label{table_jacobian}
	\tabcolsep 2.5pt
	\begin{threeparttable}
		\begin{tabular}{cc}
			\hline\noalign{\smallskip}
			$\mathcal{D}$ &  $\partial\Pi_\mathcal{D}(\cdot)$ \\
			\noalign{\smallskip}\hline\noalign{\smallskip}
			$\mathcal{D}=\left\{ x\in \mathbb{R}^d \left| 
			\begin{aligned}
			& x_{K_1}\geq 0\\
			& x_{K_2}\leq 0 
			\end{aligned}
			\right\}\right.$  & $\partial \Pi_\mathcal{D}(x)= {\rm Diag}(u),\ u_i\in \left\{\begin{aligned}
			&\{0\} && \begin{aligned} &\mbox{if}\ i\in K_1, x_{i}< 0\\&\mbox{ \ or }i\in K_2,x_i>0 \end{aligned}\\
			&[0,1] && \mbox{if}\ i\in K_1\cup K_2,x_i=0\\
			&\{1\} && \mbox{otherwise}
			\end{aligned}\right.$ \\
			$\mathcal{D}=\{x\in \mathbb{R}^d\mid L\leq x\leq U\}$ &  $\partial \Pi_\mathcal{D}(x)= {\rm Diag}(u),\  u_i\in \left\{\begin{aligned}
			&\{1\} && \mbox{if}\ L_i<x_i<U_i \\
			&[0,1] && \mbox{if}\ x_i=L_i \mbox{ or }x_i=U_i \\[2pt]
			&\{0\} && \mbox{otherwise}
			\end{aligned}\right.$\\
			$\mathcal{D}=\{x\in \mathbb{R}^d\mid \|x\|_{\infty}\leq L\}$ & $\partial\Pi_\mathcal{D}(x)={\rm Diag}(u),\ u_i\in\left\{\begin{aligned}
			&\{1\} && \mbox{if}\ |x_i|<L \\
			&[0,1] && \mbox{if}\ |x_i|=L \\
			&\{0\} && \mbox{otherwise}
			\end{aligned}\right.$\\
			$\mathcal{D}=\{x\in \mathbb{R}^d\mid \|x\|_{2}\leq L\}$ & $\partial\Pi_\mathcal{D}(x)=\left\{\begin{aligned}
			&\{I_d\} && \mbox{if}\ \|x\|_2< L \\
			&\Big\{I_d-t\frac{xx^T}{L^2}\mid 0\leq t\leq 1\Big\} && \mbox{if}\ \|x\|_2= L\\
			&\Big\{\frac{L}{\|x\|_2}(I_d-\frac{xx^T}{\|x\|_2^2})\Big\} && \mbox{otherwise}
			\end{aligned}\right.$\\
			$\mathcal{D}=\{x\in \mathbb{R}^d\mid \|x\|_{1}\leq L\}$ & $H\in \partial \Pi_\mathcal{D}(x),\ \mbox{where } H=\left\{\begin{aligned}
			&I_d && \mbox{if}\ \|x\|_1\leq L \\
			&P_x \widetilde{H} P_x\tnote{*} && \mbox{otherwise}
			\end{aligned}\right.$ \\
			\noalign{\smallskip}\hline
		\end{tabular}
		\begin{tablenotes}\footnotesize
			\item[*] $\widetilde{H}={\rm Diag}(r)-\frac{1}{{\rm nnz}(r)}rr^T\in \partial\Pi_{\Delta_d}(x)$, where $r\in \mathbb{R}^d$ is defined as $r_i=1$ if $\big(\Pi_{\Delta_d}(P_x x/L)\big)_i\neq 0 $, and $r_i=0$ otherwise.
		\end{tablenotes}
	\end{threeparttable}
\end{table}

From the formula of $\partial \Pi_{\cal D}(\cdot)$ in Table \ref{table_jacobian}, we could see that $\partial {\rm Prox}_p(\cdot)$ is a nonempty, compact valued and upper-semicontinuous multifunction. We prove the strong semismoothness of ${\rm Prox}_{p} (\cdot)$ with respect to $\partial{\rm Prox}_{p} (\cdot)$ in the following proposition.
\begin{proposition}\label{prop: semismooth_p}
	For the closed convex set $\mathcal{D}$ defined in Proposition \ref{prop:defD}, ${\rm Prox}_{p} (\cdot)$ is strongly semismooth with respect to $\partial{\rm Prox}_{p} (\cdot)$.
\end{proposition}
\begin{proof}
	By the formula of $p(\cdot)$ and the definition of strong semismoothness, it suffices to prove that for each choice of $\mathcal{D}$ defined in Proposition \ref{prop:defD}, $\Pi_{\cal D}(\cdot)$ is strongly semismooth with respect to the corresponding Clarke generalized Jacobian $\partial \Pi_{\cal D}(\cdot)$ defined in Table \ref{table_jacobian}.
	
	For the case when $\mathcal{D}=\{x\in \mathbb{R}^d\mid x_{K_1}\geq 0,x_{K_2}\leq 0\}$, $\mathcal{D}=\{x\in \mathbb{R}^d\mid L\leq x\leq U\}$ or $\mathcal{D}=\{x\in \mathbb{R}^d\mid \|x\|_{\infty}\leq L\}$, we can see that $\Pi_{\cal D}(\cdot)$ is a Lipschitz continuous piecewise affine function, and thus $\Pi_{\cal D}(\cdot)$ is strongly semismooth everywhere with respect to the corresponding Clarke generalized Jacobian $\partial \Pi_{\cal D}(\cdot)$ defined in Table \ref{table_jacobian} due to \cite[Proposition 7.4.7]{facchinei2007finite}. For the case when $\mathcal{D}=\{x\in \mathbb{R}^d\mid \|x\|_{2}\leq L\}$, the strong semismoothness of $\Pi_{\cal D}(\cdot)$ with respect to $\partial \Pi_{\cal D}(\cdot)$ follows from the fact that the projection onto the second order cone is strongly semismooth \cite[Proposition 4.3]{chen2003complementarity}. When $\mathcal{D}=\{x\in \mathbb{R}^d\mid \|x\|_{1}\leq L\}$, $\Pi_{\cal D}(\cdot)$ is strongly semismooth with respect to the corresponding $\partial \Pi_{\cal D}(\cdot)$ in Table \ref{table_jacobian}, which is the so-called HS-Jacobian \cite{han1997newton,li2020efficient}. \qed
\end{proof}

\subsection{Finding a computable element in $\hat{\partial}^2\Phi(\theta,\xi)$}
\label{subsec:hessian}
As already mentioned, the most difficult part of the {\tt proxALM} is in solving the Newton system \eqref{eq: cg-system}. For efficient practical implementation, we need to find an efficiently computable element in $\hat{\partial}^2\Phi(\theta,\xi)$ for any given $(\theta,\xi)\in\mathbb{R}^n\times \mathbb{R}^{dn}$. From the definition of $\hat{\partial}^2\Phi(\theta,\xi)$, we can rewrite it as
\begin{align*}
\hat{\partial}^2\Phi(\theta,\xi) = \mathcal{M}_1(\theta,\xi)+\mathcal{M}_2(\xi),
\end{align*}
where
\begin{align*}
\mathcal{M}_1(\theta,\xi) &= \sigma \left(\begin{aligned}
A^T\\
B^T
\end{aligned}\right)\Big(I_{n^2}-\partial \Pi_{+}(A\theta+B\xi-\frac{\tilde{u}}{\sigma})\Big)\left(\begin{aligned}
&A & B
\end{aligned}\right),\\
\mathcal{M}_2(\xi) &= \left(\begin{aligned}
I_n+&\frac{1}{\sigma}H_1\\
&0
\end{aligned}\right.
\left.\begin{aligned}
&0\\
\sigma \Big[ I_{dn}-\partial{\rm Prox}_p&(\xi-\frac{\tilde{v}}{\sigma})\Big]+\frac{1}{\sigma}H_2
\end{aligned}\right).
\end{align*}

Based on our discussion on $\partial {\rm Prox}_p(\cdot)$ in \eqref{eq:pfunction} and $\partial \Pi_\mathcal{D}(\cdot)$ in Table \ref{table_jacobian}, we can see that the elements in $\partial {\rm Prox}_p(\cdot)$ are block diagonal matrices. In order to maintain the block diagonal structure, we choose $H_1$ and $H_2$ to be diagonal matrices, and hence the elements in $\mathcal{M}_2(\xi)$ for any $\xi\in \mathbb{R}^{dn}$ will also be block diagonal matrices. One can easily pick an element in $\mathcal{M}_2(\xi)$ by choosing an element in $\partial{\rm Prox}_p(\xi-\tilde{v}/\sigma)$.

For $\mathcal{M}_1(\theta,\xi)$, we choose an element ${\rm Diag}(w)$ in $\partial \Pi_{+}(A\theta+B\xi-\tilde{u}/\sigma)$, where
\begin{align*}
w_i = \left\{\begin{aligned}
&1 && \mbox{if}\ (A\theta+B\xi-\frac{\tilde{u}}{\sigma})_{i}\geq 0 \\
&0 && \mbox{otherwise}
\end{aligned}\right.,\quad i = 1,\cdots,n^2.
\end{align*}
By denoting $\bar{w}\in \mathbb{R}^{n^2}$ as $\bar{w}_i = 1-w_i$ for $i=1,\cdots,n^2$, we have
\begin{align*}
M=\sigma \left(\begin{aligned}
A^T\\
B^T
\end{aligned}\right){\rm Diag}(\bar{w})\left(\begin{aligned}
&A & B
\end{aligned}\right)=\sigma\left(\begin{aligned}
&A^T {\rm Diag}(\bar{w})A && A^T {\rm Diag}(\bar{w})B\\
&B^T {\rm Diag}(\bar{w})A && B^T {\rm Diag}(\bar{w})B
\end{aligned}\right)
\end{align*}
is an element in $\mathcal{M}_1(\theta,\xi)$. After some algebraic manipulations by making use of the structure of $A$ and $B$, we can get the following results:
\begin{align*}
A^T {\rm Diag}(\bar{w})A &= {\rm Diag}\left(\sum_{i=1}^n \bar{w}_{(i)}\right)+{\rm Diag}\left(\bar{w}_{(1)}^T\bar{w}_{(1)},\cdots,\bar{w}_{(n)}^T\bar{w}_{(n)}\right)\\
&\quad -(\bar{w}_{(1)},\cdots,\bar{w}_{(n)})-(\bar{w}_{(1)}^T;\cdots;\bar{w}_{(n)}^T)\in \mathbb{R}^{n\times n},\\
A^T {\rm Diag}(\bar{w})B&=\left({\rm Diag}(\bar{w}_{(1)})B_1,\cdots,{\rm Diag}(\bar{w}_{(n)})B_n \right)\\
&\quad -{\rm Diag}\left(\bar{w}_{(1)}^T B_1,\cdots,\bar{w}_{(n)}^T B_n\right)\in \mathbb{R}^{n\times dn},\\
B^T {\rm Diag}(\bar{w})B &={\rm Diag}\left(B_1^T {\rm Diag}(\bar{w}_{(1)})B_1,\cdots,B_n^T {\rm Diag}(\bar{w}_{(n)})B_n\right)\in\mathbb{R}^{dn\times dn},
\end{align*}
where $\bar{w}_{(i)}:=\bar{w}_{(i-1)n+1:in}\in \mathbb{R}^n$. It can be seen that the $0$-$1$ structure of $\bar{w}$ will reduce many operations in matrix-matrix multiplications, and hence highly reduce the computational cost for computing $M$ or matrix-vector products with $M$. Note that for all $i$, ${\rm Diag}(\bar{w}_{(i)})B_i$ is a matrix in $\mathbb{R}^{n\times d}$, with its $j$-th row being the $j$-th row of $B_i$ if $(\bar{w}_{(i)})_j=1$, or the zero vector if $(\bar{w}_{(i)})_j=0$. Then the computation of $\bar{w}_{(i)}^T B_i$ can be obtained by summing the non-zero rows of ${\rm Diag}(\bar{w}_{(i)})B_i$, and the computation of $B_i^T {\rm Diag}(\bar{w}_{(i)})B_i=({\rm Diag}(\bar{w}_{(i)})B_i)^T({\rm Diag}(\bar{w}_{(i)})B_i)$ can be highly reduced in the same way.

The special structure of the elements in $\hat{\partial}^2\Phi(\theta,\xi)$, which we call as the second-order sparsity, makes it possible for us to apply the {\tt SSN} based {\tt proxALM} algorithm to solve the huge QP problem \eqref{reformulated_problem} that contains $(d+1)n$ variables, $n(n-1)$ linear inequality constraints and $n$ possibly non-polyhedral constraints.

\section{A constraint generation method to accelerate the computation}
\label{sec:adaptive sieving}

Due to the existence of $n(n-1)$ linear inequality constraints, the problem \eqref{reformulated_problem} is quite difficult to solve for the case when the number of observations $n$ is huge. This naturally motivated us to consider a constraint generation method to avoid handling the full set of constraints when solving the problem. In this section, we design a practical implementation of the constraint generation method for solving the problem \eqref{reformulated_problem} with large $n$, where each reduced problem is solved by the proposed {\tt proxALM}.

The basic idea of the constraint generation method is to start solving the constrained QP with a subset of constraints, then add the most violated constraints (or part of violated constraints) to form a new reduced problem until the optimality conditions are satisfied. In our implementation, there are three points that we should emphasize. First, we add a relatively large number of most violated constraints in each round to highly reduce the number of rounds needed for the constraint generation method to terminate. Second, we apply our proposed {\tt proxALM} to solve each reduced problem to high accuracy, which is demonstrated to be quite efficient, especially for large-scale problems. Third, we divide the $O(n^2)$ constraints into blocks and check the optimality conditions block-wise to cope with the memory demand.	

Suppose that $(\theta^*,\xi^*,u^*,v^*)\in \mathbb{R}^n\times \mathbb{R}^{dn}\times \mathbb{R}^{n^2}\times \mathbb{R}^{dn}$ is a KKT solution of the problems \eqref{reformulated_problem} and \eqref{D}. Note that in the problem \eqref{reformulated_problem}, the condition
\begin{align*}
A\theta+B\xi\geq 0
\end{align*}
imposes $n^2$ linear inequality constraints on $(d+1)n$ variables. For the case when $n\gg d$, no more than $n(d+1)$ independent constraints would be active at $(\theta^*,\xi^*)$. That is to say, there exists an index set $I^*\subset \{1,2,\cdots,n^2\}$ with $|I^*|\leq n(d+1)$ such that
\begin{align*}
(A\theta^*+B\xi^*)_{I^*}= 0,\quad (A\theta^*+B\xi^*)_{\bar{I}^*}\geq 0,
\end{align*}
where $\bar{I}^*$ denotes the complement of $I^*$ in $\{1,2,\cdots,n^2\}$. The small proportion of active constraints inspires us to apply the idea of the constraint generation as an acceleration technique to solve the problems with large $n$.

Given an index set $I\subset \{1,2,\cdots,n^2\}$, we consider a variant of the problem \eqref{reformulated_problem} as
\begin{align}
\min_{\theta\in \mathbb{R}^n,\xi\in \mathbb{R}^{dn}} \ \Big\{\frac{1}{2}\|\theta-Y\|^2+p(\xi)+\delta_{+}(A_{I}\theta+B_{I}\xi)\Big\},\label{eq: newP}
\end{align}
where $A_{I}$ denotes the matrix consisting of the rows of $A$ indexed by $I$. The corresponding dual problem is
\begin{equation}\label{eq: newD}
\begin{aligned}
&\max_{u\in\mathbb{R}^{n^2},v\in \mathbb{R}^{dn}}\ \Big\{-\frac{1}{2}\|A^Tu\|^2-\langle Y,A^Tu\rangle-p^*(-v)-\delta_{+}(u)\Big\}\\
& \quad {\rm s.t.} \quad B^Tu+v=0,\quad u_{\bar{I}} =0.
\end{aligned}
\end{equation}
The KKT system associated with the problems \eqref{eq: newP} and \eqref{eq: newD} is
\begin{equation}\label{eq: reduced_KKT}
\begin{aligned}
&\theta-Y-A^T u=0,\quad B^T u+v=0,\quad u_{\bar{I}} = 0,\\
&-v\in\partial p(\xi),\quad -u_{I}\in \partial \delta_{+}(A_{I}\theta+B_{I}\xi).
\end{aligned}
\end{equation}
Suppose that $(\bar{\theta},\bar{\xi},\bar{u},\bar{v})\in \mathbb{R}^n\times \mathbb{R}^{dn}\times \mathbb{R}^{n^2}\times \mathbb{R}^{dn}$ satisfies the KKT system \eqref{eq: reduced_KKT}. We could see that $(\bar{\theta},\bar{\xi},\bar{u},\bar{v})$ naturally satisfies the KKT system \eqref{KKT_system} associated with the problems \eqref{reformulated_problem} and \eqref{D}, except for the following inequality
\begin{align*}
A_{\bar{I}}\theta+B_{\bar{I}}\xi \geq 0.
\end{align*}
Therefore, we add the indices in the index set $I':=\{i\in \bar{I}\mid A_{\{i\}}\bar{\theta}+B_{\{i\}}\bar{\xi}<0\}$ into $I$ to get a new variant of the problem \eqref{reformulated_problem} as stated in \eqref{eq: newP}, then repeat the procedure until the stopping criteria of the problems \eqref{reformulated_problem} and \eqref{D} are satisfied.

Note that in this paper, we use the relative KKT residual
\begin{align}
R_{\rm KKT}:=\max\Big\{&\frac{\|\theta-Y-A^Tu\|}{1+\|Y\|+\|\theta\|+\|u\|}, \frac{\|B^Tu+v\|}{1+\|u\|+\|v\|},\frac{\|\xi-{\rm Prox}_p(\xi-v)\|}{1+\|\xi\|+\|v\|},\notag \\
&\ \frac{\|A\theta+B\xi-\Pi_{+}(A\theta+B\xi-u)\|}{1+\|A\theta\|+\|B\xi\|+\|u\|}\Big\},\label{eq: Rkkt}
\end{align}
to measure the accuracy of an approximate optimal solution $(\theta,\xi,u,v)$ to the KKT system \eqref{KKT_system}. In addition, given an index set $I\subset \{1,2,\cdots,n^2\}$, we define
\begin{align*}
R_{\rm KKT}^I = \max\Big\{&\frac{\|\theta-Y-A_{I}^Tu_I\|}{1+\|Y\|+\|\theta\|+\|u_I\|}, \frac{\|B_I^Tu_I+v\|}{1+\|u_I\|+\|v\|},\frac{\|\xi-{\rm Prox}_p(\xi-v)\|}{1+\|\xi\|+\|v\|},\\
&\ \frac{\|A_I\theta+B_I\xi-\Pi_{+}(A_I\theta+B_I\xi-u_I)\|}{1+\|A_I\theta\|+\|B_I\xi\|+\|u_I\|}\Big\}.
\end{align*}

Next we present our practical implementation of the constraint generation method for solving the problem \eqref{reformulated_problem} in Algorithm \ref{alg:as}, where we apply our proposed {\tt proxALM} to solve each of the reduced problems.

\begin{algorithm}[h]
	\caption{: Constraint generation method for \eqref{reformulated_problem}}
	\label{alg:as}
	\begin{algorithmic}[1]
		\STATE {\bfseries Initialization:} Given a tolerance $\epsilon>0$ and an initial index set $I^0\subset \{1,2,\cdots,n^2\}$, solve the problem
		\begin{align}
		\min_{\theta\in \mathbb{R}^n,\xi\in \mathbb{R}^{dn}} \ \Big\{\frac{1}{2}\|\theta-Y\|^2+p(\xi)+\delta_{+}(A_{I_0}\theta+B_{I_0}\xi)\Big\}\tag{$P_{I^0}$}\label{eq: newP0}
		\end{align}
		to get an approximate KKT solution $(\theta^0,\xi^0,u^0,v^0)\in \mathbb{R}^n\times \mathbb{R}^{dn}\times \mathbb{R}^{n^2}\times \mathbb{R}^{dn}$ such that $u^0_{\bar{I}^0}=0$ and
		$R_{\rm KKT}^{I_0}\leq \epsilon$. Compute $R_{\rm KKT}$ and set $k=1$.
		\REPEAT
		\STATE {\bfseries Step 1}. Let
		\begin{align*}
		S^{k}:=\{j\in \bar{I}^{k-1}\mid A_{\{j\}}\theta^{k-1}+B_{\{j\}}\xi^{k-1} <0\}.
		\end{align*}
		If $|S^{k}|>|I^{k-1}|$, set
		\begin{equation*}
		I^{k} = I^{k-1}\cup \left\{j\in \bar{I}^{k-1}\left|
		\begin{aligned}
		&A_{\{j\}}\theta^{k-1}+B_{\{j\}}\xi^{k-1} \mbox{ is among the first $|I^{k-1}|$}\\
		&\mbox{ smallest values in }A_{S^{k}}\theta^{k-1}+B_{S^{k}}\xi^{k-1}
		\end{aligned}
		\right\}\right.;
		\end{equation*}
		and otherwise, set $I^{k}=I^{k-1}\cup S^{k}$.
		\\[3pt]
		\STATE {\bfseries Step 2}. Solve the problem
		\begin{align}
		\min_{\theta\in \mathbb{R}^n,\xi\in \mathbb{R}^{dn}} \ \Big\{\frac{1}{2}\|\theta-Y\|^2+p(\xi)+\delta_{+}(A_{I_k}\theta+B_{I_k}\xi)\Big\}\tag{$P_{I^k}$}\label{eq:newPk}
		\end{align}
		to get an approximate KKT solution $(\theta^{k},\xi^{k},u^{k},v^{k})\in \mathbb{R}^n\times \mathbb{R}^{dn}\times \mathbb{R}^{n^2}\times \mathbb{R}^{dn}$ such that
		$u^k_{\bar{I}^k}=0$ and
		$R_{\rm KKT}^{I_k}\leq \epsilon$.
		\\[3pt]
		\STATE {\bfseries Step 3}. Compute $R_{\rm KKT}$ and
		set $k\leftarrow k+1$.
		\UNTIL{Stopping criteria $R_{\rm KKT}\leq \epsilon$ is satisfied.}
	\end{algorithmic}
\end{algorithm}

\begin{remark}\label{remark: warmstart}
	As a side note, in the $k$th iteration of Algorithm \ref{alg:as}, we apply a warm start technique by setting the initialization as the solution  obtained in the $(k-1)$th iteration.
\end{remark}

The convergence property of Algorithm \ref{alg:as} is presented in the following theorem.
\begin{theorem}\label{thm:as}
	For any given tolerance $\epsilon$ and initial index set $I^0\subset \{1,\cdots,n^2\}$, Algorithm \ref{alg:as} will terminate after a finite number of rounds.
\end{theorem}
\begin{proof}
	We first prove that if $S^{k+1}=\emptyset$, the corresponding $(\theta^{k},\xi^{k},u^{k},v^{k})\in \mathbb{R}^n\times \mathbb{R}^{dn}\times \mathbb{R}^{n^2}\times \mathbb{R}^{dn}$ satisfies $R_{\rm KKT}\leq \epsilon$. Suppose $S^{k+1}=\emptyset$, then we have
	\begin{align*}
	A_{\bar{I}^{k}}\theta^{k}+B_{\bar{I}^{k}}\xi^{k}\geq 0.
	\end{align*}
	Together with $u_{\bar{I}^{k}}^{k}=0$, we know that
	\begin{align*}
	\frac{\|\theta^{k}-Y-A^Tu^{k}\|}{1+\|Y\|+\|\theta^{k}\|+\|u^{k}\|}&=\frac{\|\theta^{k}-Y-A_{I^{k}}^Tu_{I^{k}}^{k}\|}{1+\|Y\|+\|\theta^{k}\|+\|u^{k}\|},\\ \frac{\|B^Tu^{k}+v^{k}\|}{1+\|u^{k}\|+\|v^{k}\|}&=\frac{\|B_{I^{k}}^Tu_{I^{k}}^{k}+v^{k}\|}{1+\|u_{I^{k}}^{k}\|+\|v^{k}\|},
	\end{align*}
	and
	\begin{align*}
	&\frac{\|A\theta^{k}+B\xi^{k}-\Pi_{+}(A\theta^{k}+B\xi^{k}-u^{k})\|}{1+\|A\theta^{k}\|+\|B\xi^{k}\|+\|u^{k}\|}\\
	&\qquad \leq \frac{\|A_{I^{k}}\theta^{k}+B_{I^{k}}\xi^{k}-\Pi_{+}(A_{I^{k}}\theta^{k}+B_{I^{k}}\xi^{k}-u_{I^{k}}^{k})\|}{1+\|A_{I^{k}}\theta^{k}\|+\|B_{I^{k}}\xi^{k}\|+\|u_{I^{k}}^{k}\|}.
	\end{align*}
	Combining with the fact that $(\theta^{k},\xi^{k},u^{k},v^{k})$ satisfies $R_{\rm KKT}^{I_k}\leq \epsilon$, we have the corresponding relative KKT residual $R_{\rm KKT}\leq R_{\rm KKT}^{I_k}\leq \epsilon$. As a result, if $R_{\rm KKT}>\epsilon$, we have $S^{k+1}\neq \emptyset$, which means that new constraints will be added to construct a new reduced primal problem. Since the total number of the constraints in the primal problem \eqref{reformulated_problem} is finite, our algorithm will terminate after a finite number of rounds. \qed
\end{proof}

Note that in the algorithm, we add a relatively large number of violated constraints instead of adding $n$ violated constraints in each round as done in \cite{chen2020multivariate,bertsimas2021sparse}. The reason is that we have a highly efficient {\tt proxALM} algorithm which can solve each reduced problem \eqref{eq:newPk} with a relatively large number of constraints. The superior performance of this acceleration technique will be demonstrated in the numerical experiments.

\section{Numerical experiments}
\label{sec:numerical}
In this section, we conduct some numerical experiments\footnote{The code is available at {\tt https://doi.org/10.5281/zenodo.5543733}.}  to demonstrate the performance of the {\tt proxALM} for solving \eqref{reformulated_problem}, under each case of $\mathcal{D}$ mentioned in Proposition \ref{prop:defD}, as well as the performance of the constraint generation method for the acceleration. In addition, we design a data-driven Lipschitz estimation method to deal with the boundary effect of the convex regression problem. All our computational results are obtained by running MATLAB R2018b on a windows workstation (12-core, Intel Xeon E5-2680 @ 2.50GHz, 128G RAM).

\subsection{Computational performance of the {\tt proxALM} for solving \eqref{reformulated_problem}}
\label{sec:numericalcomparison}

In this subsection, we compare the performance of the {\tt proxALM}, the {\tt sGS-ADMM}, and {\tt MOSEK} for different choices of $d$ and $n$. In the experiments, we stop the algorithm when $R_{\rm KKT}\leq 10^{-4}$, where $R_{\rm KKT}$ is defined in \eqref{eq: Rkkt}. In Algorithm {\tt proxALM}, we choose $H_1=10^{-3}I_{n}$, $H_2=10^{-3}I_{dn}$, and use the stopping criteria \eqref{B} in \textbf{Step 1} with $\delta_k=\max\{0.1,10^{-6}/\|(\theta^{k+1},\xi^{k+1},u^{k+1},v^{k+1})-(\theta^{k},\xi^{k},u^{k},v^{k}) \|_{\Sigma}\}/(\left \lceil{k/20}\right \rceil^2) $. Here, the {\tt sGS-ADMM} is a symmetric Gauss-Seidel based multi-block {\tt ADMM}, which is proved to be convergent and has been demonstrated to perform better than the possibly nonconvergent directly extended multi-block {\tt ADMM} \cite{chen2017efficient}. The detailed description of the {\tt sGS-ADMM} could be found in Appendix \ref{appendix:sGSADMM}. As we can see in \cite{aybat2014parallel}, as long as there is enough memory, {\tt MOSEK} can perform quite a lot better than the parallel {\tt APG} method. Since there is enough memory on our workstation, we just compare our proposed {\tt proxALM} with the state-of-the-art algorithms {\tt MOSEK} and {\tt sGS-ADMM}.

For a given convex function $\psi:\mathbb{R}^d\rightarrow \mathbb{R}$, the synthetic dataset is generated via the procedure in \cite{mazumder2019computational}. We first generate $n$ samples $X_i\in \mathbb{R}^d$, $i=1,\cdots,n$ uniformly from $[-1,1]^d$, then the corresponding responses are given as $Y_i=\psi(X_i)+\epsilon_i$. The error vector $\epsilon$ follows the normal distribution $\mathcal{N}(0,\sigma^2 I_n)$, where $\sigma^2 = {\rm Var}(\{\psi(X_i)\}_{i=1}^n)/{\rm SNR}$. In the experiments, we take ${\rm SNR}=3$. Before we run the algorithms for the data $X=(X_1,\cdots,X_n)\in \mathbb{R}^{d\times n}$ and $Y\in\mathbb{R}^n$, we process the data so as to build a more predictive model. For the response $Y$ and each row of the predictor $X$, we mean-center the vector and then standardize it to have unit $\ell_2$-norm.

\begin{figure}[H]
	\centering
	\includegraphics[width=1\linewidth]{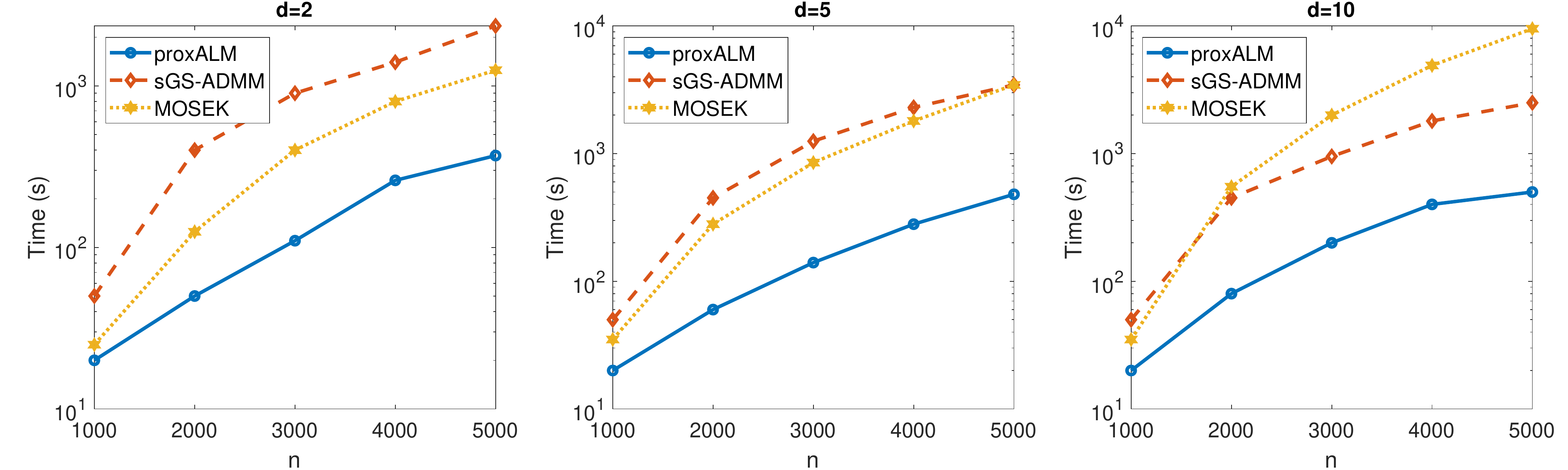}
	\caption{Convex regression for test function $\psi(x)=\exp(p^T x)$, where $p$ is a given random vector with each coordinate drawn from the standard normal distribution}\label{fig: org}
\end{figure}

\begin{figure}[H]
	\centering
	\includegraphics[width=1\linewidth]{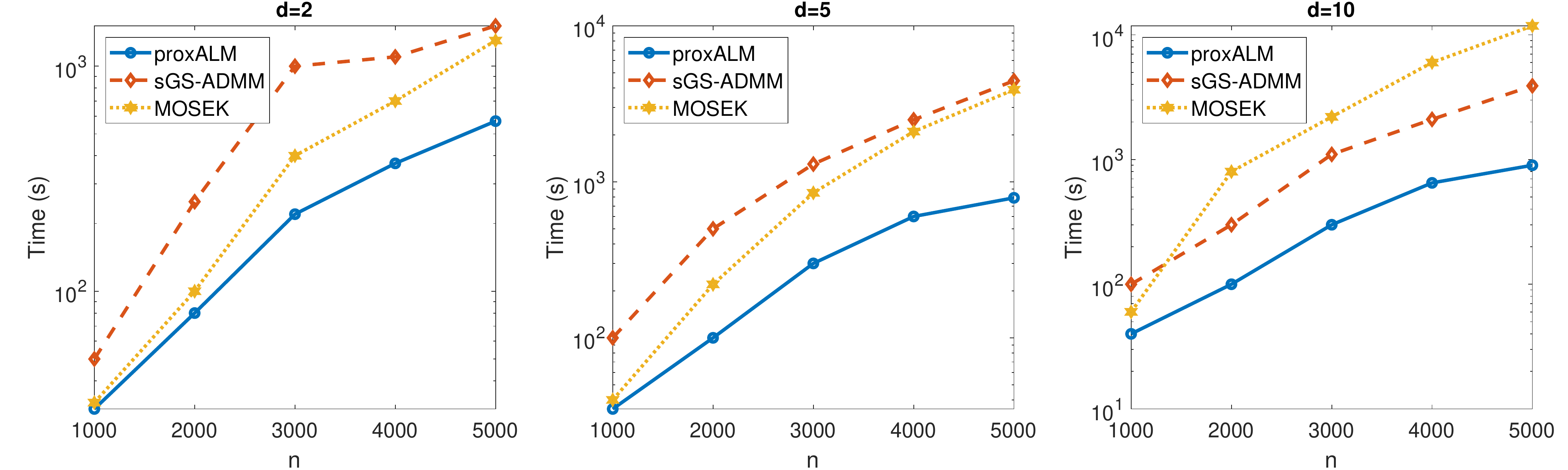}
	\caption{Convex regression with monotone constraint (non-decreasing) for the test function $\psi(x)=(e_d^T x)_{+}$}\label{fig: monotone}
\end{figure}

\begin{figure}[H]
	\centering
	\includegraphics[width=1\linewidth]{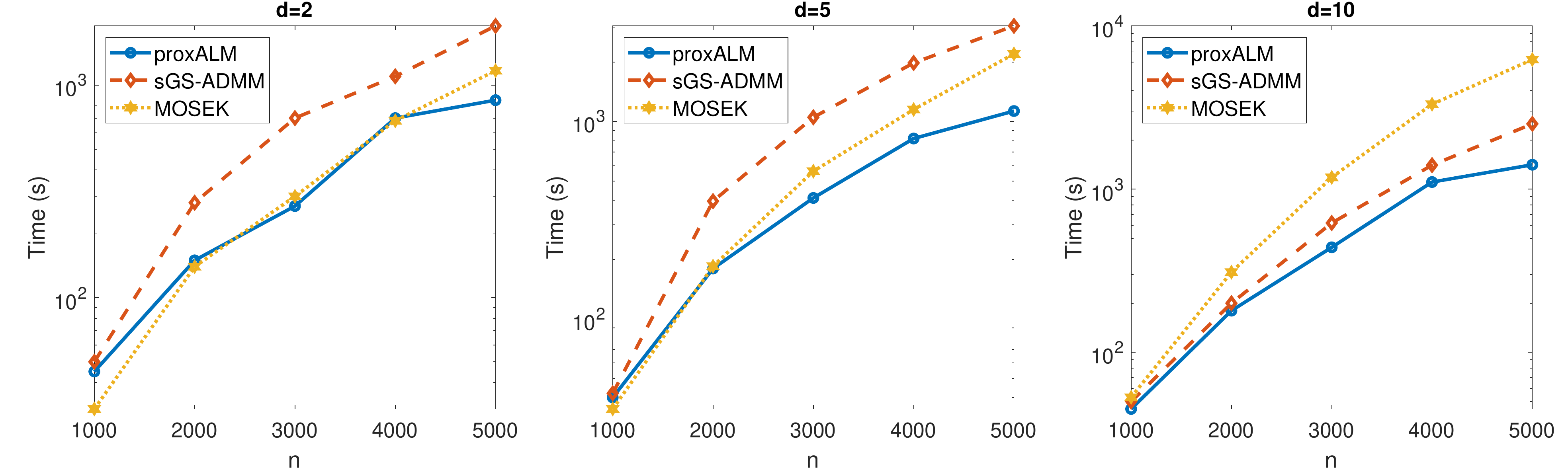}
	\caption{Convex regression with box constraint ($L=0_d$, $U=e_d$) for the test function $\psi(x)=\ln(1+\exp(e_d^Tx))$}\label{fig: twobound}
\end{figure}

\begin{figure}[H]
	\centering
	\includegraphics[width=1\linewidth]{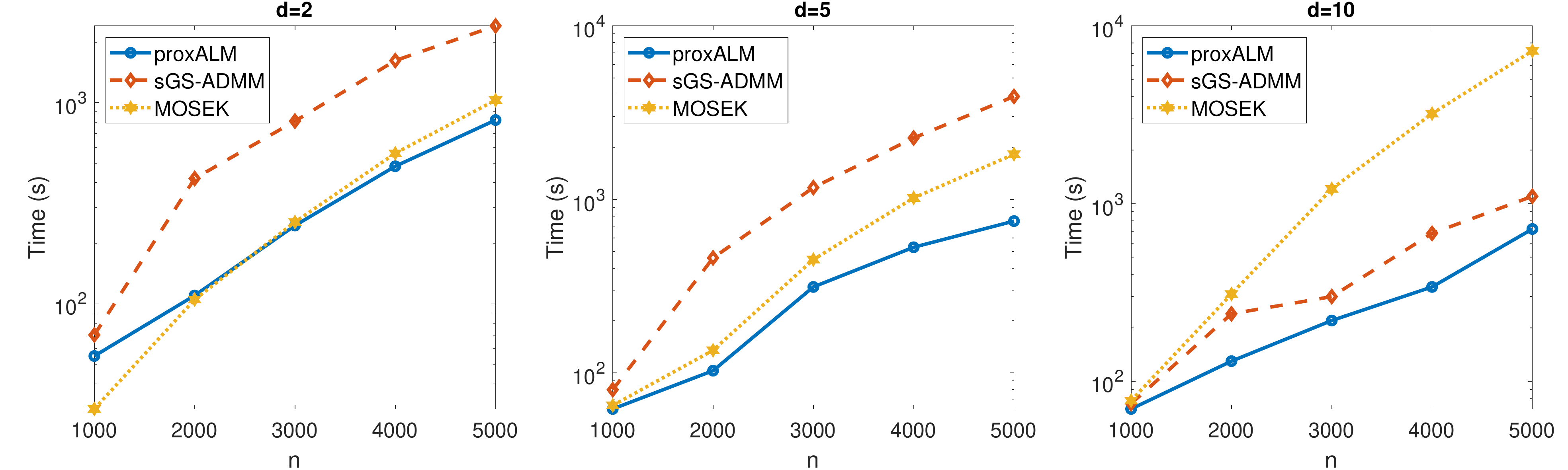}
	\caption{Convex regression with Lipschitz constraint ($p=1$, $q=\infty$, $L=1$) for the test function $\psi(x)=\sqrt{1+x^Tx}$}\label{fig: infnorm}
\end{figure}

\begin{figure}[H]
	\centering
	\includegraphics[width=1\linewidth]{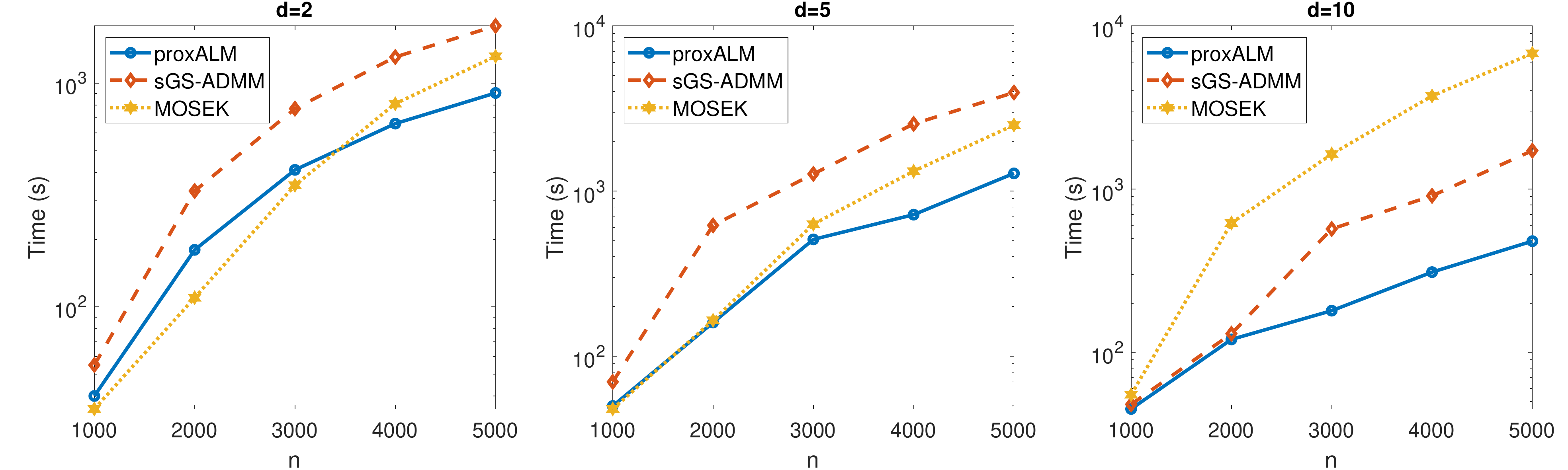}
	\caption{Convex regression with Lipschitz constraint ($p=2$, $q=2$, $L=\lambda_{\rm max}(Q)$) for the test function $\psi(x)=\sqrt{x^T Qx}$, where $Q\in \mathbb{R}^{d\times d}$ is a randomly generated symmetric and positive definite matrix with known largest eigenvalue}
	\label{fig: 2norm}
\end{figure}

\begin{figure}[H]
	\centering
	\includegraphics[width=1\linewidth]{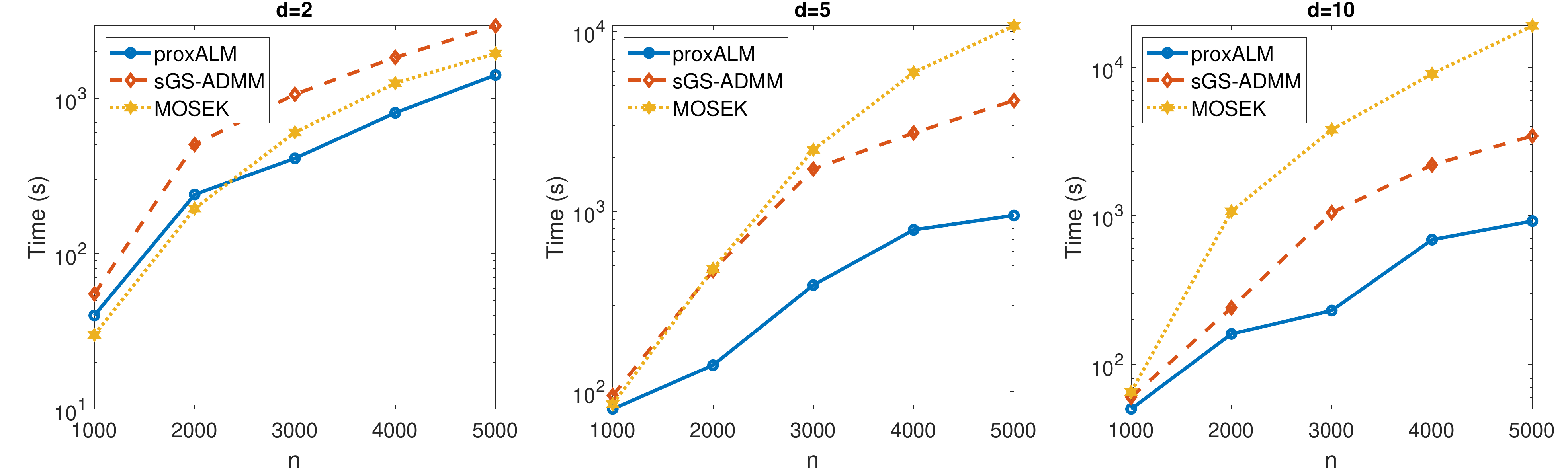}
	\caption{Convex regression with Lipschitz constraint ($p=\infty$, $q=1$, $L=1$) for the test function $\psi(x)=\ln(1+e^{x_1}+\cdots+e^{x_d})$}\label{fig: 1norm}
\end{figure}

The numerical results on the comparison among {\tt proxALM}, {\tt sGS-ADMM} and {\tt MOSEK} can be found in Figure \ref{fig: org} -- Figure \ref{fig: 1norm}. Note that we set the y-axes of all figures in log-scale to better show the functional dependence on $n$. We conduct experiments on the unconstrained convex regression problem and each case of shape-constrained convex regression we mentioned before, under different choices of $(d,n)$. All the test functions are convex on $\mathbb{R}^d$ and satisfy some specified shape constraints. As one can see from the figures, {\tt proxALM} outperforms the state-of-the-art solvers {\tt MOSEK} and {\tt sGS-ADMM} by a large margin, especially for large-scale cases. For example, for the convex regression with monotone constraint when $(d,n)=(5,5000)$, the {\tt proxALM} takes about $800$ seconds, while {\tt sGS-ADMM} and {\tt MOSEK} take around $4000$ seconds.

More numerical results of the comparison on instances with larger $d$ could be found in Appendix \ref{appendix: numerical}.

\subsection{Computational performance of the acceleration with the constraint generation method}
In this subsection, we mainly focus on the case when $n\gg d$. Consider the convex function $\psi(x)=5\|x\|_{\infty}+\|x\|^2$, we sample $n$ data points uniformly from $[-1,1]^d$ and add the Gaussian noise as stated in Section \ref{sec:numericalcomparison} with SNR$=10$. Figure \ref{time2} shows the time comparison among {\tt CGM+proxALM}, {\tt CGM+sGS-ADMM}, {\tt CGM+MOSEK}, {\tt proxALM}, {\tt sGS-ADMM}, {\tt MOSEK}, where {\tt CGM+} means the constraint generation method is used for the acceleration, to solve the convex regression problems with $d=2$. Note that in the {\tt CGM}, we take $|I^0|=10n$ and select the initial indices uniformly at random from the set $\{1,2,\cdots,n^2\}$. We stop each algorithm when $R_{\rm KKT}\leq 10^{-4}$.

\begin{figure}
	\centering
	\includegraphics[width=0.5\linewidth]{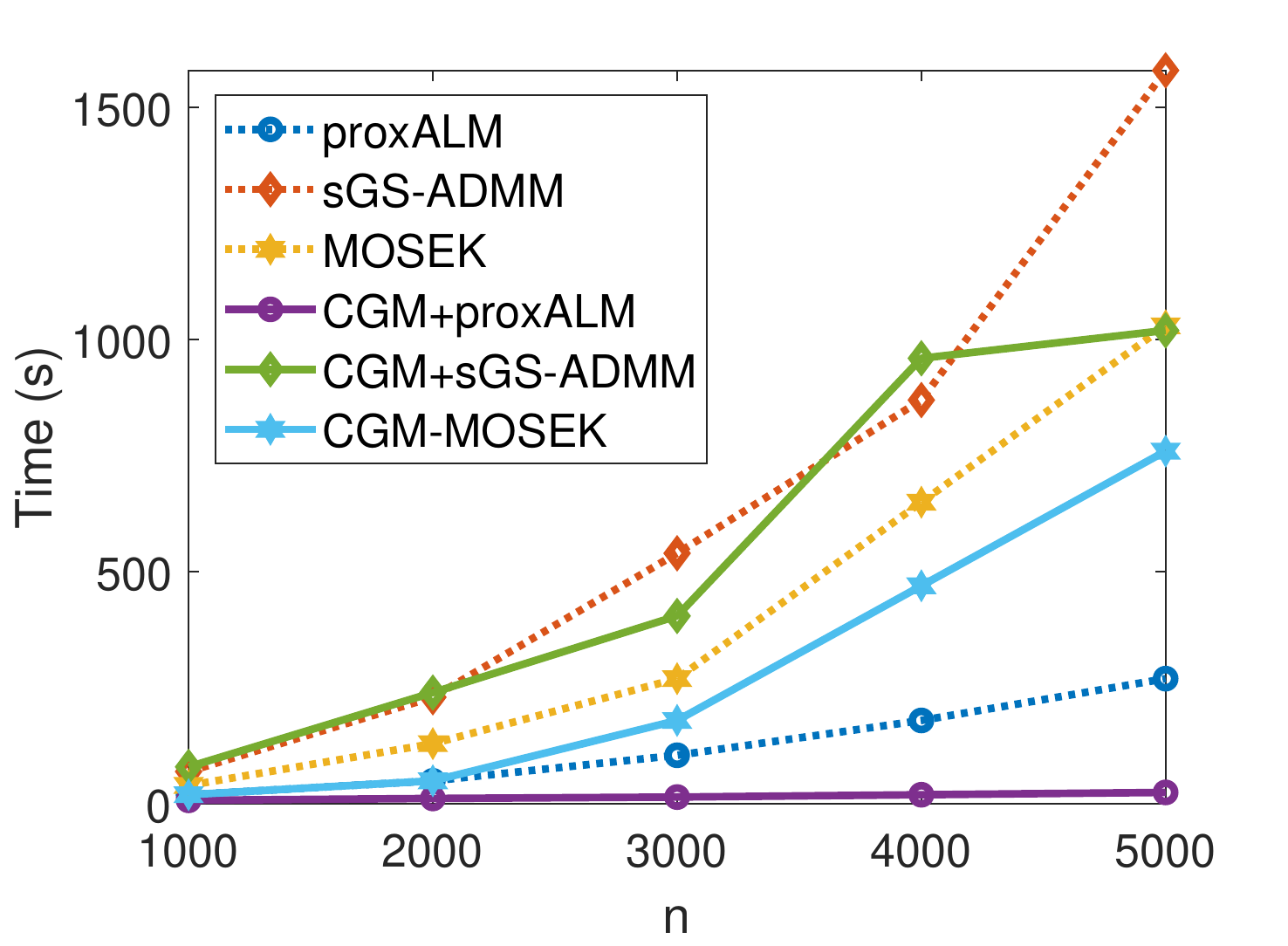}
	\caption{Computational time comparison among {\tt CGM+proxALM}, {\tt CGM+sGS-ADMM}, {\tt CGM+MOSEK}, {\tt proxALM}, {\tt sGS-ADMM}, {\tt MOSEK}}\label{time2}
\end{figure}

From the result, we can see that {\tt CGM+proxALM} outperforms all other algorithms by quite a large margin. For example, for the case $(d,n)=(2,5000)$, {\tt CGM+proxALM} takes $28$ seconds, {\tt proxALM} takes $288$ seconds, while the remaining four algorithms take around $1000$ seconds.

To further demonstrate the performance of the {\tt CGM} with the {\tt proxALM}, we conduct experiments on examples with higher dimensions and larger sample sizes. The results are shown in Table \ref{table_large}. In Algorithm {\tt CGM}, we set $|I^0|=50n$ for $d=2$, and $|I^0|=10n$ for $d=10,20$. In consideration of memory cost, we divide the $n^2$ constraints into ten parts when checking the optimality conditions \eqref{KKT_system} and when selecting the new indices in \textbf{Step 1} of Algorithm {\tt CGM}.

\begin{table}[h]
	\caption{Performance of {\tt CGM+proxALM} on convex regression on instances with large sample sizes. Time is divides into three parts: {\tt CGM} (constraint generation step), {\tt proxALM} (running time of the {\tt proxALM}), and OPT (checking optimality conditions)}
	\label{table_large}
	\tabcolsep 2.5pt
	\begin{threeparttable}
		\begin{tabular}{ccccccl}
			\hline\noalign{\smallskip}
			$(d,n)$
			& {\tt CGM} rounds & $R_{\rm KKT}$ &  $R_{\rm gap}$ & $\hat{R}_{\rm viotol}$ & $\hat{R}_{\rm pinfeas}$  & Time(s)(\,{\tt CGM}$|$\,{\tt proxALM}$|$\,OPT) \\
			\noalign{\smallskip}\hline\noalign{\smallskip}
			$(2,10000)$ & 3 & 1.39e-5 & 1.54e-4 & 1.88e-3 & 5.42e-6 & 36(\,2$|$\,31$|$\,3)  \\
			$(2,50000)$ & 3 & 9.25e-5 & 2.01e-3 & 1.87e-3 &
			1.61e-5 & 281(\,46$|$\,172$|$\,63)  \\
			$(2,100000)$ & 4 & 6.87e-5 & 5.11e-3 &  6.16e-4 &
			4.28e-6 & 1270(\,194$|$\,740$|$\,336)  \\
			$(10,10000)$ & 4 & 1.58e-5 & 1.19e-3 & 1.30e-2 &
			9.13e-6 & 27(\,3$|$\,17$|$\,7)  \\
			$(10,50000)$ & 4 & 8.87e-5 & 9.23e-3 & 1.71e-2 &
			2.24e-5 & 334(\,67$|$\,181$|$\,86)  \\
			$(10,100000)$ & 5 & 2.42e-5 & 9.91e-3 & 7.93e-3 &
			2.65e-6 & 1625(\,328$|$\,855$|$\,442)  \\
			$(20,10000)$ & 3 & 6.57e-5 & 1.89e-3 & 1.17e-2 &
			2.88e-5 & 45(\,3$|$\,37$|$\,5)  \\
			$(20,50000)$ & 4 & 8.04e-5 & 5.59e-3 & 5.14e-3 &
			1.57e-5 & 425(\,73$|$\,265$|$\,87)  \\
			$(20,100000)$ & 5  & 3.88e-5 & 7.90e-3 & 1.31e-3 &
			3.16e-7 & 1614(\,331$|$\,836$|$\,447) \\
			\noalign{\smallskip}\hline
		\end{tabular}
	\end{threeparttable}
\end{table}

Note that in the table, the number of {\tt CGM} rounds includes the initialization step, and $R_{\rm rel}$ is defined as
\begin{align*}
R_{\rm rel} = \frac{|{\rm pobj}-{\rm dobj}|}{1+|{\rm pobj}|+|{\rm dobj}|},
\end{align*}
where ${\rm pobj}$ and ${\rm dobj}$ denote the primal and dual objective function values. For better illustration, we also report the \textit{primal infeasibility} \cite{mazumder2019computational,bertsimas2021sparse} and the \textit{violation tolerance} \cite{bertsimas2021sparse} as
\begin{align*}
\hat{R}_{\rm pinfeas}= \frac{1}{n}\|(A\theta+B\xi)_{-}\|,\quad
\hat{R}_{\rm viotol}=\max|(A\theta+B\xi)_{-}|,
\end{align*}
respectively, where $x_{-}:=\min(x,0)$.

We can see from the table that the {\tt CGM} combined with the {\tt proxALM} performs quite well for estimating the convex regression functions with huge sample sizes. Note that in the table, time is divided into three parts: constraint generation step, running time of the {\tt proxALM} and checking optimality conditions. As the sample size of the instance increases, the time taken by the constraint generation step and checking optimality conditions increases rapidly due to the huge number of $n^2$ linear inequality constraints. For example, for the instance with size $(d,n)=(10,100000)$, we need to solve a constrained QP containing $1.1\times 10^6$ variables and $10^{10}$ linear inequality constraints. From the table we can see that checking the optimality conditions five times cost $442$ seconds while estimating the convex regression function with {\tt CGM+proxALM} only costs $1625$ seconds in total. The long computation time needed to check the optimality conditions for large $n$ is the reason why we choose to add more violated constraints in each round so as to reduce the number of rounds in the constraint generation method. As a comparison, we note that the implementation in \cite{bertsimas2021sparse} of the constraint generation method with each reduced problem solved by {\tt Gurobi} needs around $1$ hour and $11$ rounds of the constraint generation to solve the problem of the same size, but only achieves the accuracy $\hat{R}_{\rm viotol}=0.05$, $\hat{R}_{\rm pinfeas}=0.004$. The success of the proposed {\tt CGM} combined with the {\tt proxALM} lies in two aspects. First, the number of rounds of the constraint generation is highly reduced since we add a relatively large number of violated constraints in each round. Second, the {\tt proxALM} is quite efficient to solve each reduced problem in the {\tt CGM} compared to {\tt Gurobi} or {\tt MOSEK}.

\subsection{Data-driven Lipschitz estimation method}
\label{subsec:Lipschitz}
An important issue in convex regression is over-fitting near the boundary of ${\rm conv}(X_1,\cdots,X_n)$. That is, the norms of the fitted subgradients $\xi_i$'s near the boundary can become arbitrarily large. To deal with this problem, the authors in \cite{lim2014convergence,balazs2015near,mazumder2019computational} used the idea of Lipschitz convex regression. They propose to compute the least squares estimator over the class of convex functions that are uniformly Lipschitz with a given bound, which means that they compute the estimator defined in \eqref{eq:min_CS} with Property $\mathcal{S}$ taking the form of (S3). In practice, the challenge is in choosing the unknown Lipschitz constant in the model based on the given data. Mazumder et al. \cite{mazumder2019computational} choose to estimate the Lipschitz constant by the cross-validation. In this paper, we provide a data-driven Lipschitz estimation method for the Lipschitz convex regression.

For each $X_i$, we first find the $k$-nearest neighbors $\mathcal{N}(X_i)$ of $X_i$, and then define
\begin{align*}
L_i={\rm median}\Big\{\frac{|Y_i-Y_j|}{\|X_i-X_j\|_p},j\in \mathcal{N}(X_i)\Big\},
\end{align*}
where $p=1,2,\infty$ is given. After that, we solve the generalization form of \eqref{convex_LSE_D} as
\begin{equation}\label{generalization_convexLSE}
\begin{aligned}
&\min_{\theta_1,\ldots,\theta_n\in\mathbb{R};  \xi_1,\ldots,\xi_n \in\mathbb{R}^{d}}
\frac{1}{2} \sum_{i=1}^n (\theta_i - Y_i)^2 \\
&{\rm s.t.} \quad \theta_i \geq \theta_j + \langle \xi_j,X_i - X_j\rangle,\quad \ 1\leq i ,j \leq n,\\
&\qquad\ \xi_i\in \mathcal{D}_i,\quad i=1,\cdots,n,
\end{aligned}
\end{equation}
where $\mathcal{D}_i=\{x\in \mathbb{R}^d\mid \|x\|_q\leq L_i\}$ with $1/p+1/q=1$. The proposed {\tt proxALM} can be easily extended to solve \eqref{generalization_convexLSE} by letting $p(\xi)= \sum_{i=1}^n \delta_{\mathcal{D}_i} (\xi_i)$.

We use an example here to demonstrate the performance of Lipschitz convex regression with the data-driven Lipschitz estimation method. Consider the convex function $\psi(x)=2\|x\|_{\infty}+\|x\|^2$. We sample $n=80$ data points uniformly from $[-1,1]^d$ and add the Gaussian noise as stated in Section \ref{sec:numericalcomparison}. The results for $d=1,2$ can be seen in Figure \ref{fig:data_Lip}. When estimating the Lipschitz constant for each data point, we take $k=5$ and $p=q=2$. As shown in the figure, Lipschitz convex regression does reduce the estimation error near the boundary of the convex hull of $X_i$'s.

\begin{figure}[ht]
	\subfigure[$d=1$]{\label{fig:toy1}
		\includegraphics[width=0.45\linewidth]{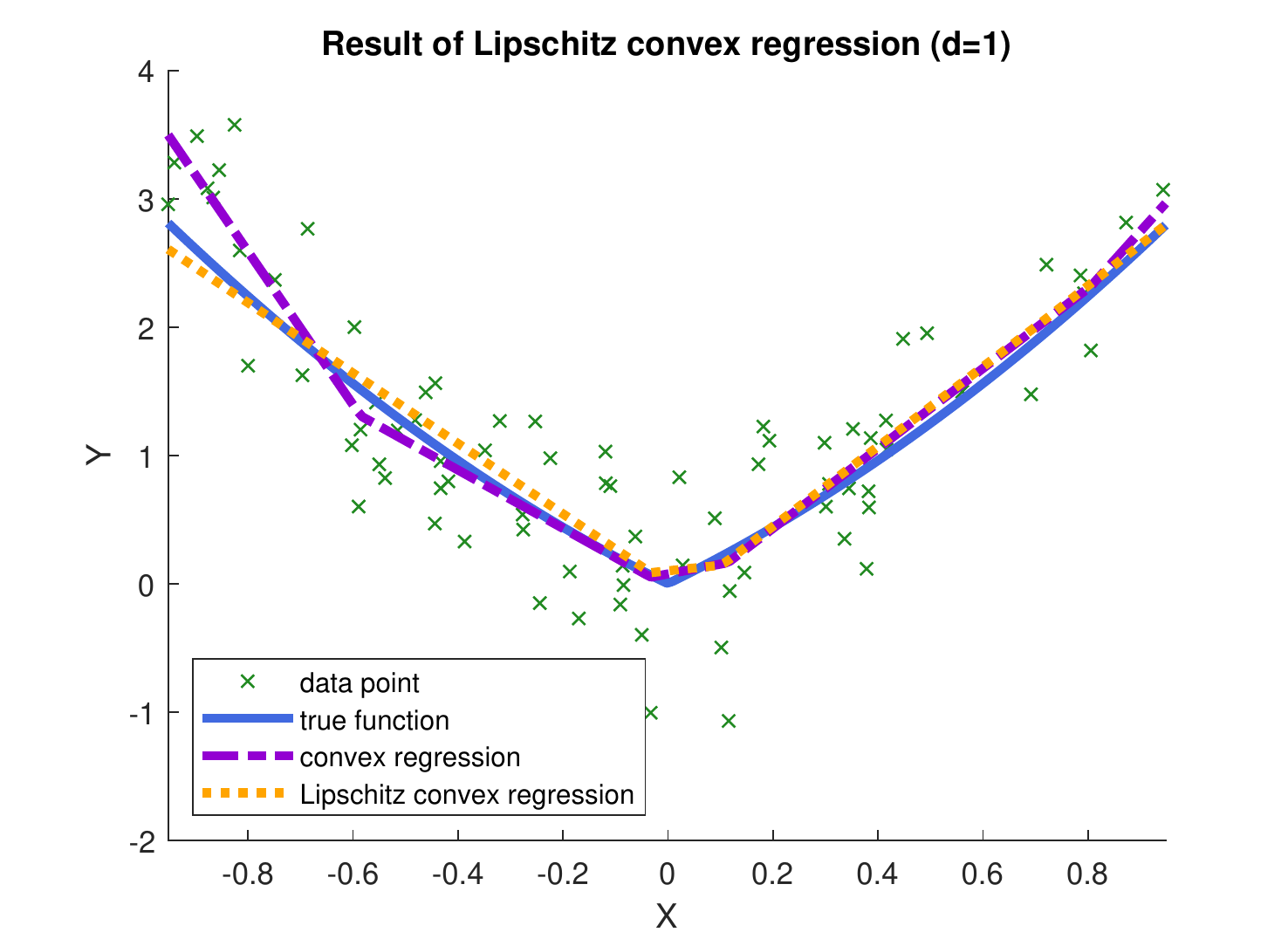}}
	\subfigure[$d=2$]{\label{fig:toy2}
		\includegraphics[width=0.45\linewidth]{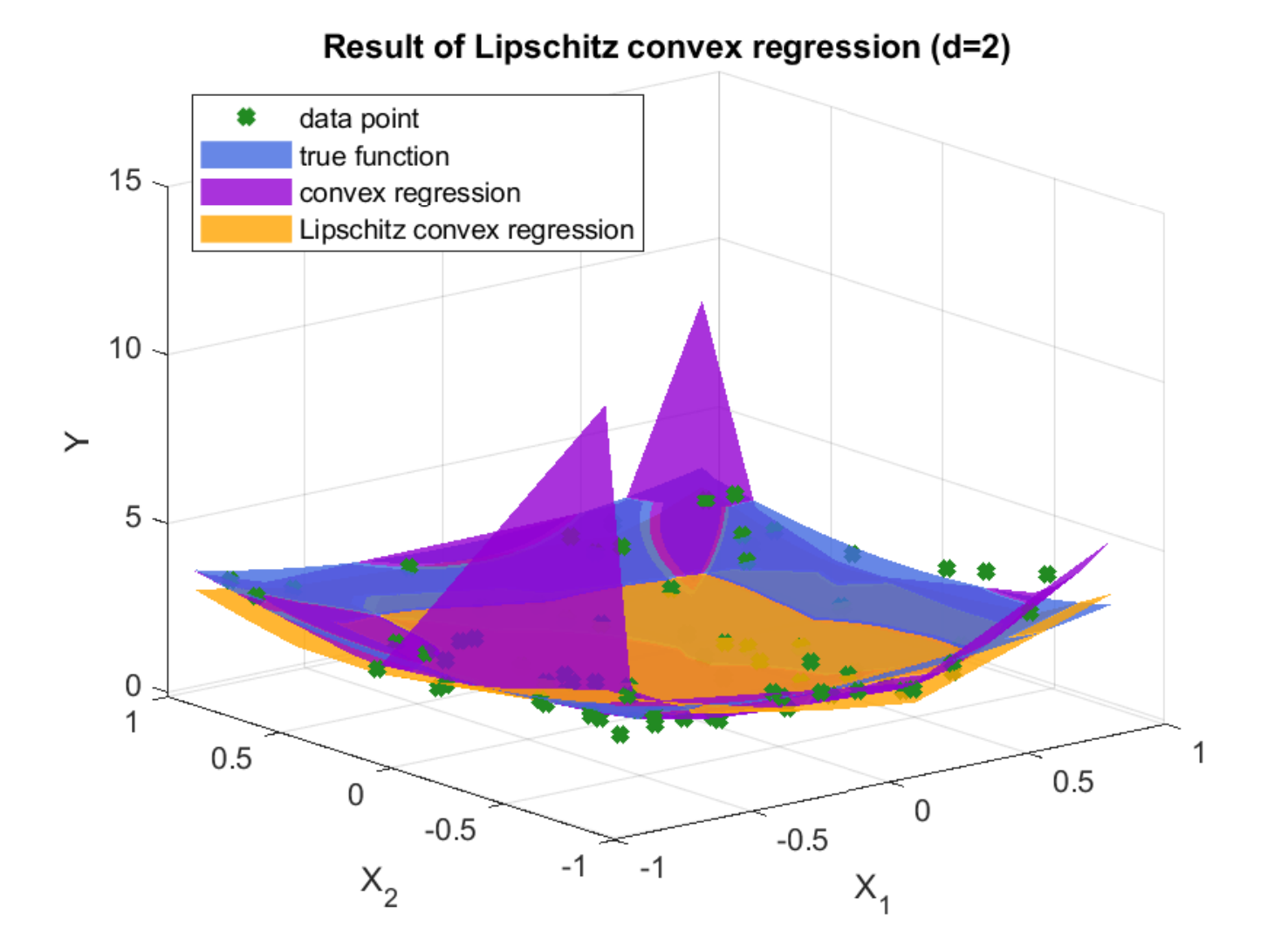}}
	\caption{Result of Lipschitz convex regression with the data-driven Lipschitz estimation method}\label{fig:data_Lip}
\end{figure}

\section{Real applications}
\label{sec:realapplication}
In this section, we apply our framework for estimating the multivariate shape-constrained convex functions in some real applications, namely, pricing of European call options, pricing of basket options, prediction of average weekly wages and estimation of production functions.

\subsection{Option pricing of European call options}
\label{subsec:calloption}
Consider a European call option whose payoff at maturity $T$ is $ (S_T-K)_+$, where $S_T$ is a random variable that stands for the stock price at $T$, and $K$ is the predetermined strike price. We are interested in the option price at time $t$, which is defined as
\begin{align*}
V(S) := \mathbb{E}[e^{-r(T-t)}(S_T-K)_+\mid S_t=S], \quad S>0,
\end{align*}
where $r$ is the risk-free interest rate. Under the Black-Scholes model, we know that the random variable $S_T$ satisfies
\begin{align*}
\log S_T \sim \mathcal{N}\Big(\log S_t+(r-\frac{1}{2}\sigma^2)(T-t),\sigma^2(T-t)\Big),
\end{align*}
where $\sigma$ is the volatility. It is well-known that $V(\cdot)$ is a convex function with $0\leq V'(S)\leq 1$ for $S>0$. Therefore, we can use the shape-constrained convex regression model with Property (S2) to estimate the function $V(\cdot)$.

There are two reasons why we consider this application to demonstrate the numerical performance of our framework. The first reason is that $V(\cdot)$ admits a closed-form solution as
\begin{align*}
V(S)=S\Phi(d_1)-Ke^{-r(T-t)}\Phi(d_2),\quad
d_{1,2} = \frac{\log \frac{S}{K}+(r\pm \frac{1}{2}\sigma^2)(T-t)}{\sigma \sqrt{T-t}},
\end{align*}
where $\Phi(\cdot)$ is the cumulative distribution function of the standard normal distribution. The second reason is that the estimation of function $V(\cdot)$ is commonly-used in pricing American-type options by approximate dynamic programming, see e.g. \cite{longstaff2001valuing}.

In our experiment, we take $t=0.1$, $T=0.4$, $K=10$, $r=0$, $\sigma=0.2$. We sample $200$ data points, denoted as $\{(S_i,V_i)\}_{i=1}^{200}$. For each $S_i$, $\log S_i$ is sampled following the distribution $\mathcal{N}(\log K+(r-\sigma^2/2)t,\sigma^2t)$, and the corresponding $V_i$ is sampled such that $\log V_i$ follows the distribution $\mathcal{N}(\log S_i+(r-\sigma^2/2)(T-t),\sigma^2(T-t))$. For comparison, we apply several regression models to estimate the conditional expectation function $V$: linear regression, least squares linear regression on a set of basis functions (e.g. weighted Laguerre basis in \cite{longstaff2001valuing}), unconstrained convex regression and convex regression with box constraint ($L=0$, $U=1$).

\begin{figure}[ht]
	\centering
	\includegraphics[width=0.8\linewidth]{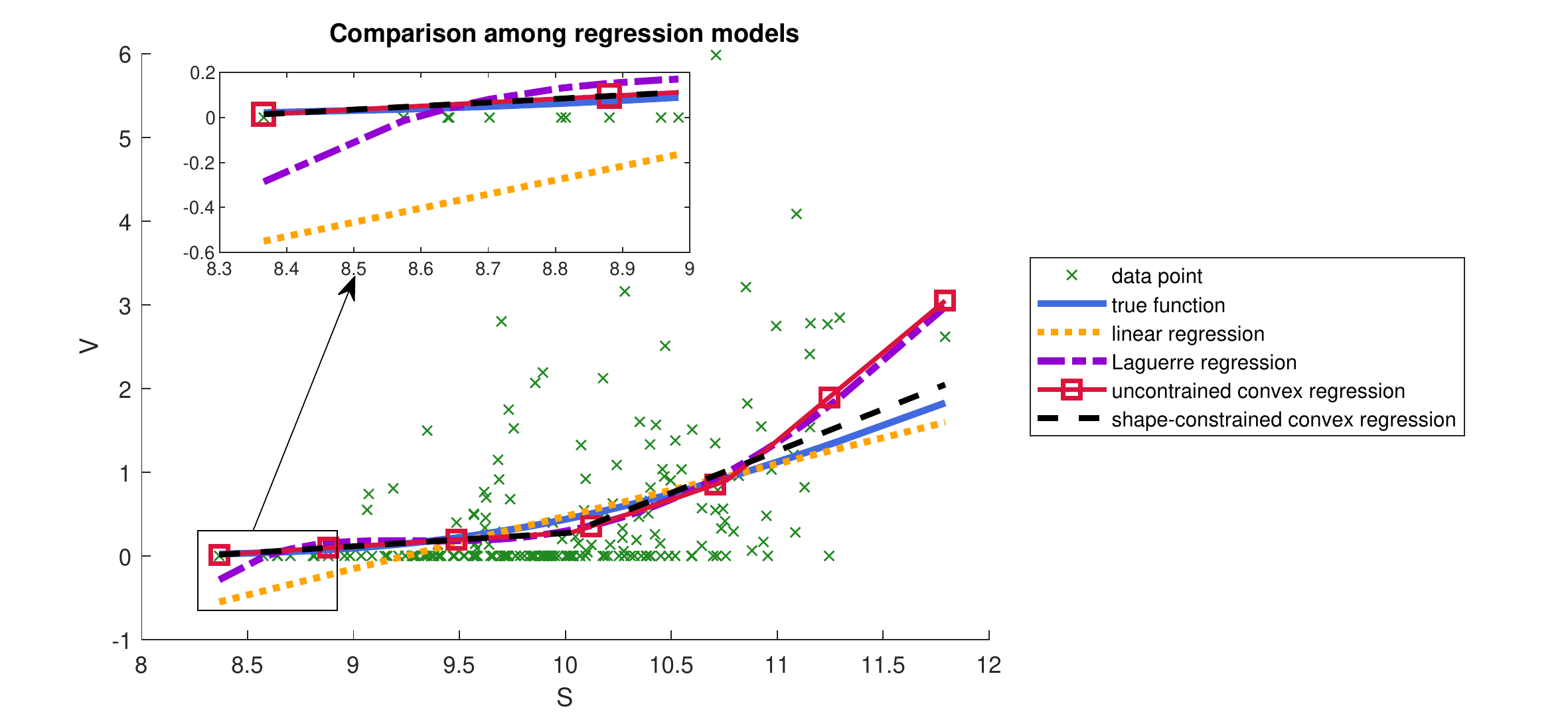}
	\caption{Results of the estimation of the option pricing of European call option}\label{toy_Lip}
\end{figure}

The comparison among four regression models is shown in Figure \ref{toy_Lip}. We can see that the performance of shape-constrained convex regression is the best. The poor performance of the other three regression models appears near the boundary in three aspects. The first is that the results from linear regression and Laguerre regression take negative values when $S$ is small, which contradicts the fact that $V$ is always non-negative. The second is that the Laguerre regression function can not obtain the required convex property. The last is that when $S$ is large, the gradients of the results obtained by Laguerre regression and unconstrained convex regression are too large. To deal with this over-fitting problem, we add the box constraint to the convex regression, which comes from prior knowledge. We can see that the result of shape-constrained convex regression performs better near the boundary, which demonstrates the advantage of the additional shape constraint.

\subsection{Option pricing of basket options}
\label{subsec:basketoption}
To test multivariate convex regression problems, we consider pricing the basket option on weighted average of $M$ underlying assets.

\paragraph{Basket option of two European call options ($M=2$).}
We first consider a basket option of two European call options, where
\begin{align*}
V(x,y)=\mathbb{E}[e^{-r(T-t)}(w_1 S_T^1+w_2S_T^2-K)_{+}\mid S_t^1=x,S_t^2=y],\quad x,y>0,
\end{align*}
and $w=(w_1,w_2)^T$ is a given weight vector such that $w\geq 0$, $w_1+w_2=1$. The random variables $S_T^1$ and $S_T^2$ satisfy
\begin{align*}
\left(\begin{aligned}
&\log S_T^1\\
&\log S_T^2
\end{aligned}
\right)\sim \mathcal{N}
\left(\left(\begin{aligned}
&\log S_t^1+(r-\sigma_1^2/2)(T-t)\\
&\log S_t^2+(r-\sigma_2^2/2)(T-t)
\end{aligned}
\right),(T-t)\begin{pmatrix}
\sigma_1^2 & \rho\sigma_1\sigma_2\\[0.4em]
\rho\sigma_1\sigma_2 &\sigma_2^2
\end{pmatrix}
\right),
\end{align*}
where $\sigma_1$, $\sigma_2$ are volatilities, $ \rho$ is the correlation coefficient. One can show that $V(\cdot,\cdot)$ is convex with $0\leq \nabla V(x,y)\leq w$, and the proof can be found in Appendix \ref{appendix: calloption}. We can apply the multivariate shape-constrained convex regression model with Property (S2) ($L=0$, $U=w$) to estimate the function $V(\cdot,\cdot)$.

Note that the convex function $V(\cdot,\cdot)$ does not admit a closed-form solution. However, it is the solution of the Black-Scholes PDE, which can be solved by the finite difference method. The details of the corresponding convection-diffusion equation and the finite difference method for solving it could be found in Appendix \ref{appendix: FDM}. We use the solution obtained by the finite difference method as the benchmark.

In the experiment, we take $r=0$, $\rho=0.1$, $\sigma_1=0.2$, $\sigma_2=0.3$, $K=10$, $t=0$, $T=0.5$, $w_1=w_2=0.5$. We sample $200$ data points, denoted as $\{(S_i,V_i)\}_{i=1}^{200}$, where $S_i$ follows the uniform distribution on the open interval $(0,5K)\times (0,5K)$ and $V_i$ follows the distribution
\begin{align*}
\mathcal{N}
\left(\log S_i+(T-t)\left(\begin{aligned}
&r-\sigma_1^2/2\\
&r-\sigma_2^2/2
\end{aligned}
\right),(T-t)\begin{pmatrix}
\sigma_1^2 & \rho\sigma_1\sigma_2\\[0.4em]
\rho\sigma_1\sigma_2 &\sigma_2^2
\end{pmatrix}
\right).
\end{align*}
The numerical result is shown in Figure \ref{fig:basket}. For better illustration, we also plot the absolute error and relative error of the results of the unconstrained convex regression and shape-constrained convex regression. As we can see, the shape-constrained convex regression performs much better than unconstrained convex regression, especially near the boundary.

\begin{figure}[ht]
	\centering
	\subfigure[]{\label{fig:basket1}
		\includegraphics[width=0.8\linewidth]{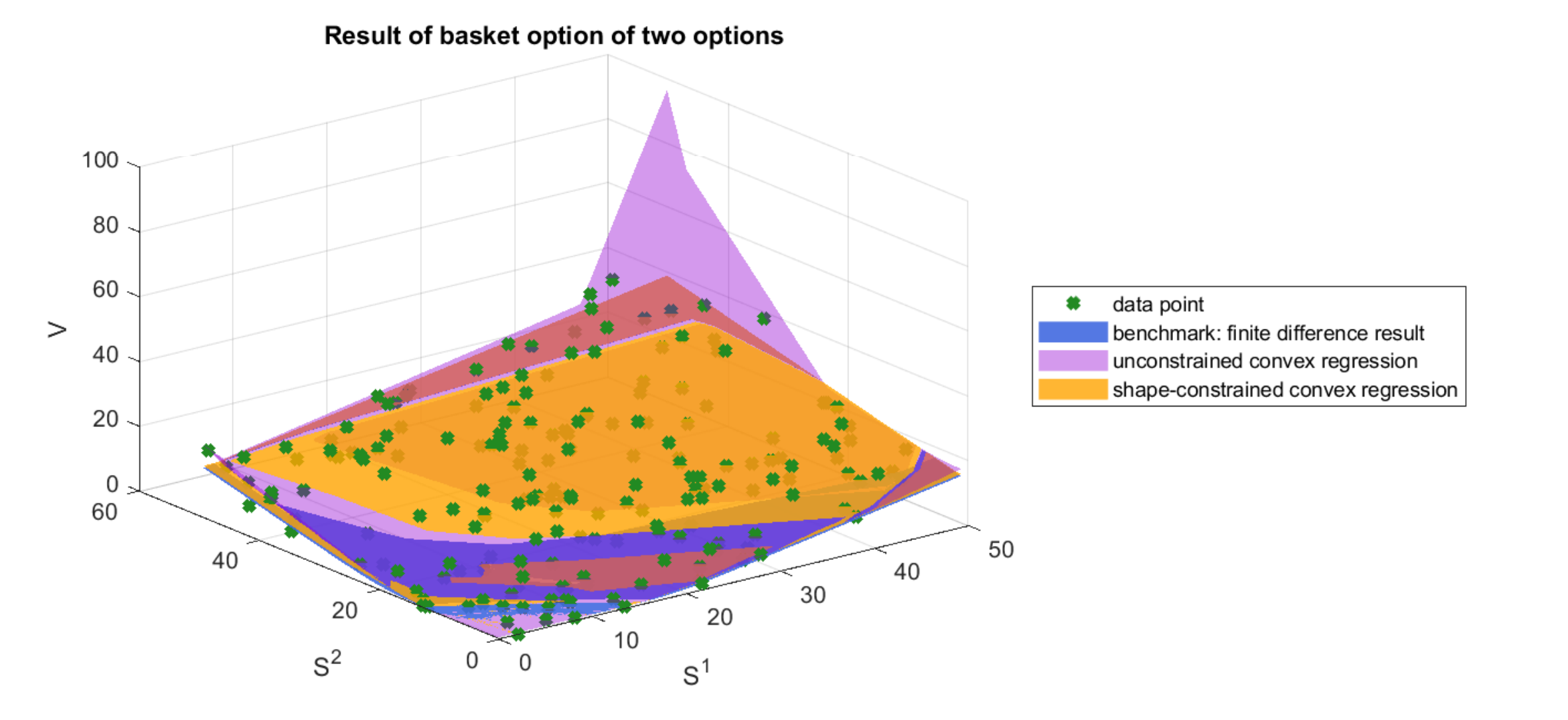}}\\
	\subfigure[]{\label{fig:basket2}
		\includegraphics[width=0.485\linewidth]{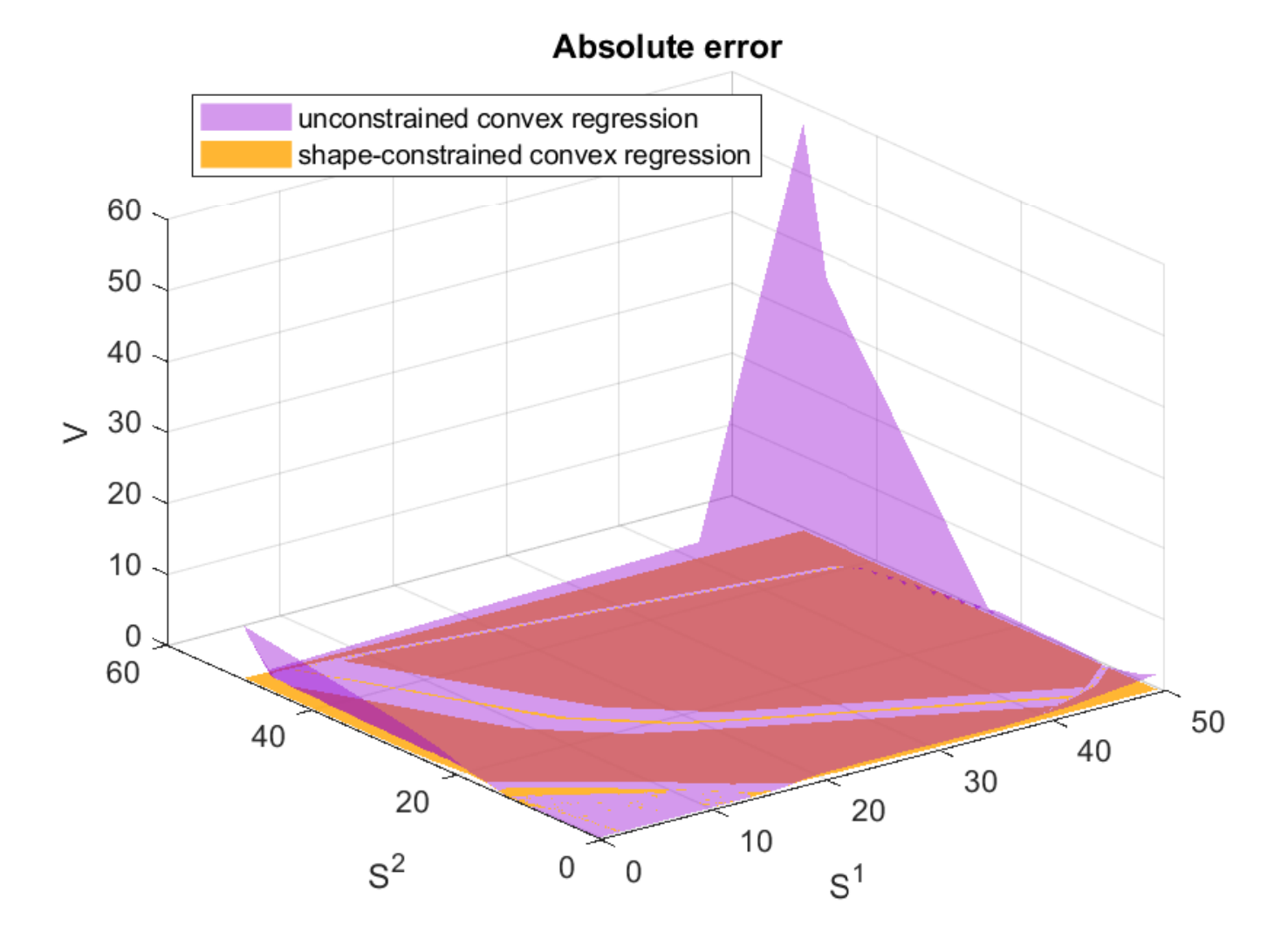}}
	\subfigure[]{\label{fig:basket3}
		\includegraphics[width=0.485\linewidth]{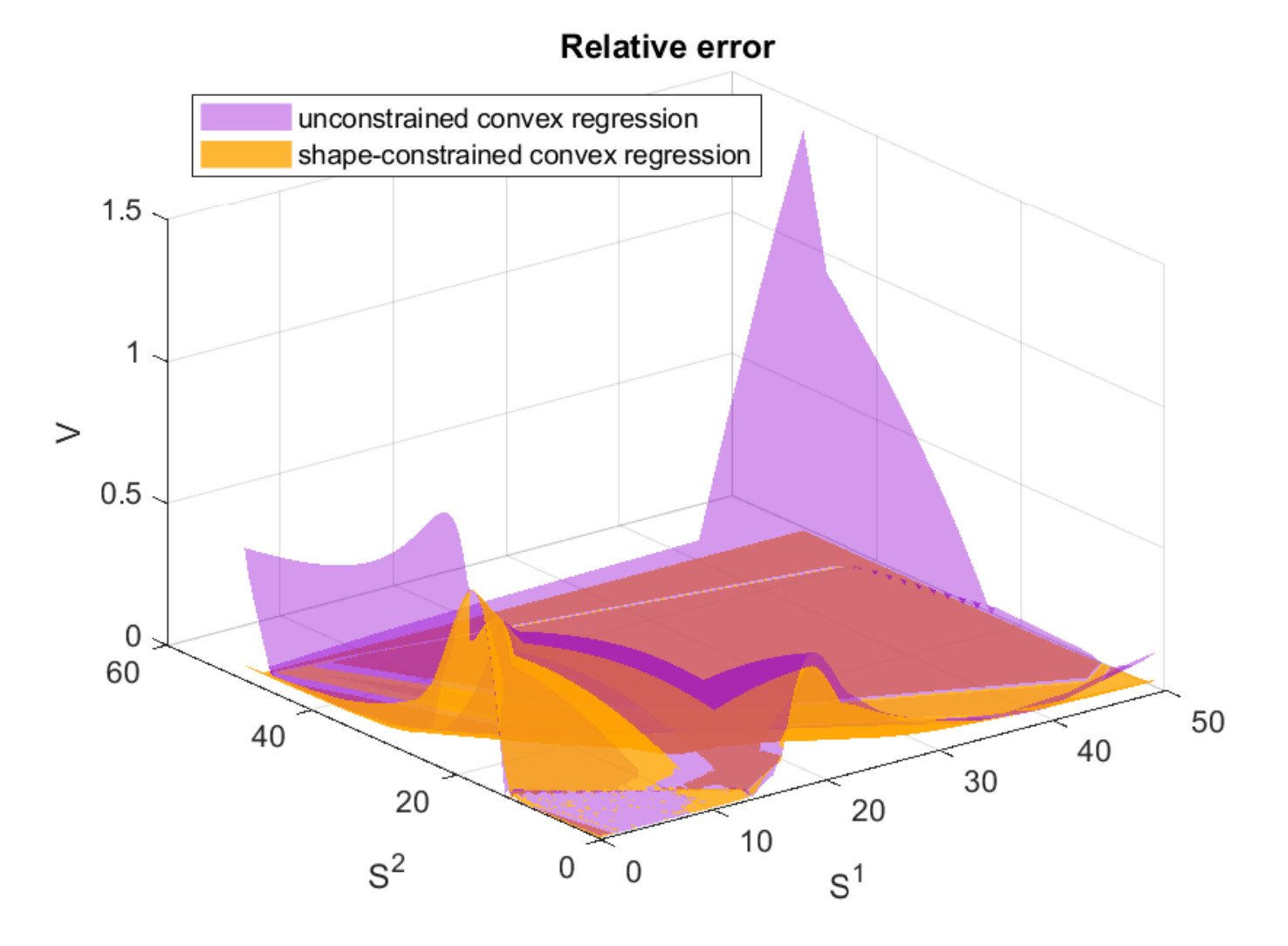}}
	\caption{Result of the estimation of the option pricing of basket option ($M=2$)}\label{fig:basket}
\end{figure}

\paragraph{Basket option of more underlying assets ($M>2$).}
The basket option in practice always contains many underlying assets, possibly greater than two. The finite difference method is very time-consuming when solving the $3$-dimensional convection-diffusion equation, and even impossible to be applied to higher dimensional cases due to the curse of dimensionality. For $M>2$, researchers tend to apply the Monte Carlo simulation to estimate the convex function associated with the basket options. Therefore, we treat the solution obtained by the Monte Carlo simulation as the benchmark.

To demonstrate the performance of the shape-constrained convex regression, we design the experiments for estimating the basket option for $M=5$ and $M=10$. Specifically, we consider a basket option of $M$ European call options, which is defined as: for any $x_1,\cdots,x_M>0$,
\begin{align*}
&V(x_1,\cdots,x_M)\\
&=\mathbb{E}[e^{-r(T-t)}(w_1 S_T^1+\cdots+w_MS_T^M-K)_{+}\mid S_t^1=x_1,\cdots,S_t^M=x_M],
\end{align*}
where $w=(w_1,\cdots,w_M)^T$ is a given weight vector such that $w\geq 0$, $w_1+\cdots+w_M=1$. The random variables $S_T^1,\cdots,S_T^M$ satisfy

\begin{align*}
\begin{pmatrix}
\! \log S_T^1\!  \\[0.4em]
\! \vdots  \! \\[0.4em]
\! \log S_T^M\!
\end{pmatrix}\!\!\sim \!\mathcal{N}
\left(\begin{pmatrix}
\! \log S_t^1\! + \! (r-\frac{\sigma_1^2}{2})(T-t)\! \\[0.4em]
\vdots \\[0.4em]
\! \log S_t^M \! + \! (r-\frac{\sigma_M^2}{2})(T-t) \!
\end{pmatrix}\!,(T-t)\! \begin{pmatrix}
\sigma_1^2 & \cdots & \! \rho\sigma_1\sigma_M \! \\[0.4em]
\vdots & \ddots & \vdots\\[0.4em]
\! \rho\sigma_1\sigma_M \! & \cdots  &\sigma_M^2
\end{pmatrix}
\right),
\end{align*}
where $\sigma_1,\cdots,\sigma_M$ are volatilities, $\rho$ is the correlation coefficient. Then $V$ is convex with $0\leq \nabla V\leq w$. We apply the multivariate shape-constrained convex regression model with Property (S2) ($L=0$, $U=w$) to do the estimation.

\begin{table}[ht]
	\parbox{0.45\linewidth}{
		\caption{Estimation of basket option with $M=5$} \label{table_M5}
		\begin{tabular}{cccc}
			\hline\noalign{\smallskip}
			Model & $n$ & MSE &  Time \\
			\noalign{\smallskip}\hline\noalign{\smallskip}
			\multirow{3}*{\tabincell{c}{\tt UC}}
			& 200 & 5.56e+1 & 00:00:07 \\
			& 400 & 1.42e+1 & 00:00:27 \\
			& 600 & 7.41e+1 & 00:00:22 \\
			\noalign{\smallskip}\hline\noalign{\smallskip}
			\multirow{3}*{\tabincell{c}{\tt SC}}
			& 200 & 4.07e-1 & 00:00:12 \\
			& 400 & 3.86e-1 & 00:00:51 \\
			& 600 & 5.95e-1 & 00:00:27 \\
			\noalign{\smallskip}\hline
		\end{tabular}
	}
	\hfill
	\parbox{0.45\linewidth}{
		\caption{Estimation of basket option with $M=10$} \label{table_M10}
		\begin{tabular}{cccc}
			\hline\noalign{\smallskip}
			Model & $n$ & MSE &  Time \\
			\noalign{\smallskip}\hline\noalign{\smallskip}
			\multirow{3}*{\tabincell{c}{\tt UC}}
			& 200 & 2.05e+1 & 00:00:12 \\
			& 400 & 4.06e+1 & 00:00:10 \\
			& 600 & 5.98e+1 & 00:00:20 \\
			\noalign{\smallskip}\hline\noalign{\smallskip}
			\multirow{3}*{\tabincell{c}{\tt SC}}
			& 200 & 2.21e+0 & 00:00:35 \\
			& 400 & 1.32e+0 & 00:00:27 \\
			& 600 & 1.00e+0 & 00:00:42 \\
			\noalign{\smallskip}\hline
		\end{tabular}
	}
\end{table}

In the experiment, we set $r=0$, $\rho=0.1$, $K=10$, $t=0$, $T=0.5$, $w_i=1/M$, $\sigma_i=0.2+0.025(i-1)$, $i=1,\cdots,M$. We sample $n$ data points as the case for $M=2$. To illustrate the performance of our procedure, we uniformly generate $1000$ test points in the range $(0,5K)^M$. At each test point, we use the Monte Carlo simulation with $10^5$ samples to compute the ``true" function value. We summarize the results of $M=5$ and $M=10$ in Table \ref{table_M5} and Table \ref{table_M10}, respectively. In the tables, ``UC" represents the unconstrained convex regression, ``SC" represents the shape-constrained convex regression, and ``MSE" represents the mean squared error. As one can see, the shape-constrained convex regression takes a little bit longer time to be solved than the unconstrained convex regression, but get a much better estimated result.

\subsection{Prediction of average weekly wages}
\label{subsec:weeklywages}

We consider the problem of estimating the average weekly wages based on years of education and experience as given in \cite[Chapter 10]{ramsey2012statistical}. This dataset is from 1988 March U.S. Current Population Survey, which can be downloaded as ex1029 in the R package Sleuth2. The set contains weekly wages in 1987 for a sample of 25632 males between the age of 18 and 70 who worked full-time, with their years of education and years of experience. After averaging over a grid with cell size of $1$ year by $1$ year and ignoring the outliers, we finally come to a dataset with $857$ samples.

A reasonable assumption for this application is that the wages are concave in years of experience and a transformation of years of education, i.e., $1.2^{\mbox{years of education}}$, according to \cite{hannah2013multivariate}. The estimated result is shown in Figure \ref{fig:ex1029}. The shape-constrained convex regression problem is solved within $1$ minute.

\begin{figure}[ht]
	\subfigure[Estimated function values at each $X_i$]{\label{fig:ex1}
		\includegraphics[width=0.485\linewidth]{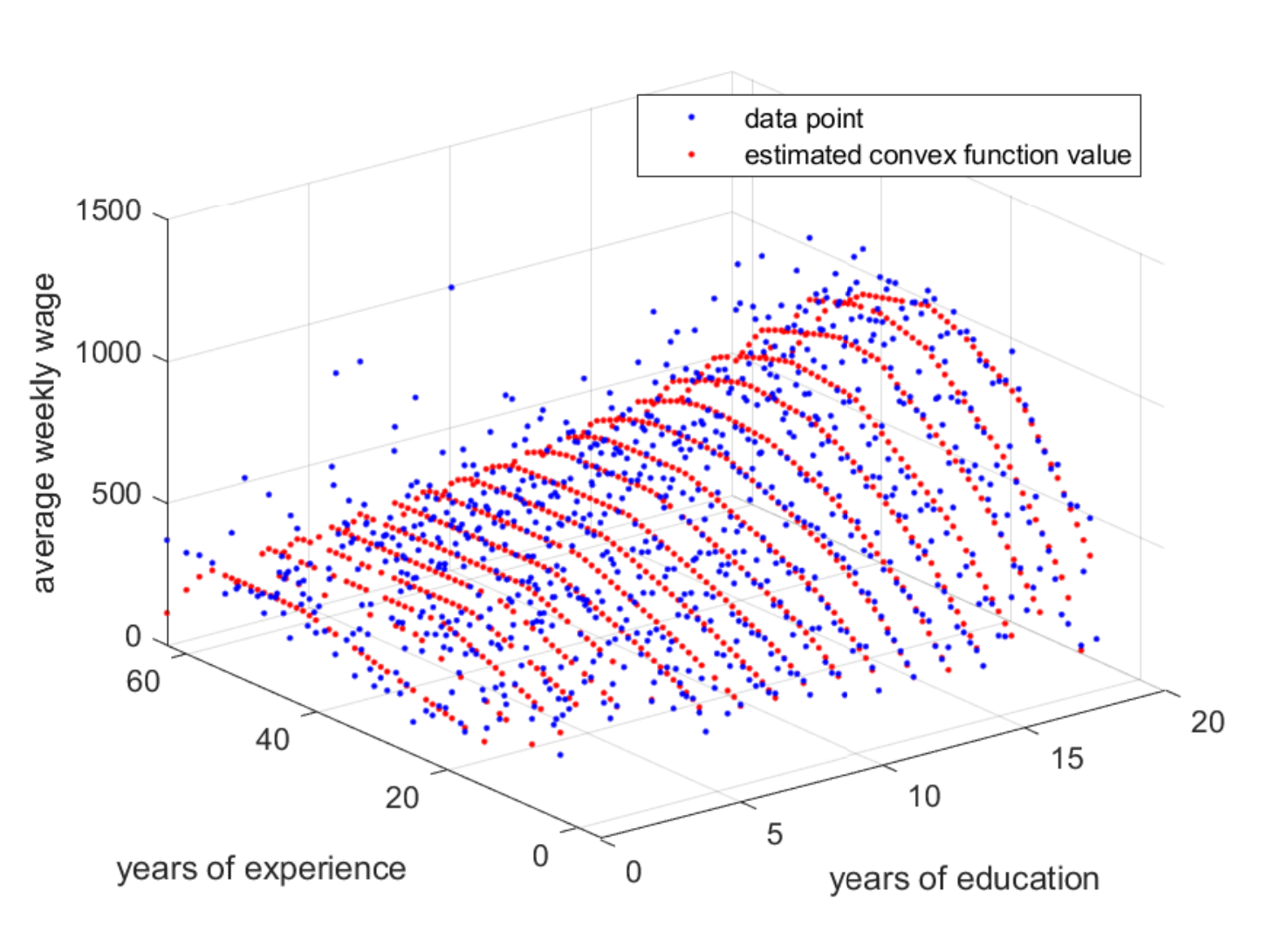}}
	\subfigure[Visualization of the function]{\label{fig:ex2}
		\includegraphics[width=0.485\linewidth]{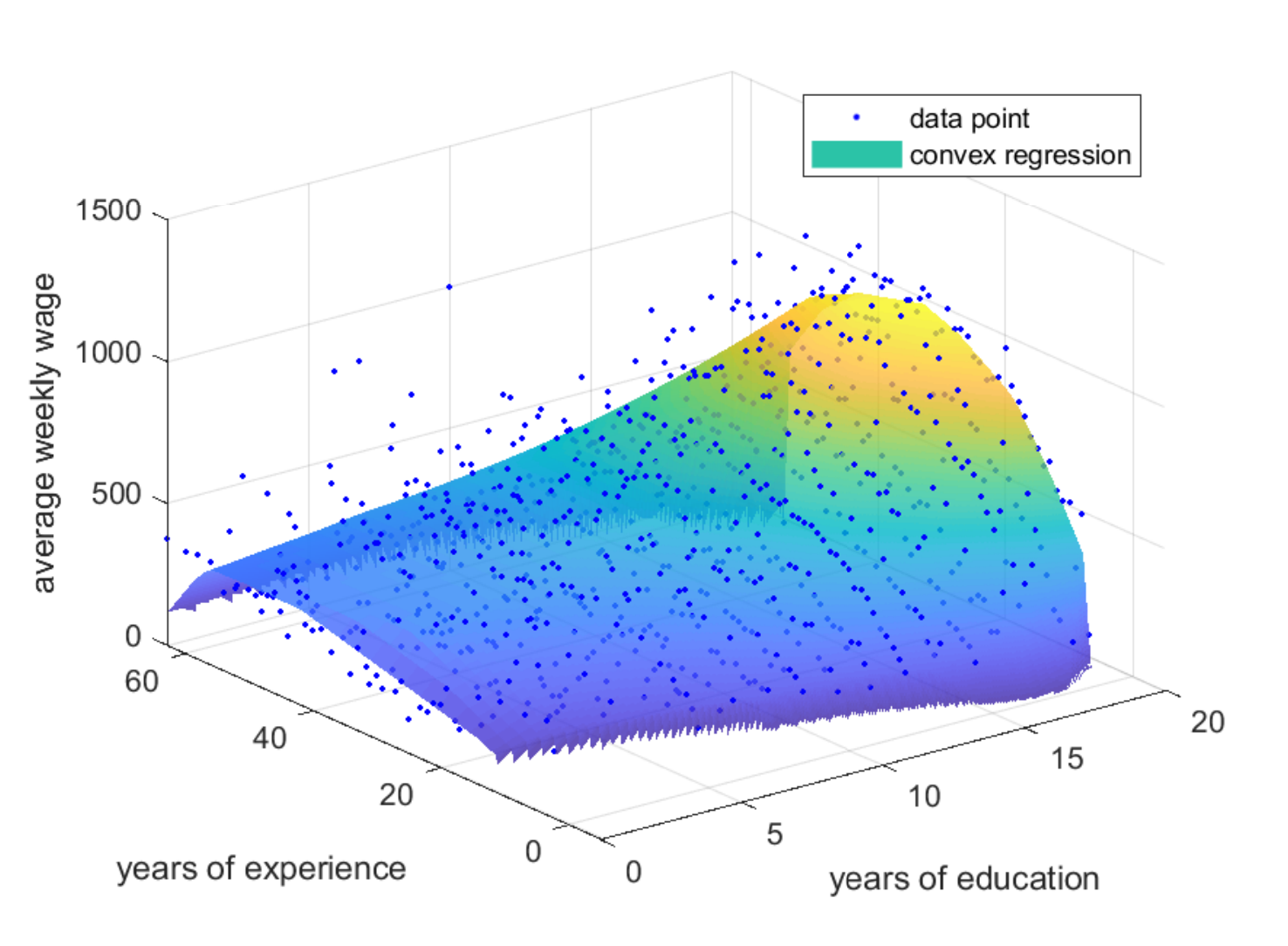}}
	\caption{Results of the estimation of average weekly wages}\label{fig:ex1029}
\end{figure}

\subsection{Estimation of production functions}
\label{subsec:productionfun}
In economics, a production function gives the technological relation between quantities of inputs and quantities of output of goods. Production functions are known to be concave and non-decreasing \cite{hanoch1972testing,varian1984nonparametric,yagi2018shape}. We apply our framework to estimate the production function for the plastic industry (CIIU3 industry code: 2520) in the year 2011. The dataset can be downloaded from the website of Chile's National Institute of Statistics. As in the setting in \cite{yagi2018shape}, we use labor and capital as the input variables, and value added as the output variable. In the dataset, labor is measured as the total man-hours per year, capital and value added are measured in millions of Chilean peso. After removing some outliers, the dataset contains 250 samples. The numerical results can be found in Figure \ref{fig:prod1}. The shape-constrained convex regression problem is solved within $3$ seconds.

\begin{figure}[h]
	\subfigure[Correlogram of the function]{\label{fig:prod4_3}
		\includegraphics[width=0.45\linewidth]{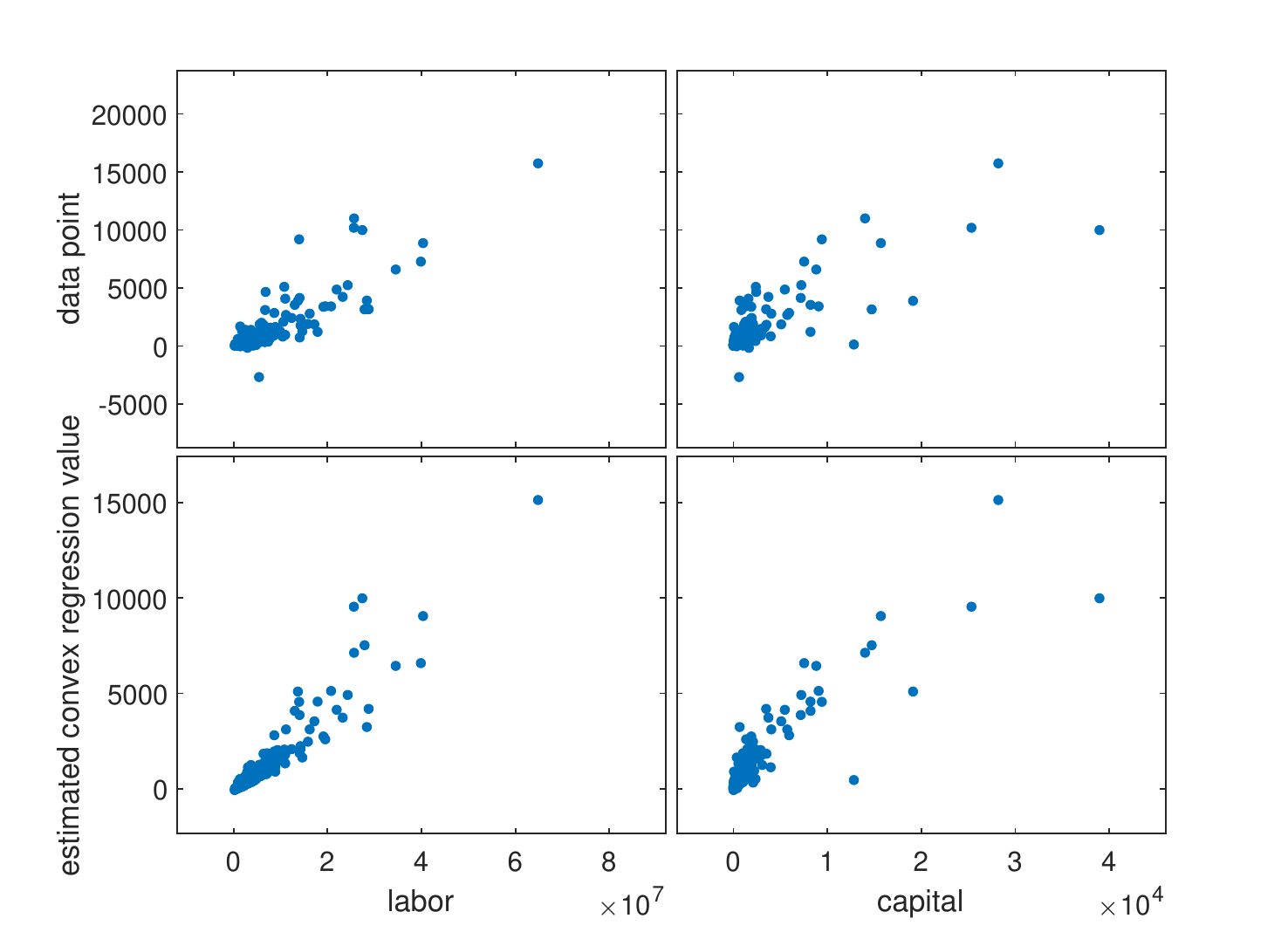}}
	\subfigure[Visualization of the function]{\label{fig:prod4_4}
		\includegraphics[width=0.45\linewidth]{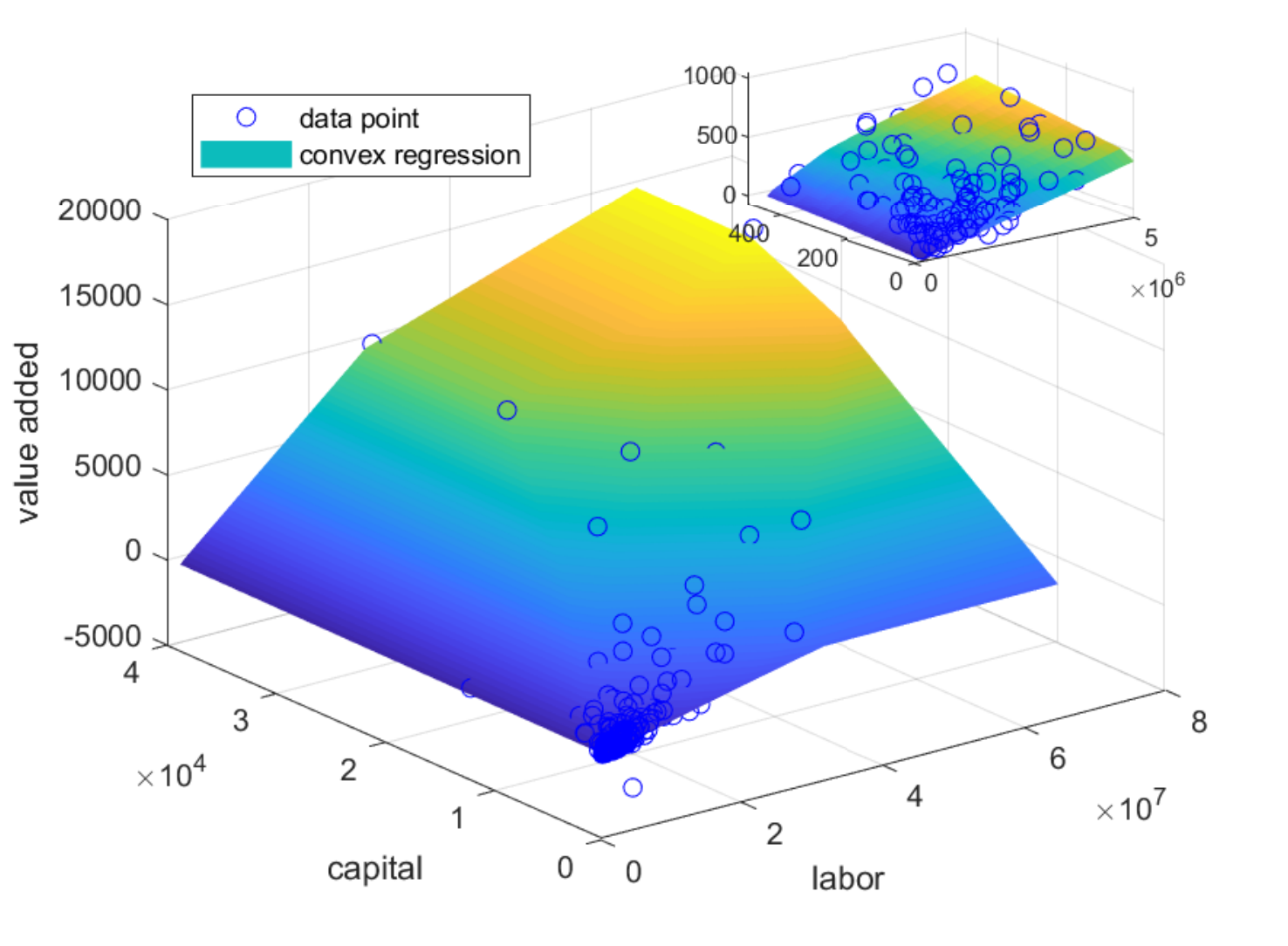}}	
	\caption{Result of estimation of production function of plastic in Chile}\label{fig:prod1}
\end{figure}

\begin{figure}[h]
	\subfigure[Correlogram of the function]{\label{fig:prod9_4}
		\includegraphics[width=0.45\linewidth]{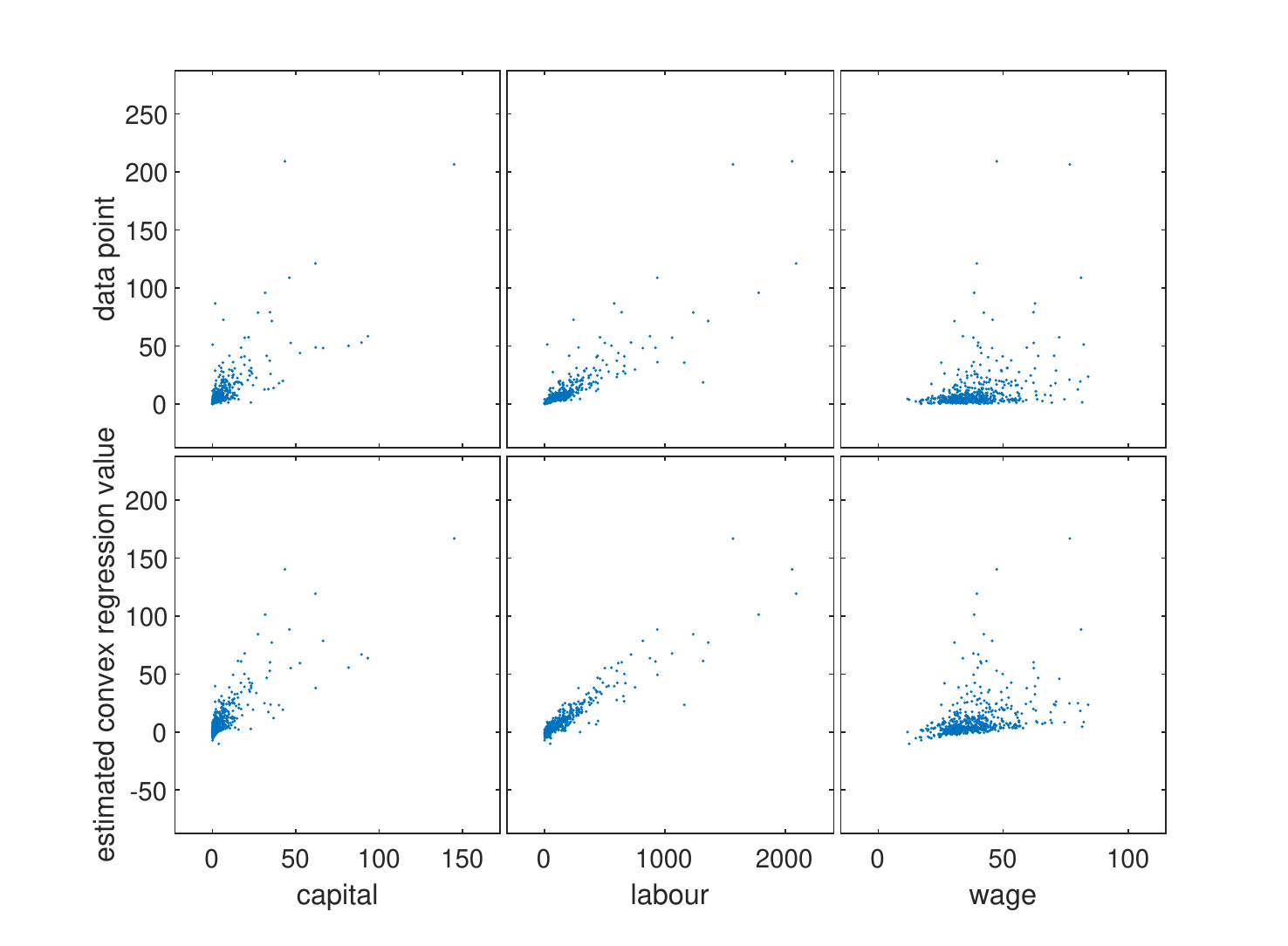}}
	\subfigure[Visualization of the function]{\label{fig:prod9_5}
		\includegraphics[width=0.45\linewidth]{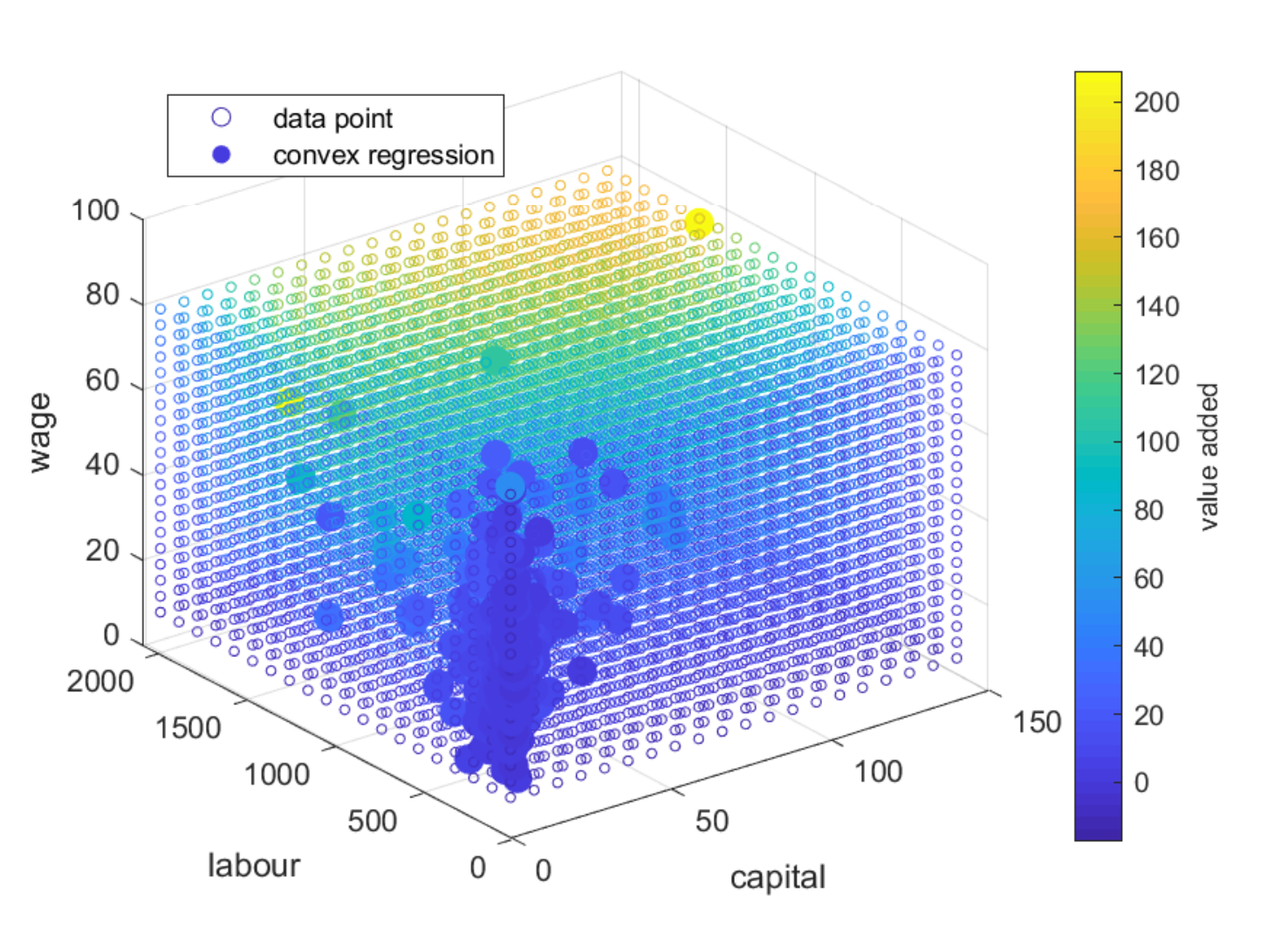}}	
	\caption{Result of estimation of production function of Belgian firms}\label{fig:prod2}
\end{figure}

Another example is to explain the labour demand of 569 Belgian firms for the year 1996. The dataset can be obtained from \cite{verbeek2008guide}\footnote{{\tt https://www.wiley.com/legacy/wileychi/verbeek2ed/datasets.html}}. The dataset includes the total number of employees (labour), their average wage (wage), the amount of capital (capital) and a measure of output (value added). The labour is measured as the number of workers, the wage is measured in units of 1000 euro, and the capital and value added are measured in units of a million euro. After removing the outliers, the dataset contains 562 samples. The result can be found in Figure \ref{fig:prod2} and the problem is solved in $22$ seconds.

\section{Conclusion and future work}
\label{sec:futurework}
In this paper, we provide a unified framework for computing a least squares estimator for the multivariate shape-constrained convex regression function. In addition, we propose an efficient algorithm, which is a semismooth Newton based proximal augmented Lagrangian method, to solve the large-scale constrained QP in the framework. Moreover, in order to further accelerate the computation under the large-sample setting, we design a practical implementation of the constraint generation method, where the reduce problem in each round is solved by the proposed {\tt proxALM}. We conduct extensive numerical experiments to demonstrate the efficiency and robustness of our proposed {\tt proxALM}, as well as the superior performance of the acceleration with the constraint generation method.

\section*{Acknowledgements}
The authors would like to thank Professor Necdet S. Aybat for helpful clarifications on his work in \cite{aybat2014parallel}.

\appendix
\section*{Appendices}
\section{Derivation of the proximal mapping and generalized Jacobian associated with $\mathcal{D}=\{x\in \mathbb{R}^d\mid \|x\|_{1}\leq L\}$}
\label{appendix: derive_1norm}
For $x\in \mathbb{R}^d$, let $P_x={\rm Diag}({\rm sign}(x))\in \mathbb{R}^{d\times d}$, then
\begin{align*}
\Pi_\mathcal{D}(x) &= \underset{y\in\mathbb{R}^d}{\arg\min}\
\Big\{\frac{1}{2}\|y-x\|^2\mid \|y\|_1\leq L\Big\}\\
&= LP_x \Big( \underset{y\in\mathbb{R}^d}{\arg\min}\  \Big\{\frac{1}{2}\|y-P_x x/L\|^2\mid e_d^T y \leq 1,y\geq 0\Big\}\Big)\\
&=\left\{\begin{aligned}
&x && \mbox{if}\ \|x\|_1\leq L, \\
&L P_x \Pi_{\Delta_d}(P_x x/L)&& \mbox{otherwise,}
\end{aligned}\right.
\end{align*}
where $\Delta_d=\{x\in \mathbb{R}^d\mid e_d^T x=1,x\geq 0\}$. To derive the generalized Jacobian of $\Pi_\mathcal{D}(\cdot)$, we need the generalized Jacobian of $\Pi_{\Delta_d}(\cdot)$. Following the idea in \cite{han1997newton,li2020efficient}, we can explicitly compute an element of the generalized Jacobian of $\Pi_{\Delta_d}(\cdot)$ at $P_x x/L$. Let $K$ be the set of index $i$ such that $(\Pi_{\Delta}(P_x x/L))_i=0$. Then
\begin{align*}
\widetilde{H}=I_d-\begin{bmatrix}
I_K^T & e_d
\end{bmatrix}\Big(
\begin{bmatrix}
I_K \\[0.4em]
e_d^T
\end{bmatrix}
\begin{bmatrix}
I_K^T & e_d
\end{bmatrix}
\Big)^{\dagger}\begin{bmatrix}
I_K \\[0.4em]
e_d^T
\end{bmatrix}
\end{align*}
is an element in $\partial \Pi_{\Delta_d}(P_x x/L)$, where $I_K$ means the matrix consisting of the rows of the identity matrix $I_d$, indexed by $K$. After some algebraic computation, we can see
\begin{align*}
\widetilde{H}&=I_d-\begin{bmatrix}
I_K^T & e_d
\end{bmatrix}
\begin{bmatrix}
I_{|K|}+\frac{1}{n-|K|}e_{|K|} e_{|K|}^T & -\frac{1}{n-|K|}e_{|K|} \\[0.4em]
-\frac{1}{n-|K|}e_{|K|}^T & \frac{1}{n-|K|}
\end{bmatrix}
\begin{bmatrix}
I_K \\[0.4em]
e_d^T
\end{bmatrix}={\rm Diag}(r)-\frac{1}{{\rm nnz}(r)}rr^T,
\end{align*}
where $r\in \mathbb{R}^d$ is defined as $r_i=1$ if $(\Pi_{\Delta}(P_x x/L))_i\neq 0 $ and $r_i=0$ otherwise. Therefore,
\begin{align*}
H\in \partial \Pi_\mathcal{D}(x),\quad \mbox{where } H=\left\{\begin{aligned}
&I_d && \mbox{if}\ \|x\|_1\leq L, \\
&P_x \widetilde{H}  P_x && \mbox{otherwise}.
\end{aligned}\right.
\end{align*}

\section{A symmetric Gauss-Seidel based alternating direction method of multipliers ({\tt sGS-ADMM}) for \eqref{reformulated_problem}}
\label{appendix:sGSADMM}

In the literature, popular first-order methods based on the framework of the alternating direction method of multipliers have been applied to solve \eqref{reformulated_problem}.  In \cite[Section A.2]{mazumder2019computational}, the problem \eqref{reformulated_problem} is reformulated as
\begin{align*}
\min_{\theta\in \mathbb{R}^n,\xi\in \mathbb{R}^{dn},\eta\in \mathbb{R}^{n^2}}\ \Big\{\frac{1}{2}\|\theta-Y\|^2+p(\xi)+\delta_{+}(\eta)\mid A\theta+B\xi-\eta=0\Big\}.
\end{align*}
The corresponding augmented Lagrangian function for a fixed $\sigma>0$ is defined by
\begin{align*}
\widetilde{\mathcal{L}}_{\sigma}(\theta,\xi,\eta;u)=\frac{1}{2}\|\theta-Y\|^2+p(\xi)+\delta_{+}(\eta)+\frac{\sigma}{2}\|A\theta+B\xi-\eta-\frac{u}{\sigma}\|^2-\frac{1}{2\sigma}\|u\|^2.
\end{align*}
Then the two-block {\tt ADMM} is given as
\begin{align*}
\left\{\begin{aligned}
&\xi^{k+1} = \arg\min\  \widetilde{\mathcal{L}}_{\sigma}(\theta^{k},\xi,\eta^{k};u^{k})=\arg\min\ \Big\{
p(\xi)+\frac{\sigma}{2}\|A\theta^k+B\xi-\eta^k-\frac{u^k}{\sigma}\|^2\Big\},\\
&(\theta^{k+1},\eta^{k+1}) = \arg\min\ \widetilde{\mathcal{L}}_{\sigma}(\theta,\xi^{k+1},\eta;u^{k}),\\
& u^{k+1}=u^k-\tau\sigma(A\theta^{k+1}+B\xi^{k+1}-\eta^{k+1}),
\end{aligned}\right.
\end{align*}
where $\tau\in(0,(1+\sqrt{5}/2))$ is a given step length. As described in \cite{mazumder2019computational}, the subproblem of updating $\xi$ is separable in the variables $\xi_i$'s for $i=1,\cdots,n$, and the update of each $\xi_i$ can be solved by using an interior point method. The update of $\theta$ and $\eta$ is performed by using a block coordinate descent method, which may converge slowly. One can also apply the directly extended three-block {\tt ADMM} algorithm as in \cite[Section 2.1]{mazumder2019computational} to solve \eqref{reformulated_problem}, and the steps are given by
\begin{align*}
\left\{\begin{aligned}
&\xi^{k+1} = \arg\min\ \widetilde{\mathcal{L}}_{\sigma}(\theta^{k},\xi,\eta^{k};u^{k}),\\
&\theta^{k+1} = \arg\min\ \widetilde{\mathcal{L}}_{\sigma}(\theta,\xi^{k+1},\eta^k;u^{k}),\\
&\eta^{k+1} = \arg\min\ \widetilde{\mathcal{L}}_{\sigma}(\theta^{k+1},\xi^{k+1},\eta;u^{k}),\\
& u^{k+1}=u^k-\tau\sigma(A\theta^{k+1}+B\xi^{k+1}-\eta^{k+1}).
\end{aligned}\right.
\end{align*}
In the directly extended three-block {\tt ADMM}, the subproblem of updating $\theta$ can be computed by solving a linear system, and that of updating $\eta$ can be solved by the projection onto $\mathbb{R}_{+}^{n^2}$. However, it is shown in \cite{chen2016direct} that the directly extended three-block {\tt ADMM} may not be convergent. Thus it is desirable to employ an algorithm that is guaranteed to converge.

In this section, we aim to present an efficient and convergent multi-block {\tt ADMM} for solving \eqref{reformulated_problem}. The authors in \cite{chen2017efficient} have proposed an inexact symmetric Gauss-Seidel based multi-block {\tt ADMM} for solving high-dimensional convex composite conic optimization problems, and it was demonstrated to perform better than the possibly nonconvergent directly extended multi-block {\tt ADMM}. To adapt the {\tt sGS-ADMM} in \cite{chen2017efficient} to solve \eqref{reformulated_problem}, we first rewrite \eqref{reformulated_problem} as follows:
\begin{align}
\min_{\theta\in \mathbb{R}^n,\xi,y\in \mathbb{R}^{dn},\eta\in \mathbb{R}^{n^2}} \ \Big\{\frac{1}{2}\|\theta-Y\|^2+p(y)+\delta_{+}(\eta)\Bigm\lvert A\theta+B\xi-\eta=0,\ \xi-y=0\Big\}.\label{P}
\end{align}
Given a parameter $\sigma>0$, the augmented Lagrangian function associated with \eqref{P} is defined by
\begin{align}
\widehat{\mathcal{L}}_{\sigma}(\theta,\xi,y,\eta;u,v)&=\frac{1}{2}\|\theta-Y\|^2+p(y)+\delta_{+}(\eta)-\langle u,A\theta+B\xi-\eta\rangle-\langle v,\xi-y\rangle \notag\\
&\qquad +\frac{\sigma}{2}\|A\theta+B\xi-\eta\|^2+\frac{\sigma}{2}\|\xi-y\|^2\notag\\
&=\frac{1}{2}\|\theta-Y\|^2+p(y)+\delta_{+}(\eta)+\frac{\sigma}{2}\|A\theta+B\xi-\eta-\frac{u}{\sigma}\|^2+\frac{\sigma}{2}\|\xi-y-\frac{v}{\sigma}\|^2\notag\\
&\qquad -\frac{1}{2\sigma}\|u\|^2-\frac{1}{2\sigma}\|v\|^2.
\label{augmented_lagrangian}
\end{align}
Then the {\tt sGS-ADMM} algorithm for solving \eqref{reformulated_problem} is given as in Algorithm \ref{alg:sGSadmm}.

\begin{algorithm}
	\caption{: Symmetric Gauss-Seidel based {\tt ADMM} for \eqref{reformulated_problem}}
	\label{alg:sGSadmm}
	\begin{algorithmic}[1]
		\STATE {\bfseries Initialization:} Choose an initial point $(\theta^0,\xi^0,y^0,\eta^0,u^0,v^0)\in \mathbb{R}^n\times \mathbb{R}^{dn}\times \mathbb{R}^{dn}\times\mathbb{R}^{n^2}\times\mathbb{R}^{n^2}\times \mathbb{R}^{dn}$, and a positive parameter $\sigma > 0$. For $k = 0, 1, 2, \dots$
		\REPEAT
		\STATE {\bfseries Step 1}. Compute
		\begin{align*}
		(y^{k+1},\eta^{k+1}) = \arg\min\ \widehat{\mathcal{L}}_{\sigma}(\theta^{k},\xi^k,y,\eta;u^{k},v^{k}).
		\end{align*}
		\\
		\STATE {\bfseries Step 2}. Compute
		\begin{align*}
		&\mbox{{\bfseries Step 2a}.  } \widehat{\theta}^{k+1} = \arg\min \ \widehat{\mathcal{L}}_{\sigma}(\theta,\xi^{k},y^{k+1},\eta^{k+1};u^{k},v^{k}),\\
		&\mbox{{\bfseries Step 2b}.  } \xi^{k+1} = \arg\min\ \widehat{\mathcal{L}}_{\sigma}(\widehat{\theta}^{k+1},\xi,y^{k+1},\eta^{k+1};u^{k},v^{k}),\\
		&\mbox{{\bfseries Step 2c}.  } \theta^{k+1} = \arg\min \ \widehat{\mathcal{L}}_{\sigma}(\theta,\xi^{k+1},y^{k+1},\eta^{k+1};u^{k},v^{k}).
		\end{align*}
		\\
		\STATE {\bfseries Step 3}. Compute
		\begin{align*}
		u^{k+1}=u^k-\tau\sigma(A\theta^{k+1}+B\xi^{k+1}-\eta^{k+1}),\quad v^{k+1}=v^k-\tau\sigma(\xi^{k+1}-y^{k+1}),
		\end{align*}
		where $\tau\in(0,(1+\sqrt{5})/2)$ is the step length that is typically chosen to be $1.618$.
		\UNTIL{Stopping criterion is satisfied.}
	\end{algorithmic}
\end{algorithm}

In Algorithm \ref{alg:sGSadmm}, all the subproblems can be solved explicitly. In {\bf Step 1}, $\eta^{k+1}$ and $y^{k+1}$ are separable and can be solved independently as
\begin{align*}
y^{k+1} = {\rm Prox}_{p/\sigma}(\xi^k-v^k/\sigma),\quad \eta^{k+1} = \Pi_{+}(A\theta^k+B\xi^k-u^k/\sigma),
\end{align*}
where $\Pi_{\pm}(\cdot)$ denotes the projection onto $\mathbb{R}^{n^2}_{\pm}$. In {\bf Step 2a} and {\bf Step 2c}, $\theta$ can be computed by solving the following linear system
\begin{align*}
(I_n+\sigma A^T A)\theta = Y-\sigma A^T (B\xi-\eta-u/\sigma).
\end{align*}
By noting that $A^TA=2nI_n-2e_ne_n^T$, one can apply the Sherman-Morrison-Woodbury formula to compute
\begin{align*}
(I_n+\sigma A^T A)^{-1} = \frac{1}{1+2\sigma n}(I_n+2\sigma e_ne_n^T).
\end{align*}
Thus $\theta$ can be computed in $O(n)$ operations. For {\bf Step 2b}, $\xi^{k+1}$ can be computed by solving the linear equation
\begin{align*}
(I_{dn}+B^T B)\xi = y^{k+1}+v^k/\sigma-B^T(A\widehat{\theta}^{k+1}-\eta^{k+1}-u^k/\sigma).
\end{align*}
As the coefficient matrix $I_{dn}+B^TB$ is a block diagonal matrix consisting of $n$ blocks of $d\times d$ submatrices, each $\xi_i$ can be computed separately, and the inverse of each block only needs to be computed once.

The convergence result of Algorithm \ref{alg:sGSadmm} is presented in the following theorem, which is taken directly from \cite[Theorem 5.1]{chen2017efficient}.
\begin{theorem}\label{thm:sgsadmm}
	Suppose that the solution set to the KKT system \eqref{KKT_system} is nonempty. Let $\{(\theta^k,\xi^k,y^k,\eta^k,u^k,v^k)\}$ be the sequence generated by Algorithm \ref{alg:sGSadmm}. Then $\{(\theta^k,\xi^k,y^k,\eta^k)\}$ converges to an optimal solution of problem \eqref{P}, and $\{(u^k,v^k)\}$ converges to an optimal solution of its dual \eqref{D}.
\end{theorem}

\section{More results on comparison of algorithms for solving \eqref{reformulated_problem}}
\label{appendix: numerical}
Table \ref{table_org} -- Table \ref{table_1norm} show the comparison among {\tt proxALM}, {\tt sGS-ADMM} and {\tt MOSEK} on instances with relatively large $d$ and $n$. Note that here we set the stopping criterion to $R_{\rm KKT}\leq 10^{-6}$ to show that our proposed {\tt proxALM} is capable of solving the problem \eqref{reformulated_problem} to relatively high accuracy. As one can see that, when estimating the function $\psi(x)=\exp(p^T x)$ for moderate $(d,n)=(100,1000)$, {\tt proxALM} is about $3$ times faster than {\tt sGS-ADMM}, and about $29$ times faster than {\tt MOSEK}. For the case when $d=100$, $n=4000$, which is a large problem with $404,000$ variables and about $16,000,000$ inequality constraints, {\tt MOSEK} runs out of memory, while {\tt proxALM} could solve it within $7$ minutes and {\tt sGS-ADMM} takes $17$ minutes. From the tables, we can see that {\tt sGS-ADMM} performs much better than {\tt MOSEK} in each instance, and {\tt proxALM} performs even better than {\tt sGS-ADMM}. In most of the cases, {\tt proxALM} is at least $10$ times faster than {\tt MOSEK}.

\begin{table}
	\caption{Convex regression for test function $\psi(x)=\exp(p^T x)$, where $p$ is a given random vector with each coordinate drawn from the standard normal distribution}
	\label{table_org}
	\tabcolsep 2.5pt
	\begin{threeparttable}
		\begin{tabular}{ccccccccc}
			\hline\noalign{\smallskip}
			\multicolumn{2}{c}{
				\tikz{
					\node[below left, inner sep=-1pt] (def) {Algorithm};
					\node[above right,inner sep=0.5pt] (abc) {$(d,n)$};
					\draw (def.north west|-abc.north west) -- (def.south east-|abc.south east);
			}}
			& $(50,500)$ & $(50,1000)$ &  $(50,2000)$  & $(100,1000)$ & $(100,2000)$ & $(100,4000)$\\
			\noalign{\smallskip}\hline\noalign{\smallskip}
			\multirow{3}*{\tabincell{c}{\tt proxALM}}
			& Iteration & 12(11)\tnote{*} & 16(20) & 21(38) & 15(20) & 20(38) & 26(51)\\
			& Time & 00:00:02 & 00:00:06 & 00:00:57 & 00:00:07 & 00:01:14 & 00:06:44\\
			& $R_{\rm KKT}$ & 4.18e-8 & 8.97e-8 & 9.14e-7 & 6.14e-7 & 3.41e-7 & 9.48e-7 \\
			\noalign{\smallskip}\hline\noalign{\smallskip}
			\multirow{3}*{\tabincell{c}{\tt sGS-ADMM} } 	
			& Iteration & 389 & 562 & 1206 & 355 & 701 & 1263\\
			& Time & 00:00:05 & 00:00:25 & 00:03:57 & 00:00:19 & 00:02:39 & 00:16:59 \\
			& $R_{\rm KKT}$ & 9.95e-7  & 9.88e-7 & 9.92e-7 & 9.99e-7 & 9.91e-7 & 9.98e-7 \\
			\noalign{\smallskip}\hline\noalign{\smallskip}
			\multirow{3}*{\tabincell{c}{\tt MOSEK} } 	
			& Iteration & 10 & 11 & 13 & 11 & 10 & O.M.\\
			& Time & 00:00:20 & 00:01:50 & 00:10:50 & 00:03:22 & 00:19:46 & O.M. \\
			& $R_{\rm KKT}$ & 6.59e-9  & 3.92e-9 & 1.53e-7 & 7.98e-10 & 7.65e-8 & O.M. \\
			\noalign{\smallskip}\hline
		\end{tabular}
		\begin{tablenotes}\footnotesize
			\item[*] ``12(11)" means ``{\tt proxALM} iterations (total inner {\tt SSN} iterations)". O.M. means the algorithm runs out of memory. Time is in the format of hours:minutes:seconds.
		\end{tablenotes}
	\end{threeparttable}
\end{table}

\begin{table}
	\caption{Convex regression with monotone constraint (non-decreasing) for the test function $\psi(x)=(e_d^T x)_{+}$}
	\label{table_nondecreasing}
	\tabcolsep 2.5pt
		\begin{tabular}{ccccccccc}
			\hline\noalign{\smallskip}
			\multicolumn{2}{c}{
				\tikz{
					\node[below left, inner sep=-1pt] (def) {Algorithm};
					\node[above right,inner sep=0.5pt] (abc) {$(d,n)$};
					\draw (def.north west|-abc.north west) -- (def.south east-|abc.south east);
			}}
			& $(50,500)$ & $(50,1000)$ &  $(50,2000)$  & $(100,1000)$ & $(100,2000)$ & $(100,4000)$\\
			\noalign{\smallskip}\hline\noalign{\smallskip}
			\multirow{3}*{\tabincell{c}{\tt proxALM}}
			& Iteration & 15(18) & 17(23) & 23(77) & 17(29) & 21(59) & 32(96)\\
			& Time & 00:00:02 & 00:00:07 & 00:02:48 & 00:00:12 & 00:02:09 & 00:12:34\\
			& $R_{\rm KKT}$ & 1.87e-7  & 1.50e-7 & 1.38e-7 & 8.16e-7 & 8.23e-7 & 8.97e-7 \\
			\noalign{\smallskip}\hline\noalign{\smallskip}
			\multirow{3}*{\tabincell{c}{\tt sGS-ADMM} } 	
			& Iteration & 529 & 917 & 1685 & 541 & 905 & 1582\\
			& Time & 00:00:08 & 00:00:49 & 00:06:18 & 00:00:34 & 00:03:50 & 00:25:18 \\
			& $R_{\rm KKT}$ & 9.79e-7 & 9.99e-7 & 9.98e-7 & 9.85e-7 & 9.88e-7 & 9.98e-7 \\
			\noalign{\smallskip}\hline\noalign{\smallskip}
			\multirow{3}*{\tabincell{c}{\tt MOSEK} } 	
			& Iteration & 14 & 13 & 14 & 13 & 16 & O.M.\\
			& Time & 00:00:24 & 00:02:00 & 00:11:32 & 00:03:47 & 00:25:23 & O.M. \\
			& $R_{\rm KKT}$ & 1.54e-9  & 1.45e-9 & 2.63e-8 & 2.37e-7 & 1.31e-9 & O.M. \\
			\noalign{\smallskip}\hline
		\end{tabular}
\end{table}

\begin{table}
	\caption{Convex regression with box constraint ($L=0_d$, $U=e_d$) for the test function $\psi(x)=\ln(1+\exp(e_d^Tx))$}
	\label{table_twobound}
	\tabcolsep 2.5pt
	\begin{tabular}{ccccccccc}
		\hline\noalign{\smallskip}
		\multicolumn{2}{c}{
			\tikz{
				\node[below left, inner sep=-1pt] (def) {Algorithm};
				\node[above right,inner sep=0.5pt] (abc) {$(d,n)$};
				\draw (def.north west|-abc.north west) -- (def.south east-|abc.south east);
		}}
		& $(50,500)$ & $(50,1000)$ &  $(50,2000)$  & $(100,1000)$ & $(100,2000)$ & $(100,4000)$\\
		\noalign{\smallskip}\hline\noalign{\smallskip}
		\multirow{3}*{\tabincell{c}{\tt proxALM}}
		& Iteration & 23(40) & 24(67) & 30(135) & 17(28) & 21(60) & 33(102)\\
		& Time & 00:00:03 & 00:00:18 & 00:04:35 & 00:00:12 & 00:02:32 & 00:12:56\\
		& $R_{\rm KKT}$ & 9.55e-7  & 8.79e-7 & 7.02e-7 & 6.65e-8 & 3.54e-7 & 9.39e-7 \\
		\noalign{\smallskip}\hline\noalign{\smallskip}
		\multirow{3}*{\tabincell{c}{\tt sGS-ADMM} } 	
		& Iteration & 663 & 1016 & 2689 & 513 & 871 & 1541\\
		& Time & 00:00:11 & 00:00:54 & 00:10:05 & 00:00:33 & 00:03:50 & 00:23:32 \\
		& $R_{\rm KKT}$ & 9.60e-7  & 9.73e-7 & 9.98e-7 & 9.92e-7 & 9.95e-7 & 1.00e-6 \\
		\noalign{\smallskip}\hline\noalign{\smallskip}
		\multirow{3}*{\tabincell{c}{\tt MOSEK} } 	
		& Iteration & 19 & 24 & 31 & 18 & 15 & O.M.\\
		& Time & 00:00:31 & 00:02:52 & 00:19:03 & 00:04:50 & 00:25:10 & O.M. \\
		& $R_{\rm KKT}$ & 2.40e-7  & 6.03e-8 & 1.11e-8 & 3.18e-9 & 2.23e-9 & O.M. \\
		\noalign{\smallskip}\hline
	\end{tabular}
\end{table}

\begin{table}
	\caption{Convex regression with Lipschitz constraint ($p=1$, $q=\infty$, $L=1$) for the test function $\psi(x)=\sqrt{1+x^Tx}$}
	\label{table_infnorm}
	\tabcolsep 2.5pt
	\begin{tabular}{ccccccccc}
		\hline\noalign{\smallskip}
		\multicolumn{2}{c}{
			\tikz{
				\node[below left, inner sep=-1pt] (def) {Algorithm};
				\node[above right,inner sep=0.5pt] (abc) {$(d,n)$};
				\draw (def.north west|-abc.north west) -- (def.south east-|abc.south east);
		}}
		& $(50,500)$ & $(50,1000)$ &  $(50,2000)$  & $(100,1000)$ & $(100,2000)$ & $(100,4000)$\\
		\noalign{\smallskip}\hline\noalign{\smallskip}
		\multirow{3}*{\tabincell{c}{\tt proxALM}}
		& Iteration & 13(14) & 17(30) & 24(51) & 16(26) & 21(44) & 33(72)\\
		& Time & 00:00:02 & 00:00:08 & 00:01:12 & 00:00:11 & 00:01:26 & 00:07:59\\
		& $R_{\rm KKT}$ & 5.05e-7  & 5.08e-7 & 9.45e-7 & 4.31e-7 & 2.41e-7 & 9.77e-7 \\
		\noalign{\smallskip}\hline\noalign{\smallskip}
		\multirow{3}*{\tabincell{c}{\tt sGS-ADMM} } 	
		& Iteration & 531 & 928 & 1730 & 509 & 973 & 1691\\
		& Time & 00:00:09 & 00:00:50 & 00:06:47 & 00:00:33 & 00:04:21 & 00:27:33 \\
		& $R_{\rm KKT}$ & 9.77e-7 & 9.97e-7  & 9.84e-7 & 9.89e-7 & 9.90e-7 & 9.98e-7 \\
		\noalign{\smallskip}\hline\noalign{\smallskip}
		\multirow{3}*{\tabincell{c}{\tt MOSEK} } 	
		& Iteration & 10 & 11 & 12 & 10 & 11 & O.M.\\
		& Time & 00:00:23 & 00:01:55 & 00:10:32 & 00:03:38 & 00:21:27 & O.M. \\
		& $R_{\rm KKT}$ & 7.51e-9  & 3.46e-10 & 1.16e-9 & 5.87e-13 & 3.00e-10 & O.M. \\
		\noalign{\smallskip}\hline
	\end{tabular}
\end{table}

\begin{table}
	\caption{Convex regression with Lipschitz constraint ($p=2$, $q=2$, $L=\lambda_{\rm max}(Q)$) for the test function $\psi(x)=\sqrt{x^T Qx}$}
	\label{table_2norm}
	\tabcolsep 2.5pt
	\begin{threeparttable}
		\begin{tabular}{ccccccccc}
		\hline\noalign{\smallskip}
		\multicolumn{2}{c}{
			\tikz{
				\node[below left, inner sep=-1pt] (def) {Algorithm};
				\node[above right,inner sep=0.5pt] (abc) {$(d,n)$};
				\draw (def.north west|-abc.north west) -- (def.south east-|abc.south east);
		}}
		& $(50,500)$ & $(50,1000)$ &  $(50,2000)$  & $(100,1000)$ & $(100,2000)$ & $(100,4000)$\\
		\noalign{\smallskip}\hline\noalign{\smallskip}
		\multirow{3}*{\tabincell{c}{\tt proxALM}}
		& Iteration & 12(11) & 17(30) & 21(41) & 15(20) & 21(41) & 23(48)\\
		& Time & 00:00:02 & 00:00:08 & 00:00:55 & 00:00:08 & 00:01:12 & 00:06:22\\
		& $R_{\rm KKT}$ & 1.27e-10  & 4.07e-7 & 1.94e-7 & 3.28e-7 & 7.10e-7 & 9.69e-7 \\
		\noalign{\smallskip}\hline\noalign{\smallskip}
		\multirow{3}*{\tabincell{c}{\tt sGS-ADMM} } 	
		& Iteration & 541 & 953 & 1481 & 494 & 934 & 1591\\
		& Time & 00:00:11 & 00:00:53 & 00:05:39 & 00:00:35 & 00:04:22 & 00:23:59 \\
		& $R_{\rm KKT}$ & 9.76e-7  & 9.99e-7 & 9.99e-7 & 9.94e-7 & 9.91e-7 & 9.91e-7 \\
		\noalign{\smallskip}\hline\noalign{\smallskip}
		\multirow{3}*{\tabincell{c}{\tt MOSEK} } 	
		& Iteration & 10 & 13 & 13 & 11 & 12 & O.M.\\
		& Time & 00:00:23 & 00:02:03 & 00:10:57 & 00:03:44 & 00:22:47 & O.M. \\
		& $R_{\rm KKT}$ & 2.50e-7  & 1.06e-9 & 1.19e-8 & 7.53e-9 & 2.12e-12 & O.M. \\
		\noalign{\smallskip}\hline
		\end{tabular}
	\begin{tablenotes}\footnotesize
		\item[*] $Q\in \mathbb{R}^{d\times d}$ is a randomly generated symmetric and positive definite matrix with known largest eigenvalue.
	\end{tablenotes}
	\end{threeparttable}
\end{table}

\begin{table}
	\caption{Convex regression with Lipschitz constraint ($p=\infty$, $q=1$, $L=1$) for the test function $\psi(x)=\ln(1+e^{x_1}+\cdots+e^{x_d})$}
	\label{table_1norm}
	\tabcolsep 2.5pt
	\begin{tabular}{ccccccccc}
		\hline\noalign{\smallskip}
		\multicolumn{2}{c}{
			\tikz{
				\node[below left, inner sep=-1pt] (def) {Algorithm};
				\node[above right,inner sep=0.5pt] (abc) {$(d,n)$};
				\draw (def.north west|-abc.north west) -- (def.south east-|abc.south east);
		}}
		& $(50,500)$ & $(50,1000)$ &  $(50,2000)$  & $(100,1000)$ & $(100,2000)$ & $(100,4000)$\\
		\noalign{\smallskip}\hline\noalign{\smallskip}
		\multirow{3}*{\tabincell{c}{\tt proxALM}}
		& Iteration & 12(12) & 16(22) & 22(44) & 15(21) & 19(35) & 27(62)\\
		& Time & 00:00:02 & 00:00:06 & 00:00:49 & 00:00:08 & 00:01:09 & 00:08:37\\
		& $R_{\rm KKT}$ & 3.04e-7  & 4.43e-7 & 8.49e-7 & 2.07e-7 & 6.63e-7 & 8.41e-7 \\
		\noalign{\smallskip}\hline\noalign{\smallskip}
		\multirow{3}*{\tabincell{c}{\tt sGS-ADMM} } 	
		& Iteration & 413 & 767 & 1401 & 436 & 775 & 1379\\
		& Time & 00:00:08 & 00:00:45 & 00:05:25 & 00:00:31 & 00:03:24 & 00:21:59 \\
		& $R_{\rm KKT}$ & 9.88e-7  & 9.96e-7 & 9.80e-7 & 9.79e-7 & 9.99e-7 & 1.00e-6 \\
		\noalign{\smallskip}\hline\noalign{\smallskip}
		\multirow{3}*{\tabincell{c}{\tt MOSEK} } 	
		& Iteration & 12 & 12 & 14 & 13 & 8 & O.M.\\
		& Time & 00:00:41 & 00:03:28 & 00:21:06 & 00:07:26 & 00:39:55 & O.M. \\
		& $R_{\rm KKT}$ & 1.23e-8  & 1.26e-7 & 5.33e-9 & 3.09e-10 & 2.92e-9 & O.M. \\
		\noalign{\smallskip}\hline
	\end{tabular}
\end{table}

\section{Property of basket option of two European call options}
\label{appendix: calloption}
The function $V(x,y)$ is differentiable since it is the solution of the Black-Scholes PDE. By the definition of $V$, we can see that $V$ is non-decreasing in $x$ and $y$, which means that $\nabla V(x,y)\geq 0$. According to the distribution of $S_T^1$ and $S_T^2$, we have that
\begin{align*}
V(x,y) = e^{-r(T-t)} \mathbb{E}_z f(x,y,z),
\end{align*}
where
\begin{align*}
f(x,y,z)&=(xw_1 e^{(r-\sigma_1^2/2)(T-t)+\sqrt{T-t}z_1}+yw_2 e^{(r-\sigma_2^2/2)(T-t)+\sqrt{T-t}z_2}-K)_+,\\
\begin{pmatrix}
z_1\\[0.4em]
z_2
\end{pmatrix}
&\sim \mathcal{N}(0,\begin{pmatrix}
\sigma_1^2 & \rho\sigma_1\sigma_2\\[0.4em]
\rho\sigma_1\sigma_2 &\sigma_2^2
\end{pmatrix}).
\end{align*}
For any $x_1,x_2,y\in\mathbb{R}$, we can see that
\begin{align*}
|V(x_1,y)-V(x_2,y)|&=e^{-r(T-t)} \Big| \mathbb{E}_z [f(x_1,y,z)-f(x_1,y,z)]\Big|\\
&\leq e^{-r(T-t)} \mathbb{E}_z |f(x_1,y,z)-f(x_1,y,z)|\\
&\leq e^{-r(T-t)} \mathbb{E}_z  [w_1e^{(r-\sigma_1^2/2)(T-t)+\sqrt{T-t}z_1} |x_1-x_2|]\\
&=w_1|x_1-x_2|e^{-\sigma_1^2/2(T-t)}\mathbb{E}_z [e^{\sqrt{T-t}z_1}]\\
&=w_1|x_1-x_2|.
\end{align*}
Similarly, we can prove that for any $x,y_1,y_2\in\mathbb{R}$,
\begin{align*}
|V(x,y_1)-V(x,y_2)|\leq w_2|y_1-y_2|.
\end{align*}
Therefore, we have that fact that $0\leq \nabla V(x,y)\leq w$ for any $x,y$.

\section{A finite difference method for estimating the basket option of two European call options}
\label{appendix: FDM}

It is well-known that the function $V(x,y)=U(0,x,y)$, where $U$ satisfies the Black-Scholes PDE
\begin{align*}
\left\{\begin{aligned}
&\frac{\partial U}{\partial t}+rx\frac{\partial U}{\partial x}+ry\frac{\partial U}{\partial y}+\frac{1}{2}\sigma_1^2x^2\frac{\partial^2 U}{\partial^2 x^2}+\rho \sigma_1\sigma_2xy\frac{\partial^2 U}{\partial xy}+\frac{1}{2}\sigma_2^2y^2\frac{\partial^2 U}{\partial^2 y^2}-rU=0,\\
&U(T,x,y)=(w_1x+w_2y-K)^+.
\end{aligned}
\right.
\end{align*}
Let $\tau = T-t$, $u(\tau,x,y)=U(t,x,y)$, then $u$ satisfies
\begin{align*}
\left\{\begin{aligned}
&\frac{\partial u}{\partial \tau}-rx\frac{\partial u}{\partial x}-ry\frac{\partial u}{\partial y}-\frac{1}{2}\sigma_1^2x^2\frac{\partial^2 u}{\partial^2 x^2}-\rho \sigma_1\sigma_2xy\frac{\partial^2 u}{\partial xy}-\frac{1}{2}\sigma_2^2y^2\frac{\partial^2 u}{\partial^2 y^2}+ru=0,\\
&u(0,x,y)=(w_1x+w_2y-K)^+.
\end{aligned}
\right.
\end{align*}
The above convection-diffusion equation can be solved numerically on a bounded region $(0,x_{\max})\times (0,y_{\max})$ by the standard finite difference method with the artificial boundary conditions
\begin{align*}
\left\{\begin{aligned}
&u(\tau,x,0)=c(w_1x,K,r,\tau,\sigma_1),\\
&u(\tau,0,y)=c(w_2y,K,r,\tau,\sigma_2),\\
&\frac{\partial }{\partial x}u(\tau,x_{\max},y)=w_1,\\
&\frac{\partial }{\partial y}u(\tau,x,y_{\max})=w_2,\\
\end{aligned}
\right.
\end{align*}
where
\begin{align*}
c(x,K,r,\tau,\sigma)=x\Phi(d_1)-Ke^{-r\tau}\Phi(d_2),\quad
d_{1,2} = \frac{\log \frac{x}{K}+(r\pm \frac{1}{2}\sigma^2)\tau}{\sigma \sqrt{\tau}},
\end{align*}
and $\Phi(\cdot)$ is the cumulative distribution function of the standard normal distribution.

\end{document}